\documentclass[12pt,a4paper]{article}
\usepackage[a4paper,margin=3cm]{geometry}
\usepackage{amsmath,amsfonts,amsthm,amssymb,dsfont,lmodern,mathrsfs}
\usepackage{comment}
\usepackage{xfrac}
\usepackage{natbib}
\usepackage[utf8]{inputenc}
\usepackage[T1]{fontenc}
\usepackage[shortlabels]{enumitem}

\newlist{mathlist}{enumerate}{5}
\setlist[mathlist]{label={\textup{(\roman*)}}}

\usepackage[dvipsnames]{xcolor}
\usepackage[colorlinks,linkcolor=blue,citecolor=blue,urlcolor=blue]{hyperref}
\usepackage{doi}
\usepackage{tikz}
\usepackage[margin=1cm]{caption}
\usepackage{cleveref}

\usepackage{parskip}
\begingroup
\makeatletter
\@for\theoremstyle:=definition,remark,plain\do{%
  \expandafter\g@addto@macro\csname th@\theoremstyle\endcsname{%
    \addtolength\thm@preskip\parskip }%
}
\endgroup

\usepackage{titlesec}
\titleformat{\paragraph}[hang]{\bfseries}{}{0mm}{}
\titlespacing{\paragraph}{0mm}{\baselineskip}{0.5ex}

\theoremstyle{plain}
\newtheorem{theorem}{Theorem}[section]
\newtheorem{lemma}[theorem]{Lemma}
\newtheorem{corollary}[theorem]{Corollary}
\newtheorem{proposition}[theorem]{Proposition}

\newtheorem{definition}[theorem]{Definition}
\newtheorem{assumption}[theorem]{Assumption}
\Crefname{assumption}{Assumption}{Assumptions}
\theoremstyle{definition}

\newtheorem{remark}[theorem]{Remark}

\newcommand{\rmd}{\mathrm{d}}
\newcommand{\1}{\mathds{1}}
\renewcommand{\P}{\mathrm{P}}

\newcommand{\norm}[2][]{#1\lVert #2 #1\rVert}
\newcommand{\abs}[2][]{#1\lvert #2 #1\rvert}

\renewcommand{\epsilon}{\varepsilon}
\renewcommand{\rho}{\varrho}
\renewcommand{\phi}{\varphi}
\renewcommand{\emptyset}{\varnothing}
\renewcommand{\leq}{\leqslant}
\renewcommand{\le}{\leqslant}
\renewcommand{\geq}{\geqslant}
\renewcommand{\ge}{\geqslant}

\begin{document}

\title{A functional central limit theorem for kernel gradient flow and infinitesimal gradient boosting}
\author{Clément Dombry\footnote{Universit{\'e} Marie et Louis Pasteur, CNRS, LmB (UMR 6623), F-25000 Besan\c{c}on, France. Email: \url{clement.dombry@umlp.fr},\ \url{jean-jil.duchamps@umlp.fr}} \ and Jean-Jil Duchamps\footnotemark[1]}
\maketitle

\begin{abstract}
Building on the large-sample analysis of infinitesimal gradient boosting \citep{DD24}, we study the fluctuations of the process around its deterministic limit and establish a functional central limit
theorem: the rescaled deviations converge in distribution to a Gaussian process. The analysis is carried out in a
reproducing kernel Hilbert space (RKHS) naturally associated with the softmax
gradient tree base learner, in which the boosting process is characterized as the solution of an autonomous ordinary differential equation (ODE). 

The proof rests on a general stochastic perturbation analysis of ODEs
in Banach spaces, which is of independent interest: whenever a sequence of
vector fields converges and satisfies a central limit theorem, so does the
associated ODE solution.  We first illustrate this perturbation approach  in the simpler
setting of kernel gradient flow, where the Gaussian limit admits an explicit
characterization, and then consider the more complicated tree-based gradient boosting setting.
\end{abstract}



\tableofcontents

\section{Introduction}\label{sec:introduction}

\subsection{Motivations and related works}
\label{subsec:motivations}

Tree-based gradient boosting, introduced by \citet{F01}, is one of the most
successful algorithms in modern machine learning. Starting from an initial predictor, it
builds a sequence of predictors by repeatedly fitting a regression tree to the
pseudo-residuals of the current model and adding it with a small shrinkage factor,
called the learning rate~$\lambda > 0$. This procedure can be interpreted as a form of
functional gradient descent in an infinite-dimensional space of predictors: each
regression tree approximates the negative gradient of the empirical loss, and the
resulting sequence follows a path along which the training error decreases. Modern
implementations such as XGBoost~\citep{CG16} and LightGBM~\citep{KMFWCMYL17}
are widely used and consistently achieve state-of-the-art performance on tabular data.

Despite this empirical success, a rigorous mathematical understanding of gradient
boosting remains incomplete. A large part of the statistical literature has focused on
\emph{consistency}, namely whether boosting procedures can achieve near-optimal
prediction error as the sample size $n$ tends to infinity. Positive answers were
obtained for AdaBoost by \citet{J04} and for more general boosting procedures by
\citet{LV04}, \citet{BLV04}, \citet{B04} and \citet{ZY05}. These analyses typically rely on
approximating the Bayes predictor by linear combinations of simple functions such as
decision trees, together with suitable complexity controls. 

In contrast, results describing the \emph{time dynamics} of gradient boosting -- that is,
the evolution of the predictor as the number of iterations increases -- are much more
limited. A notable exception is the work of \citet{BY03} on linear $L^2$-boosting,
where a spectral analysis of the associated linear operator allows for a precise
description of the trajectory and yields consistency as a consequence. This framework of linear $L^2$-boosting was further analyzed by
\citet{dombry-esstafa-2024} in the vanishing-learning-rate regime.

A complementary perspective was recently given in \citet{DD24a}, who
introduced \emph{infinitesimal gradient boosting} as the limit of
tree-based gradient boosting in the vanishing-learning-rate regime.
More precisely, when $\lambda \to 0$ and the number of iterations is
rescaled as $k = \lfloor t/\lambda \rfloor$, the boosting sequence
converges to a deterministic continuous-time process. This limit is
characterized as the unique solution of a nonlinear ordinary
differential equation (ODE) in an infinite-dimensional function space,
driven by the \emph{infinitesimal gradient boosting operator}
$\mathcal{T}_n$, whose index $n$ indicates its dependence on the data.
A key technical ingredient is the introduction of \emph{softmax
  regression trees}, which replace the discontinuous hardmax selection
of classical regression trees with a smoother softmax selection. This
ensures Lipschitz continuity of the expected tree with respect to the
training sample, which is key in the analysis of the ODE associated to
infinitesimal gradient boosting. The resulting trajectory is
deterministic (conditional on the sample) and has non-increasing
training error.

The large-sample behavior of infinitesimal gradient boosting was
established in \citet{DD24}. In this work, the finite
sample operator $\mathcal{T}_n$ is shown to converge almost surely to
a deterministic \emph{population operator} $\mathcal{T}$, which
defines a corresponding population gradient boosting flow depending
only on the distribution $\mathrm{P}$ of the data $(X,Y)$. This result
can be viewed as a functional law of large numbers for the entire
infinitesimal gradient boosting process. In addition, qualitative
properties of the limit are derived, including the monotonicity of the
population test error.

More recently, \citet{DDD25} introduced a reproducing kernel Hilbert
space (RKHS) framework for tree-based ensemble methods. They showed that softmax
regression trees induce a natural family of data-dependent kernels. The associated RKHSs provides a convenient
functional setting for infinitesimal gradient boosting since the ODE driving infinitesimal gradient boosting can be interpreted as a gradient flow in a Hilbert manifold. This RKHS perspective will be central in the present work.

While \citet{DD24} identifies the deterministic limit of infinitesimal gradient boosting, it
does not describe the \emph{fluctuations} around this limit. Such second-order results,
in the form of central limit theorems (CLTs), are fundamental in statistics: they
quantify uncertainty, enable the construction of confidence bands, and support
hypothesis testing. Although CLTs are well understood for many classical estimators
(see, e.g., \citealp{vdV98}), analogous results for gradient boosting remain largely unexplored. The main contribution of this paper is to fill this gap by establishing a
\emph{functional central limit theorem} for infinitesimal gradient boosting:
\begin{equation*}
  \sqrt{n}\,(F^n_t - F_t)_{t \geq 0}
  \xrightarrow{d} (\mathcal{F}_t)_{t \geq 0},
\end{equation*}
where the convergence in distribution is for differentiable processes
taking values in a suitable RKHS $\widetilde{\mathcal{H}}$. The limit
is a centered continuous Gaussian process characterized as the
solution of a linear stochastic ODE driven by the asymptotic noise of
the empirical process $\sqrt{n}(P_n - P)$.

The proof relies on a general framework for perturbations of autonomous ODEs in Banach
spaces, developed in Section~\ref{sec:ODE}. We show that if the driving vector
field converges at rate $\sqrt{n}$ to a Gaussian limit and satisfies appropriate
smoothness conditions, then the solution of the ODE inherits Gaussian fluctuations,
governed by a linearized equation. This result can be viewed as an infinite-dimensional
functional delta method for ODEs. Applying this framework to infinitesimal gradient boosting requires tools from
empirical process theory, in particular Donsker-type results \citep{vW96}. The RKHS
structure introduced in \citet{DDD25} plays a crucial role in identifying suitable function
classes and controlling the fluctuations of the boosting operator.

\subsection{Outline of the paper}
\label{sec:outline}

The paper develops three layers of results, each building on the previous one. We first
establish a general perturbation theory for ODEs in Banach spaces
(Section~\ref{sec:ODE}). We then derive a functional CLT for kernel gradient flows
(Section~\ref{sec:CLT-KGF}). Finally, we obtain the main functional CLT for
infinitesimal gradient boosting (Section~\ref{sec:CLT-IGB}).

In Section~\ref{sec:ODE}, we develop a general stability theory for autonomous ODEs in a
Banach space $E$. Given a vector field $g: E \to E$, we study the sensitivity of the
solution to perturbations of both the initial condition and the vector field.
\Cref{prop:first-order-perturbation} establishes continuity of the
solution map, while a second-order result (\Cref{prop:second-order-perturbation})
shows that $\sqrt{n}$-scaled perturbations lead to a Gaussian limit governed by a
linearized ODE. A stochastic version yields both a functional law of large numbers and
a functional CLT for random dynamical systems. We do not claim strong
originality for these results in the deterministic setting; however,
we were unable to find statements and their stochastic counterparts
specifically suited to our framework in the literature.

In Section~\ref{sec:CLT-KGF}, we apply this theory to gradient flows
in an RKHS $\mathcal{H}$. We prove almost sure convergence of the
empirical flow to its population counterpart, together with a
functional CLT describing the Gaussian fluctuations. In the
least-squares case, the theory greatly simplifies and the limit solves
a linear ODE in $\mathcal{H}$.

In Section~\ref{sec:CLT-IGB}, we turn to infinitesimal gradient
boosting. We first recover the functional law of large numbers from
\citet{DD24} within the RKHS framework from \citet{DDD25}. Our main result,
\Cref{thm:clt-igb}, is then a functional CLT for the infinitesimal
gradient boosting process, obtained by combining the general
perturbation theory with a CLT for the infinitesimal gradient boosting
operator. This requires working in an enlarged RKHS
$\widetilde{\mathcal{H}}$ to control singular terms arising in the
linearization. The key arguments for the proof rely on empirical
process techniques and Donsker properties of carefully constructed
function classes as well as on technical computations for deriving
various asymptotic expansions.

All main proofs are collected in Section~\ref{sec:proofs}, with
supplementary technical material provided in the appendices.

\subsection{Framework and notation}\label{sec:framework}
\paragraph{Statistical learning framework}
We consider a sample $(X_i,Y_i)_{1\leq i\leq n}$ of $n$ independent copies of a feature/label pair $(X,Y)$ taking values in $\mathcal{X}\times\mathcal{Y}$ with distribution $\P$. Here $\mathcal{X}$ and $\mathcal{Y}$ denote metric spaces that are usually subsets of $\mathbb{R}^p$ and $\mathbb{R}$ respectively. The empirical measure associated with the sample $(X_i,Y_i)_{1\leq i\leq n}$ is denoted
\[
\P_n=\frac{1}{n}\sum_{i=1}^n \delta_{(X_i,Y_i)},
\]
with $\delta_{(x,y)}$ the Dirac mass at $(x,y)$. For a measurable $f:\mathcal{X}\times\mathcal{Y}\to\mathbb{R}$, we write
\[
\P_n[f] = \P_n[f(x,y)] = \int f(x,y)\,\P_n(dx,dy)
= \frac{1}{n}\sum_{i=1}^n f(X_i,Y_i),
\]
and, for measurable sets $A\subset \mathcal{X}$,
\[
\P_n(A) = \frac{1}{n}\sum_{i=1}^n \1_A(X_i).
\]
Note the slight abuse of notation, since we should write formally $\P_n(A\times\mathcal{Y})$. Analogous notation is used for expectations under the population distribution $\P$.

Given a measurable loss function $\ell:\mathbb{R}\times\mathcal{Y}\to [0,+\infty)$, the goal is to estimate the minimizer
\begin{equation}\label{eq:Fstar}
F^*=\mathop{\mathrm{argmin}}_{F\in\mathcal{F}} \mathrm{P}[\ell(F(x),y)],
\end{equation}
where $\mathcal{F}$ denotes the set of measurable functions $F: \mathcal{X}\to\mathbb{R}$. The \emph{theoretical risk}
\begin{equation}\label{eq:population-risk}
\mathcal{L}(F)=\P[\ell(F(x),y)] 
\end{equation}
is not observable and is approximated by its empirical counterpart, the \emph{empirical risk}
\begin{equation}\label{eq:empirical-risk}
\mathcal{L}_n(F)=\P_n[\ell(F(x),y)] .
\end{equation}

The choice of the loss function depends on the statistical task and the most common settings are:
\begin{itemize}
    \item \emph{Least square regression} with $\mathcal{Y}=\mathbb{R}$ and square loss $\ell(z,y)=\frac{1}{2}(y-z)^2$. The theoretical risk is the (half) mean squared error and is minimized by the regression function $F^*(x)=\mathbb{E}[Y\mid X=x]$.
    \item \emph{Binary classification} with $\mathcal{Y}=\{0,1\}$ and 
    $\ell(z,y)= -\,yz +\log(1+e^z)$ the binary cross-entropy loss. The optimal function is  $F^*(x)=\log(p^*(x)/(1-p^*(x))$, the logit of the success probability $p^*(x)=\mathbb{P}(Y=1\mid X=x)$.
\end{itemize}

\smallskip
We assume that the loss is twice continuously differentiable with respect to its first variable and denote
\[
\partial_1\ell(z,y)=\frac{\partial}{\partial z}\ell(z,y),
\qquad
\partial_1^2\ell(z,y)=\frac{\partial^2}{\partial z^2}\ell(z,y).
\]
We always work under the following assumption on the loss function $\ell$.
\begin{assumption}\label{ass:loss}
  For any fixed $y\in \mathcal{Y}$, the function $z\mapsto \ell(z,y)$
  is strictly convex and twice continuously differentiable. Its
  second-order derivative is uniformly bounded:
\[
L := \sup_{z\in \mathbb{R},\,y\in\mathcal{Y}} \partial_1^2 \ell(z,y) < \infty.
\]
Moreover, the function $(z,y)\mapsto \partial_1\ell(z, y)$ is
uniformly continuous, and the following integrability conditions hold:
\[
\P[|\ell(0,y)|]<\infty \quad \mbox{and} \quad \P[ |\partial_1\ell(0,y)|^2] < \infty,
\]
as well as the following coerciveness condition:
\[
  \lim_{\abs{z}\to \infty} \mathrm{P}[\ell(z, y)] = +\infty.
\]
\end{assumption}

Clearly, for least square regression,  Assumption~\ref{ass:loss} reduces to $\P[y^2]<\infty$. For  classification with binary cross-entropy, the integrability assumption is always satisfied since $\partial_1\ell(0,y)=\tfrac{1}{2}-y$ is bounded and coerciveness holds as soon as $y$ is not $\P$-almost surely constant.

\paragraph{Reproducing Kernel Hilbert Spaces}

In this paper, we focus on predictors $F:\mathcal{X}\to\mathbb{R}$ belonging to a reproducing kernel Hilbert space (RKHS). We recall below several basic definitions from RKHS theory.

A \emph{kernel} on $\mathcal{X}$ is a function $k:\mathcal{X}\times\mathcal{X}\to\mathbb{R}$ that is symmetric and positive semidefinite, meaning that:
\begin{itemize}
    \item $k(x,x')=k(x',x)$ for all $x,x'\in\mathcal{X}$;
    \item $\sum_{l,l'=1}^L a_l a_{l'}\, k(x_l,x_{l'}) \ge 0$ for all $L\ge 1$,  $a_1,\ldots,a_L\in\mathbb{R}$, and  $x_1,\ldots,x_L\in\mathcal{X}$.
\end{itemize}
When $\mathcal{X}$ is a measurable space, the kernel is said to be \emph{measurable} if $k$ is measurable on $\mathcal{X}\times\mathcal{X}$ with respect to the product $\sigma$-field.

An RKHS on $\mathcal{X}$ is a subspace $\mathcal{H}$ of the space of functions $ \{F:\mathcal{X}\to\mathbb{R}\}$  equipped with a Hilbert structure such that the evaluation maps $F \mapsto F(x)$ are continuous for all $x\in\mathcal{X}$. We denote its inner product and norm by $\langle\cdot,\cdot\rangle_{\mathcal{H}}$ and $\norm{\cdot}_{\mathcal{H}}$, respectively. By the Riesz representation theorem, for every $x\in\mathcal{X}$ there exists a unique element $k_x\in\mathcal{H}$ such that
\[
\langle k_x , F \rangle_{\mathcal{H}} = F(x),\qquad F\in\mathcal{H}.
\]
The map $\Phi:\mathcal{X}\to\mathcal{H},\ x\mapsto k_x$ is called the \emph{feature map} of the RKHS. 
The function $k:\mathcal{X}\times \mathcal{X}\to\mathcal{R}$ defined by
\[
k(x,x')=\langle k_x, k_{x'}\rangle_{\mathcal{H}}=\langle \Phi(x), \Phi(x')\rangle_{\mathcal{H}}
\]
is a kernel on $\mathcal{X}$, called the \emph{reproducing kernel} of $\mathcal{H}$. It uniquely characterizes $\mathcal{H}$. Conversely, Aronszajn's theorem~\citep{A50} ensures that any kernel gives rise to a unique RKHS for which it is the reproducing kernel; see also \citet[Theorem~2.7.7]{HE15}.

Throughout this paper, any RKHS that we use will satisfy the following
regularity assumption.
\begin{assumption}\label{ass:kernel} 
  The space $\mathcal{X}$ is a separable metric space and the kernel
  $k:\mathcal{X}\times\mathcal{X}\to\mathbb{R}$ of the RKHS
  $\mathcal{H}$ is continuous and bounded, with
\[
  \kappa^2 := \sup_{x\in\mathcal{X}} k(x,x)< \infty.
\]
\end{assumption}

Under Assumption~\ref{ass:kernel}, every function $F\in\mathcal{H}$ is
continuous and bounded and we have
\begin{equation}\label{eq:bound-uniform-norm}
\|F\|_{\infty}
  := \sup_{x\in\mathcal{X}} |F(x)|
  \le \kappa\,\|F\|_{\mathcal{H}}.
\end{equation}
Indeed, using the reproducing property
$F(x)=\langle F, \Phi(x)\rangle_{\mathcal{H}}$ and
$\|\Phi(x)\|_{\mathcal{H}} = \sqrt{k(x,x)}$, the Cauchy--Schwarz inequality
implies
\[
|F(x)|
  = |\langle F, \Phi(x)\rangle_{\mathcal{H}}|
  \le \|F\|_{\mathcal{H}}\,\|\Phi(x)\|_{\mathcal{H}}
  = \|F\|_{\mathcal{H}}\sqrt{k(x,x)}.
\]
Taking the supremum over $x\in\mathcal{X}$ yields~\eqref{eq:bound-uniform-norm}.

Similarly, the continuity of functions in the RKHS follows from the continuity of $k$ thanks to the inequality
\begin{align}
|F(x)-F(x')|&=|\langle F, \Phi(x)-\Phi(x')\rangle_{\mathcal{H}}|\nonumber\\
&\leq \|F\|_{\mathcal{H}} \big(k(x,x)+k(x',x')-2k(x,x'))\big)^{\frac{1}{2}} \label{eq:bound-uniform-continuity},  
\end{align}
where the upper bound tends to $0$ as $x'\to x$.

\paragraph{Tensor product operators.}
Let $\mathcal{H}$ be a separable Hilbert space. For $u,v\in\mathcal{H}$, the mapping
\[
(u\otimes v)(w)=\langle u, w\rangle_{\mathcal{H}}\, v, \qquad w\in\mathcal{H},
\]
defines a rank-one bounded linear operator on $\mathcal{H}$ with operator norm
\[
\|u\otimes v\|_{\mathrm{op}} = \|u\|_{\mathcal{H}}\,\|v\|_{\mathcal{H}}.
\]
Finite linear combinations of such operators form the space of finite-rank operators that can be equipped with an inner product such that
\[
\langle u_1\otimes v_1,\; u_2\otimes v_2\rangle
  = \langle u_1, u_2\rangle_{\mathcal{H}}\,
    \langle v_1, v_2\rangle_{\mathcal{H}}.
\]
The \emph{Hilbert tensor product} $\mathcal{H}\otimes\mathcal{H}$ is the completion of this pre-Hilbert space.  
It coincides with the space of Hilbert--Schmidt operators $\Lambda:\mathcal{H}\to\mathcal{H}$, that is,
operators such that
\[
\|\Lambda\|_{\mathrm{HS}}^2
  = \sum_{i\ge 1} \|\Lambda e_i\|_{\mathcal{H}}^2 < \infty,
\]
for some (equivalently, for any) complete orthonormal system $(e_i)_{i\ge 1}$ of $\mathcal{H}$.

\section{Perturbation analysis of ODEs}\label{sec:ODE}

Our analysis of the asymptotic behavior of kernel gradient flow and infinitesimal gradient boosting relies on general results describing how the solutions of autonomous ordinary differential equations (ODEs) react to perturbations of the vector field driving the dynamics. Throughout this section, let $E$ denote a Banach space, and let a \emph{vector field} be a continuous map $g : E \to E$.

A vector field $g$ is said to be \emph{complete} if it is locally Lipschitz continuous and if, for every initial condition $x_0 \in E$, the Cauchy problem
\begin{equation}\label{eq:abstract-ODE}
\left\{
\begin{array}{ll}
\dot{x}(t) &= g(x(t)), \quad t \ge 0, \\
x(0) &= x_0
\end{array}
\right.
\end{equation}
admits a unique global solution defined for all $t \ge 0$. Local
existence and uniqueness of solutions follow from the Picard--Lindelöf
theorem under the local Lipschitz assumption on $g$. An additional
condition that ensures completeness is the \emph{linear growth
  condition}
\begin{equation}\label{eq:linear-growth}
\limsup_{\|v\|_E \to +\infty} \frac{\|g(v)\|_E}{\|v\|_E} < \infty,
\end{equation}
which controls the growth of $g$ at infinity.

We denote by $C^0_b(E, E)$ the space of continuous maps $g : E \to E$
that are bounded on bounded subsets of $E$, endowed with the topology
of uniform convergence on bounded sets. Similarly,
$C^1([0,+\infty), E)$ denotes the space of continuously differentiable
functions $x : [0,+\infty) \to E$, endowed with the topology of
uniform convergence of both the functions and their derivatives on
bounded time sets. If $E$ is separable, then $C^1([0,+\infty), E)$ is
a Polish space, be $C^0_b(E, E)$ is not in general: it is a completely
metrizable space (this is easily checked since $E$ is complete), but
it is not second-countable if $E$ is infinite-dimensional. Therefore,
when considering random variables with values in $C^0_b(E, E)$, we
will sometimes need to check if they are \emph{tight}, i.e.\ that they
are supported on a countable union of compact subspaces of
$C^0_b(E, E)$.

We consider a sequence $(x_{0,n}, g_n) \in E \times C^0_b(E,E)$, $n \ge 1$, where each $g_n$ is assumed to be a complete vector field. Denote by $x_n \in C^1([0,+\infty),E)$ the unique solution of 
\begin{equation}\label{eq:abstract-ODE-2}
\dot{x}_n(t) = g_n(x_n(t)), \quad t \ge 0, \qquad x_n(0) = x_{0,n}.
\end{equation}

The following two propositions are central to the analysis of the
first- and second-order asymptotics of infinitesimal gradient boosting
as $n \to \infty$. They establish, respectively, the continuity and a
second-order (differentiability-type) stability property of the
solution map $(x_0, g) \mapsto x$, which associates to each pair
$(x_0, g)$ -- consisting of an initial condition and a complete vector
field -- the corresponding solution of~\eqref{eq:abstract-ODE}.

\begin{proposition}[First-order perturbation]\label{prop:first-order-perturbation}
Assume that $(x_{0,n}, g_n) \to (x_0, g)$ in $E \times C^0_b(E, E)$ as $n \to \infty$, and that the limit vector field $g$ is complete. Then the sequence of solutions $(x_n)_{n \ge 1}$ converges in $C^1([0, +\infty), E)$ to the unique solution $x$ of the differential equation~\eqref{eq:abstract-ODE}.
\end{proposition}

For the second-order analysis, we consider the rescaled fluctuations $\sqrt{n}(g_n - g)$ of the vector field around its limit $g$. 
In the setting of infinitesimal gradient boosting, it is natural to allow these fluctuations to take values in a Banach space larger than $E$. 
Let $\mathcal{E}$ be a Banach space such that $E \subset \mathcal{E}$ densely and $\|x\|_{\mathcal{E}} \le \|x\|_E$ for all $x \in E$. 
Define $C^0_b(E, \mathcal{E})$ analogously to $C^0_b(E, E)$ and assume that
\[
\sqrt{n}(g_n - g) \to h \quad \text{in } C^0_b(E, \mathcal{E})
\]
for some limit $h$. Once again this space is completely metrizable,
but not Polish in general.

We further require a regularity property of $g$ relative to the embedding $E \subset \mathcal{E}$. 
We denote by $C^1_{\mathcal{E}}(E, E)$ the class of functions $f : E \to E$ that are Fréchet differentiable on $E$ and satisfy:
\begin{itemize}
\item[(i)] for every $x \in E$, the derivative $\mathrm{d}_x f : (E, \norm{\cdot}_{\mathcal{E}}) \to (E, \norm{\cdot}_{\mathcal{E}})$ is linear and continuous;
\item[(ii)] the map $x \mapsto \mathrm{d}_x f$ is continuous on
  $(E, \norm{\cdot}_E)$ with values in the space of continuous linear
  operators
  $(E, \norm{\cdot}_{\mathcal{E}}) \to (E,
  \norm{\cdot}_{\mathcal{E}})$, equipped with the operator norm
  $\norm{\cdot}_{\mathrm{op}}$.
\end{itemize}
By density of $E$ in $\mathcal{E}$ and the continuity assumptions above, each $\mathrm{d}_x f$ admits a unique continuous extension to $(\mathcal{E}, \norm{\cdot}_{\mathcal{E}})$, which we continue to denote by $\mathrm{d}_x f$. 
Moreover, the map $x \mapsto \mathrm{d}_x f$ remains continuous on $E$.


\begin{proposition}[Second-order perturbation]\label{prop:second-order-perturbation}
Under the assumptions of Proposition~\ref{prop:first-order-perturbation}, suppose there exists a Banach space $\mathcal{E}$ as above such that $g \in C^1_{\mathcal{E}}(E, E)$, and that
\[
\big( \sqrt{n}(x_{0,n} - x_0), \sqrt{n}(g_n - g) \big) \to (y_0, h)
\quad \text{in } \mathcal{E} \times C^0_b(E, \mathcal{E}) \text{ as } n \to \infty.
\]
Then the sequence $y_n := \sqrt{n}(x_n - x)$ converges in $C^1([0, +\infty), \mathcal{E})$ to the unique solution $y$ of the linear differential equation in $\mathcal{E}$
\[
\dot{y}(t) = (\mathrm{d}_{x(t)} g)(y(t)) + h(x(t)), 
\qquad t \ge 0, \qquad y(0) = y_0.
\]
\end{proposition}

The limiting equation for $y$ is \emph{linear} with continuous,
time-dependent coefficients. Existence and uniqueness of a global
solution for all $t \ge 0$ therefore follow from the linear
Picard--Lindelöf theorem.

In the setting of infinitesimal gradient boosting, the initial condition and the vector field depend on the training data and are therefore random. We thus provide, as a corollary, a stochastic counterpart of Propositions~\ref{prop:first-order-perturbation} and~\ref{prop:second-order-perturbation}. 

For each $n \ge 1$, let $X_{0,n}$ be an $E$-valued random variable
(stochastic initial condition) and let $G_n$ be a $C^0_b(E,E)$-valued random variable (stochastic vector field). Assume that, with probability one, $G_n$ is complete, and denote by $X_n \in C^1([0,+\infty),E)$ the unique solution of
\[
\dot{X}_n(t) = G_n(X_n(t)), \quad t \ge 0, \qquad X_n(0) = X_{0,n}.
\]
The following corollary establishes a law of large numbers and a central limit theorem for the processes $X_n = (X_n(t))_{t \ge 0}$.

\begin{corollary}\label{cor:stochastic-perturbation}
\begin{itemize}
    \item[(i)] Suppose that
    \[
    (X_{0,n}, G_n) \xrightarrow{a.s.} (x_0, g)
    \qquad \text{in } E \times C^0_b(E,E),
    \]
    where the limit vector field $g$ is complete. Then 
    \[
    X_n \xrightarrow{a.s.} x 
    \qquad \text{in } C^1([0,+\infty),E),
    \]
    where $x$ denotes the unique solution of the deterministic differential equation~\eqref{eq:abstract-ODE}.
    
  \item[(ii)] In addition, assume that $x_0$ and $g$ are
    deterministic, that $g \in C^1_{\mathcal{E}}(E, E)$ and that
    $(X_{0,n}, G_n)$ satisfy a central limit theorem in the sense that
    \[
    \big( \sqrt{n}(X_{0,n} - x_0), \sqrt{n}(G_n - g) \big)
    \xrightarrow{d} (\mathscr{X}_0, \mathscr{G})
    \qquad \text{in } \mathcal{E} \times C^0_b(E, \mathcal{E}),
    \]
    where the limit $(\mathscr{X}_0, \mathscr{G})$ is tight and
    Gaussian. Then $X_n$ exhibits Gaussian fluctuations:
    \[
    \mathscr{X}_n := \sqrt{n}(X_n - x)
    \xrightarrow{d} \mathscr{X} 
    \qquad \text{in } C^1([0,+\infty), \mathcal{E}),
    \]
    where the limit $\mathscr{X}$ is Gaussian and characterized as the unique solution of
    \begin{equation} \label{eq:gaussian-ODE}
      \dot{\mathscr{X}}(t) = (\mathrm{d}_{x(t)} g)(\mathscr{X}(t)) + \mathscr{G}(x(t)), 
    \qquad t \ge 0, \qquad \mathscr{X}(0) = \mathscr{X}_0.
    \end{equation}
\end{itemize}
\end{corollary}

Let us stress that the tightness of the limit distribution
$(\mathscr{X}_0,\mathscr{G})$ must be assumed in $(ii)$. It is not
automatic here because the space
$\mathcal{E} \times C^0_b(E, \mathcal{E})$ is completely metrizable
but not separable. Tightness allows us to use the Skorokhod
representation theorem (Theorem~1.10.3 from \citet{vW96}) and to
derive the convergence in distribution
$\mathscr{X}_n\xrightarrow{d} \mathscr{X}$ by a functional
delta-method argument. The limit is then automatically tight, since
$C^1([0,+\infty),\mathcal{E})$, equipped with the topology of uniform
convergence of the function and its derivative on compact sets, is
Polish.

\section{Gaussian fluctuations of kernel gradient flows}\label{sec:CLT-KGF}

We now apply the perturbation analysis of Section~\ref{sec:ODE} to study the asymptotic behavior of gradient flows in a Reproducing Kernel Hilbert Space (RKHS) framework.

\subsection{Kernel gradient flow} \label{sec:gradient-flow}

Let $\mathcal{H}$ be an RKHS over a metric space $\mathcal{X}$ with
bounded and continuous reproducing kernel $k$, i.e.\ satisfying
\Cref{ass:kernel}. Recall the definition of the feature map
$\Phi \colon \mathcal{X} \to \mathcal{H}$, defined by
$\Phi(x) = k(x, \cdot)$ and which satisfies the reproducing property
\begin{equation}\label{eq:reproducing}
\langle \Phi(x), F \rangle_{\mathcal{H}} = F(x),
\qquad F \in \mathcal{H}, \ x \in \mathcal{X}.
\end{equation}
Recall also the notation $\mathrm{P}_n$ for the empirical measure
associated with an i.i.d.\ sample $\{(x_i, y_i)\}_{i=1}^n$ with
distribution $\mathrm{P}$. The corresponding empirical and population
risks $\mathcal{L}_n$ and $\mathcal{L}$ are defined
in~\eqref{eq:empirical-risk} and~\eqref{eq:population-risk}.
For all $F, G \in \mathcal{H}$, the directional derivative of $\mathcal{L}_n$ satisfies
\begin{align*}
\langle \nabla_F \mathcal{L}_n, G \rangle_{\mathcal{H}}
&= \lim_{\varepsilon \to 0} 
\frac{1}{\varepsilon}\,
\mathrm{P}_n[\ell(F(x) + \varepsilon G(x), y) - \ell(F(x), y)] \\
&= \mathrm{P}_n[\partial_1 \ell(F(x), y)\, G(x)]\\
&= \mathrm{P}_n[\langle \partial_1 \ell(F(x), y)\, \Phi(x), G \rangle_{\mathcal{H}}],
\end{align*}
so that the gradient of the empirical loss at $F$ is
\begin{equation}\label{eq:empirical-gradient}
\nabla_F \mathcal{L}_n = \mathrm{P}_n[\partial_1 \ell(F(x), y)\, \Phi(x)].
\end{equation}
The associated empirical gradient flow $F^n = (F_t^n)_{t \ge 0}$ is defined as the  solution of
\begin{equation}\label{eq:empirical-flow}
\frac{\mathrm{d} F_t^n}{\mathrm{d}t} = -\nabla_{F_t^n}\mathcal{L}_n,
\qquad t \ge 0, 
\qquad F_0^n = 0.
\end{equation}

Similarly, the population gradient of~\eqref{eq:population-risk} is given by
\begin{equation}\label{eq:population-gradient}
\nabla_F \mathcal{L} = \mathrm{P}[\partial_1 \ell(F(x), y)\, \Phi(x)],
\end{equation}
and the corresponding population gradient flow $F = (F_t)_{t \ge 0}$ satisfies
\begin{equation}\label{eq:population-flow}
\frac{\mathrm{d} F_t}{\mathrm{d}t} = -\nabla_{F_t}\mathcal{L},
\qquad t \ge 0,
\qquad F_0 = 0.
\end{equation}

We first verify in the next lemma that the gradient flows $F^n$ and $F$ are well defined for all $t\geq 0$, and moreover that $F^n$ depends measurably (indeed continuously) on the training sample.

\begin{lemma}\label{lem:gradient-field-regularity} ~
\begin{itemize}
    \item[(i)] For all $n \ge 1$, the gradient vector field 
    $F \longmapsto \nabla_F \mathcal{L}_n$
    is $C^1$ with a bounded Fréchet derivative, and so belongs to $C_b^0(\mathcal{H}, \mathcal{H})$ and satisfies the linear growth condition
    \begin{equation}\label{eq:linear-growth-RKHS}
    \limsup_{\|F\|_{\mathcal{H}} \to \infty}
    \frac{\|\nabla_F \mathcal{L}_n\|_{\mathcal{H}}}{\|F\|_{\mathcal{H}}} < \infty.
    \end{equation}
    The same property holds for the population gradient vector field $F \mapsto \nabla_F \mathcal{L}$.
    
    \item[(ii)] For every $n \ge 1$, the empirical gradient flow~\eqref{eq:empirical-flow} admits a unique global solution in $C^1([0,+\infty), \mathcal{H})$, and the same property holds for the population gradient flow~\eqref{eq:population-flow}. Moreover, $F^n$ is continuous as a function of the sample $\{(x_i, y_i)\}_{1 \le i \le n}$.
\end{itemize}
\end{lemma}
 
\subsection{Large sample asymptotics}  
We now investigate the asymptotic behavior of $F^n$ as $n \to \infty$.  
By applying Corollary~\ref{cor:stochastic-perturbation} to the stochastic vector fields $(\nabla_F \mathcal{L}_n)_{F \in \mathcal{H}}$, we obtain the following law of large numbers and Gaussian limit for the fluctuations.

\begin{theorem}\label{thm:CLT-gradient-flow}
  In addition to Assumptions~\ref{ass:loss} and~\ref{ass:kernel},
  assume that
\begin{equation}\label{eq:phi0-bounded}
\sup_{y\in\mathcal{Y}} \abs{\partial_1\ell(0,y)}<\infty.
\end{equation}
Then:
\begin{itemize}
    \item[(i)] As $n \to \infty$,
    \[
    (F_t^n)_{t \ge 0} \xrightarrow{a.s.} (F_t)_{t \ge 0}
    \quad \text{in } C^1([0,+\infty), \mathcal{H}).
    \]
    
    \item[(ii)] Moreover,
    \[
    \big(\sqrt{n}\,(F_t^n - F_t)\big)_{t \ge 0} 
    \xrightarrow{d} (\mathscr{F}_t)_{t \ge 0}
    \quad \text{in } C^1([0,+\infty), \mathcal{H}),
    \]
    where $(\mathscr{F}_t)_{t \ge 0}$ is a centered Gaussian process in $\mathcal{H}$ characterized as the unique solution of
    \begin{equation}\label{eq:limit-ODE}
    \frac{\mathrm{d} \mathscr{F}_t}{\mathrm{d}t} 
    = A_t \mathscr{F}_t + \mathscr{G}_t, 
    \qquad \mathscr{F}_0 = 0,
    \end{equation}
    where $A = (A_t)_{t \ge 0} \in C^1([0,+\infty), \mathcal{H} \otimes \mathcal{H})$ is given by
    \begin{equation}\label{eq:At}
    A_t = \mathrm{P}\big[\partial_1^2 \ell(F_t(x), y)\, \Phi(x) \otimes \Phi(x)\big],
    \qquad t \ge 0,
    \end{equation}
    and where $(\mathscr{G}_t)_{t \ge 0}$ is an $\mathcal{H}$-valued
    centered Gaussian process with continuous sample paths with the
    same covariance structure as the $\mathcal{H}$-valued process
    $(\partial_1\ell(F_t(x),y)\Phi(x))_{t\geq 0}$ under the
    probability measure $\P$.
\end{itemize}
\end{theorem}

The operator $A_t$ in~\eqref{eq:At} corresponds to the Fréchet derivative of the gradient map 
$F \mapsto \nabla_F \mathcal{L}$ evaluated at $F_t$, while the Gaussian process $\mathscr{G}$ in~\eqref{eq:limit-ODE} represents the asymptotic noise generated by the empirical process $\sqrt{n}(\mathrm{P}_n - \mathrm{P})$.

Let us stress that the boundedness assumption \eqref{eq:phi0-bounded}
is restrictive and is used only in the proof of the technical
Lemma~\ref{lem:Donsker}$(i)$. Alternatively, it could be replaced by
any assumption ensuring that the class $\mathcal{C}_M$ -- which is
defined in the lemma -- is Donsker. Note however
that~\eqref{eq:phi0-bounded} always holds in the classification
setting since $\mathcal{Y}$ is then finite. In the least square
regression setting, the assumption may not be satisfied, but the
result still holds thanks to a different proof -- see
Theorem~\ref{thm:KGF-linear-case} below and its proof relying on the
central limit theorem for square integrable variables in a Hilbert
space.

\subsection{The least squares regression case}
We briefly discuss the case of kernel gradient flow in the least squares regression setting and exploit the fact that the gradient vector field is linear, so the gradient flow has an explicit expression.

For $\ell(z,y)=\frac{1}{2}(z-y)^2$, the gradient vector field is
\[
\nabla_F\mathcal{L}_n=\mathrm{P}_n[(F(x)-y)\Phi(x)]=B_nF - C_n,\quad F\in\mathcal{H},
\]
with random $B_n\in\mathcal{H}\otimes\mathcal{H}$ and $C_n\in\mathcal{H}$ defined by
\begin{equation} \label{eq:B_n-C_n-def}
  B_n=\mathrm{P}_n[\Phi(x)\otimes \Phi(x)]
  \quad\text{and}\quad
  C_n=\mathrm{P}_n[y\Phi(x)].
\end{equation}
The gradient flow then yields the linear differential equation
\[
\frac{\mathrm{d}F^n_t}{\mathrm{d}t}=-B_nF^n_t+C_n,\quad t\geq 0,\qquad F^n_0=0,
\]
with explicit solution 
\begin{equation} \label{eq:least-squares-F_n-explicit}
  F^n_t
  =\int_0^t \mathrm{e}^{-s B_n}\,C_n\,\mathrm{d}s
  = t\,\widetilde{\mathrm{exp}}(-tB_n)\,C_n,\quad t\geq 0,
\end{equation}
where $\widetilde{\mathrm{exp}}(z)=\sum_{k=0}^\infty \frac{z^k}{(k+1)!}=(\mathrm{e}^z-1)/z$. 
The same expressions hold for the infinite population limit:
\[
F_t=\int_0^t \mathrm{e}^{-s B}\,C\,\mathrm{d}s
=t\,\widetilde{\mathrm{exp}}(-tB)\,C,\quad t\geq 0,
\]
with 
\[
B=\mathrm{P}[\Phi(x)\otimes \Phi(x)]
\quad\text{and}\quad
C=\mathrm{P}[y\Phi(x)].
\]
In view of these expressions, the asymptotic behavior of the gradient
flow in the least squares regression case is driven by the asymptotic
behavior of $(B_n,C_n)$.

\begin{theorem}\label{thm:KGF-linear-case}
In the case $\ell(z,y)=\frac12(z-y)^2$, we have:
\begin{itemize}
    \item[(i)] \emph{(LLN and CLT for $(B_n,C_n)$)}. As $n\to\infty$,
    \[
    (B_n,C_n) \xrightarrow{a.s.} (B,C) \quad \text{in } (\mathcal{H}\otimes\mathcal{H})\times\mathcal{H},
    \]
    and
    \[
    \big(\sqrt{n}(B_n{-}B),\ \sqrt{n}(C_n{-}C)\big) \xrightarrow{d}  (\mathscr{B},\mathscr{C}),
    \]
    where $(\mathscr{B},\mathscr{C})$ is centered Gaussian in $(\mathcal{H}\otimes\mathcal{H})\times\mathcal{H}$ with the same covariance structure as $(\Phi(x)\otimes \Phi(x),y\Phi(x))$ under the probability measure $\P$.
    \item[(ii)] \emph{(Limit and fluctuations of the gradient flow)}. Consequently, 
    \[
    F^n\xrightarrow{a.s.} F \quad \text{in } C^1([0,+\infty),\mathcal{H}),
    \]
    and
    \[
      \sqrt{n}\,(F^n-F)\ \xrightarrow{d}\ \mathscr{F}\quad \text{in } C^1([0,+\infty),\mathcal{H}),
    \]
    where the centered Gaussian limit process $\mathscr{F}=(\mathscr{F}_t)_{t\ge 0}$ is characterized by the linear ODE
    \[
      \frac{\mathrm{d}\mathscr{F}_t}{\mathrm{d}t}
      = -B\,\mathscr{F}_t\;-\;\mathscr{B}\,F_t\;+\;\mathscr{C},\qquad \mathcal{F}_0=0,
    \]
    equivalently,
    \[
      \mathscr{F}_t=\int_0^t \mathrm{e}^{-(t-s)B}\,\big(-\mathscr{B}\,F_s+\mathscr{C}\big)\,\mathrm{d}s,\quad t\geq 0.
    \]
\end{itemize}
\end{theorem}

As a comment, we can provide a simple bound for the squared norm of the fluctuations.  
Using $\|e^{-Bt}\|_{\mathrm{op}}\leq 1$, which holds because $B$ is a positive semidefinite operator, we have  
\[
\|F_t\|_{\mathcal{H}} \leq \|C\|_{\mathcal{H}}\, t,
\]
and consequently,
\[
\mathbb{E}\big[\|\mathscr{F}_t\|_{\mathcal{H}}^2\big]
\leq \tfrac{1}{2}\,\sigma_B^2\,\|C\|_{\mathcal{H}}^2\,t^{4}
\;+\;2\,\sigma_C^2\,t^{2},
\]
with $\sigma_B^2=\mathbb{E}[\|\mathscr{B}\|_{\mathrm{HS}}^2]$ and $\sigma_C^2=\mathbb{E}[\|\mathscr{C}\|_{\mathcal{H}}^2]$. Here the constant terms admit the following simple upper bounds:
\begin{align*}
&\|C\|_{\mathcal{H}} \leq \P[\abs{y}k(x,x)^{1/2}] \leq  \kappa \P[y^2]^{1/2},\\
&\sigma_B^2=\P[k(x,x)] \leq \kappa^2,\\
&\sigma_C^2=\P[y^2k(x,x)]\leq \kappa^2 \P[y^2].
\end{align*}

\section{Gaussian fluctuations of infinitesimal gradient boosting}\label{sec:CLT-IGB}

\subsection{Infinitesimal gradient boosting}\label{subsec:IGB-background}

We provide here the necessary background on infinitesimal gradient boosting, a continuous-time limit of gradient boosting obtained when the learning rate tends to zero while the number of iterations is rescaled accordingly. For brevity, we only summarize the main ideas and refer to \cite{DD24a,DD24} for a more detailed presentation.

\textbf{Gradient boosting in the vanishing learning-rate regime.}
Introduced by \cite{F01}, gradient boosting is a sequential algorithm designed to minimize the empirical risk. At each iteration, a (randomized) regression tree is fitted to the pseudo-residuals of the current predictor and added to it with a shrinkage factor $\lambda>0$ (the learning rate). This yields a sequence of predictors $(F_k^{n,\lambda})_{k\ge 0}$ defined recursively from an initialization $F_0^n$ by
\[
F_{k+1}^{n,\lambda}
= F_k^{n,\lambda} + \lambda\, T_{k+1}^{n,\lambda},
\qquad k \ge 0,
\]
where $T_{k+1}^{n,\lambda}$ is a (randomized) regression tree fitted to the pseudo-residuals
\[
r_{k,i}^{n,\lambda}
=
-\partial_1 \ell\big(F_k^{n,\lambda}(x_i),\, y_i\big),
\qquad 1 \le i \le n.
\]
The sequence $(F_k^{n,\lambda})_{k\ge 0}$ therefore forms a Markov chain with values in the space $\mathcal{F}$ of predictors. A standard choice of initialization is the constant predictor minimizing the empirical risk:
\begin{equation}\label{eq:igb-initialization}
  F_0^n
  \equiv
  \mathop{\mathrm{arg\,min}}_{z \in \mathbb{R}} \, \P_n[\ell(z,y)].
\end{equation}
Note that this may not be well-defined for any sample
$(x_i,y_i)_{1\leq i\leq n}$, e.g.\ in the case of binary
classification with the $y_i$ taking only a single value. However, the
convexity and coerciveness conditions of Assumption~\ref{ass:loss} are
enough to ensure that for an i.i.d.\ sample with distribution
$\mathrm{P}$, this initialization makes sense for all $n$ large
enough.

For the specific class of randomized regression trees called \emph{softmax regression trees} (described below), \cite{DD24} showed that a continuous-time limit emerges when $\lambda \to 0$, provided the iteration index is rescaled as $k = \lfloor t/\lambda \rfloor$. More precisely, for all $t \ge 0$,
\[
F^{n,\lambda}_{\lfloor t/\lambda \rfloor}
\;\xrightarrow{a.s.}\;
F_t^n
\qquad \text{as } \lambda \to 0,
\]
and the convergence holds in an appropriate functional sense. The limiting process $(F_t^n)_{t\ge 0}$ is called \emph{infinitesimal gradient boosting}. Remarkably, it is deterministic given the sample $(x_i,y_i)_{1\le i\le n}$, since the randomness inherent in the softmax regression trees vanishes as $\lambda \to 0$.

\textbf{Softmax regression trees and softmax gradient trees.}
The motivation for introducing this class of randomized trees is a Lipschitz continuity property that is essential for analyzing the differential equation governing infinitesimal gradient boosting. A more detailed presentation is given in \citet[Section~2.1]{DD24}.

Given a sample $(x_i,r_i)_{1\le i\le n}$ of covariate/pseudo-residual pairs, a softmax regression tree of depth $d\ge 1$ is constructed by first generating a randomized partition of $[0,1]^p$ into $2^d$ hypercubes (the leaves), denoted $(A_v)_{v\in\{0,1\}^d}$, and then assigning to each leaf the mean of the residuals falling into it. This yields the piecewise-constant predictor
\[
T(z)
=
\sum_{v\in\{0,1\}^d}
\frac{\sum_{i=1}^n r_i\,\1_{A_v}(x_i)}
     {\sum_{i=1}^n \1_{A_v}(x_i)}
\,\1_{A_v}(z),
\qquad z\in [0,1]^p,
\]
with the convention $0/0 = 0$ for empty leaves.

As in standard regression tree construction \citep{BFOS84}, partitions are built by recursive binary splitting. For convenience, we work on $[0,1)^p$ with right-open hypercubes. A split is encoded by a pair $(j,u)$, where $j\in[\![1,p]\!]$ denotes the covariate along which the split is made, and $u\in(0,1)$ specifies the relative split position. For a region $A=\prod_{j=1}^p [a_j,b_j)$,
splitting at $(j,u)$ produces
\begin{align*}
A_0 &= A \cap \{x : x_j < a_j + u(b_j-a_j)\},\\
A_1 &= A \cap \{x : x_j \ge a_j + u(b_j-a_j)\}.
\end{align*}
Iterating this procedure $d$ times yields a partition of $[0,1)^p$ into $2^d$ hypercubes.

The splitting rule determines how splits are selected and hence the distribution of the random partition. In Breiman’s classical model, the split $(j^*,u^*)$ maximizes the score
\begin{align}
\Delta_n
&= \frac{n(A_0)}{n}\,\bar r_n(A_0)^2+\frac{n(A_1)}{n}\,\bar r_n(A_1)^2\nonumber\\
&=\frac{\P_n[\partial_1\ell(F(x),y)\1_{A_0}(x)]^2}{\P_n(A_0)}+\frac{\P_n[\partial_1\ell(F(x),y)\1_{A_1}(x)]^2}{\P_n(A_1)},\label{eq:score}
\end{align}
where $n(A_l)$ and $\bar r_n(A_l)$ denote the number of observations
and the mean residual in $A_l$, $l=0,1$. Softmax regression trees
replace the hard $\mathrm{argmax}$ selection with a $\mathrm{softmax}$
selection applied to $K$ randomly drawn candidate splits. Namely, we
generate independent candidates $\zeta^1,\dots,\zeta^K$ from a
reference distribution, compute the corresponding scores
$\Delta_n^1,\dots,\Delta_n^K$, and select one of the candidate splits
at random with probability
\[
\P(\xi = \zeta^k)
=
\frac{\exp(\beta\,\Delta_n^k)}
     {\sum_{l=1}^K \exp(\beta\,\Delta_n^l)},
\qquad 1\le k \le K.
\]
The limiting case $\beta \to 0$ yields uniform selection among candidates, effectively reducing the model to the completely random tree with $K=1$. When $\beta\to\infty$, the softmax distribution concentrates on the splits maximizing the score, implying that the softmax regression tree model with $\beta,K\to\infty$ converges to Breiman's regression tree model.

To formalize the construction, we use the framework of \emph{splitting
  schemes} \citep[Section~2.1]{DD24a}. Successive splits are indexed
by the internal nodes $\mathscr{T}_{d-1}$ of the complete binary tree
of depth $d$, denoted $\mathscr{T}_d$. A splitting scheme is a
collection $\xi = (\xi_v)_{v\in\mathscr{T}_{d-1}}$ with
$\xi_v = (j_v,u_v)$ encoding the split at internal node $v$. The
partition resulting from the splitting scheme $\xi$ is denoted
$(A_v^\xi)_{v\in\{0,1\}^d}$, corresponding to the $2^d$ tree leaves.

A softmax regression tree depends on the following hyperparameters: the depth $d\ge 1$, the number $K\ge 1$ of candidate splits at each node, the softmax parameter $\beta>0$, and the reference distribution used to generate candidate splits. In \cite{DD24a}, this reference distribution was taken to be uniform on $[\![1,p]\!]\times(0,1)$. We generalize this by considering
\begin{equation}\label{eq:reference-distribution}
\mathcal{U}_\alpha = \mathrm{Unif}([\![1,p]\!]) \otimes \mathrm{Beta}(\alpha,\alpha),
\end{equation}
meaning that $j$ is drawn uniformly from $[\![1,p]\!]$ and, independently, $u$ follows a $\mathrm{Beta}(\alpha,\alpha)$ distribution. uniform case corresponds to $\alpha=1$.

For gradient boosting, the softmax regression tree depends on the pseudo-residuals induced by the current predictor $F$. We denote by $Q_{n,F}(\mathrm{d}\xi)$ the distribution of the splitting scheme of a softmax regression tree built from the sample $(x_i,y_i)_{1\le i\le n}$ with pseudo-residuals at $F$
\[
r_i = -\partial_1 \ell(F(x_i),y_i), \qquad 1\le i\le n.
\]
For a splitting scheme $\xi$ with corresponding partition  $(A_v^\xi)_{v\in\{0,1\}^d}$, the associated gradient tree is defined by
\begin{equation}\label{eq:def-gradient-tree}
T_n(z;F,\xi)=-\sum_{v\in\{0,1\}^d}
\frac{\P_n[\partial_1\ell(F(x),y)\1_{A_v^\xi}(x)]}{\P_n(A_v^\xi)}\,\1_{A_v^\xi}(z),\qquad z\in[0,1]^p.
\end{equation}

For future reference, we summarize the notation in the following definition.
\begin{definition}\label{def:softmax-gradient-tree}
For fixed hyperparameters $d\geq 1$ (tree depth), $K\geq 1$ (number of proposals at each step), $\beta\geq 0$ (softmax selection parameter) and $\alpha >0$ (beta distribution parameter for proposal) and with the sample $(x_i,y_i)_{1\leq i\leq n}$ and a predictor $F\in\mathcal{F}$ as input, we define:
\begin{itemize}
\item $\xi = (\xi_v)_{v\in \mathscr{T}_{d-1}}$ the (random) splitting scheme encoding the successive splits, with distribution $Q_{n,F}$;
\item $(A_v^\xi)_{v\in \{0,1\}^d}$ the partition of $[0,1]^p$ resulting from the splitting scheme $\xi$;
\item $T_n(\,\cdot\,;F,\xi)$ the corresponding softmax gradient tree, as defined in Equation~\eqref{eq:def-gradient-tree}.
\end{itemize}
In the totally random tree model (i.e., when $K=1$ or $\beta=0$), the distribution of the splitting scheme does not depend on the sample $(x_i,y_i)_{1\le i\le n}$ nor on $F$, and is denoted by $Q_0$.
\end{definition}
Importantly, the softmax selection rule used in the construction is reflected in the distribution $Q_{n,F}$ of the splitting scheme. This distribution admits an explicit Radon--Nikodym derivative with respect to the reference measure $Q_0$; see Proposition~\ref{prop:app-RN-splitting-scheme} in Section~\ref{sec:splitting-schemes}. This representation will play a crucial role in our computations.

\textbf{An RKHS framework for infinitesimal gradient boosting.}
As justified in \cite{DD24a}, the vanishing-learning-rate limit of gradient boosting based on softmax gradient trees is well defined and is characterized by a nonlinear ordinary differential equation in a suitable function space. We refer to this limit as \emph{infinitesimal gradient boosting}. Here we adopt the convenient RKHS perspective  introduced in \cite{DDD25}. Recall that some background on RKHS theory is introduced in Section~\ref{sec:framework}.

Following \cite{DDD25}, we define the (infinite) gradient forest as the averaged gradient tree:
\begin{align}
\mathcal{T}_n(z;F)
&= \int T_n(z;F,\xi)\, Q_{n,F}(\mathrm{d}\xi) \label{eq:def-rf}\\
&= - \int \sum_{v\in\{0,1\}^d}
\frac{\P_n\big[\,\partial_1 \ell(F(x),y)\,\1_{A_v^\xi}(x)\,\big]}
     {\P_n(A_v^\xi)}\,
\1_{A_v^\xi}(z)\,
Q_{n,F}(\mathrm{d}\xi) \nonumber
\end{align}
for all $z \in [0,1]^p$. 

We let $\mathcal{H}$ be the RKHS with  reproducing kernel $k$  defined by
\begin{equation}\label{eq:RKHS-kernel}
k(z,z')
=
\int \sum_{v\in\{0,1\}^d}
\1_{A_v^\xi}(z)\1_{A_v^\xi}(z') \,
Q_0(\mathrm{d}\xi),
\qquad z,z' \in [0,1]^p.
\end{equation}
It will be a convenient function space in our approach because of the
following lemma.
\begin{lemma}\label{lem:RKHS}
For all  $F\in\mathcal{F}$, the function $\mathcal{T}_n(F)=\mathcal{T}_n(\cdot;F)$ defined by~\eqref{eq:def-rf} belongs to $\mathcal{H}$.
\end{lemma}
The kernel $k$ depends only on the proposal distribution for the
candidate splits. By \citet[Proposition~8]{DDD25}, it is continuous and
bounded on $[0,1]^p\times[0,1]^p$ under the choice
\eqref{eq:reference-distribution} for the proposal distribution. This
implies that the RKHS $\mathcal{H}$ satisfies \Cref{ass:kernel}, and
therefore all functions in $\mathcal{H}$ are continuous and
bounded on $[0,1]^p$.

We are now ready to define infinitesimal gradient boosting as the
solution of an ordinary differential equation in the RKHS
$\mathcal{H}$.
\begin{definition}\label{def:igb}
  Let $d \ge 1$, $K \ge 1$, $\beta > 0$, and $\alpha > 0$ be fixed
  hyperparameters, and let $(x_i,y_i)_{1\le i\le n}$ be a fixed input
  sample such that \eqref{eq:igb-initialization} is well-defined. The
  infinitesimal gradient boosting process $(F_t^n)_{t\ge 0}$ is the
  unique solution in $\mathcal{H}$ of the ODE
\begin{equation}\label{eq:ode-igb}
\frac{\mathrm{d}}{\mathrm{d}t} F_t^n = \mathcal{T}_n(F_t^n),
\qquad t \ge 0,
\end{equation}
with initial condition $F_0^n$ given by \eqref{eq:igb-initialization}.
\end{definition}

A key qualitative property of infinitesimal gradient boosting is that
the empirical risk $t \mapsto \mathcal{L}_n(F_t^n)$ defined
by~\eqref{eq:empirical-risk} is non-increasing. This reflects the
fundamental design principle of gradient boosting, which aims at
monotonically decreasing the empirical risk along its trajectory.

We conclude this background on finite-sample infinitesimal gradient
boosting by verifying that the solution of~\eqref{eq:ode-igb} is well
defined and measurable (and even continuous) with respect to the input
sample. The next lemma is very similar to
Lemma~\ref{lem:gradient-field-regularity} from
Section~\ref{sec:gradient-flow}.

\begin{lemma}\label{lem:gradient-field-regularity-igb}~
\begin{enumerate}[(i)]
    \item For every $n \ge 1$, the infinitesimal gradient boosting operator
    $F \longmapsto \mathcal{T}_n(F)$  belongs to $C_b^1(\mathcal{H}, \mathcal{H})$ and satisfies the linear growth condition
    \begin{equation}\label{eq:linear-growth-Tn}
    \limsup_{\|F\|_{\mathcal{H}} \to \infty}
    \frac{\|\mathcal{T}_n(F)\|_{\mathcal{H}}}{\|F\|_{\mathcal{H}}}
    < \infty.
    \end{equation}
  \item For every $n \ge 1$ and input sample
    $\{(x_i, y_i)\}_{1 \le i \le n}$ such
    that~\eqref{eq:igb-initialization} is well-defined, the
    differential equation~\eqref{eq:ode-igb} on $\mathcal{H}$ admits a
    unique global solution $F^n=(F^n_t)_{t\geq 0}$ in
    $C^1([0,+\infty), \mathcal{H})$. Moreover, $F^n$ is continuous as
    a function of the sample $\{(x_i, y_i)\}_{1 \le i \le n}$.
\end{enumerate}
\end{lemma}
A version of this Lemma holds for the population limit and is provided as Lemma~\ref{lem:gradient-field-regularity-igb-pop}
in Section~\ref{sec:proofs}.

\subsection{Large sample asymptotics}\label{subsec:IGB-large-sample}

We now study the large-sample behavior of infinitesimal gradient boosting and establish a functional law of large numbers together with a functional central limit theorem for the continuous-time process $(F_t^n)_{t\geq 0}$, taking values in the RKHS $\mathcal{H}$. A related law of large numbers was proved in \cite{DD24} in a different functional setting. The main contribution of the present section is therefore the functional central limit theorem, which characterizes the fluctuations around the deterministic limit.

\subsubsection{Functional law of large numbers} According to
Definition~\eqref{eq:ode-igb}, the differential equations governing
the dynamics of infinitesimal gradient boosting are driven by the
vector fields $\mathcal{T}_n:\mathcal{H}\to\mathcal{H}$, $n\geq 1$. We
first state a law of large numbers for these vector fields, from which
the law of large numbers for infinitesimal gradient boosting follows
by Corollary~\ref{cor:stochastic-perturbation}. The following theorem
is a reformulation in the RKHS $\mathcal{H}$ of convergence results
established in a different functional setting in \cite[Theorems~2.10
and~2.13]{DD24}.

 
\begin{theorem}\label{thm:lln-igb}
In addition to \Cref{ass:loss}, assume that \begin{equation}\label{eq:cond-expectation-bounded}
    \sup_{x\in\mathcal{X}} \mathbb{E}[|\partial_1\ell(0,Y)|^2 \,\mid\, X=x]<\infty.
\end{equation}
Then:
\begin{itemize}
\item[(i)] The vector fields converge almost surely:
\[
(\mathcal{T}_n(F))_{F\in\mathcal{H}} \stackrel{a.s.}{\longrightarrow} (\mathcal{T}(F))_{F\in\mathcal{H}}
\qquad \text{in  $C_b^0(\mathcal{H},\mathcal{H})$ as $n\to\infty$}.
\]
\item[(ii)] As a consequence, if $F^n_0 \stackrel{a.s.}{\longrightarrow} F_0$, then the infinitesimal gradient boosting processes converge almost surely:
\[
(F_t^n)_{t\geq 0} \stackrel{a.s.}{\longrightarrow} (F_t)_{t\geq 0}
\qquad \text{in $C^1([0,+\infty),\mathcal{H})$ as $n\to\infty$},
\]
where the limit $F=(F_t)_{t\geq 0}$ is the unique solution of the ODE 
\begin{equation}\label{eq:ode-igb-population}
\frac{\mathrm{d}}{\mathrm{d}t} F_t = \mathcal{T}(F_t),\qquad t \ge 0,
\end{equation}
with initial condition $F_0$.
\end{itemize}
\end{theorem}

In the least square regression and binary classification settings introduced in Section~\ref{sec:framework}, the almost sure convergence of initial positions $F^n_0 \stackrel{a.s.}{\longrightarrow} F_0$ is satisfied, as can be seen by a direct application of the law of large numbers.

\begin{remark}
 For the sake of brevity, we do not provide here a detailed description of the limit operator $\mathcal{T}:\mathcal{H}\to\mathcal{H}$. This definition is given in the section devoted to proofs, see Lemma~\ref{lem:RKHS-rf} below. More intuitively, as discussed in Definition~2.9 of \cite{DD24}, one can define a population softmax gradient tree $T(z;F,\xi)$ and a population gradient forest $\mathcal{T}(z;F)$ by replacing the empirical measure $\P_n$ by the population measure $\P$ in the formulas defining the finite sample versions. In particular, the mean residual $\bar r_n(A)$ becomes
    \[
    \bar r(A)=-\frac{\P[\partial_1\ell(F(x),y)\1_{A}(x)]}{\P(A)}
    \]
    and the score $\Delta_n$ of the split $A=A_0\cup A_1$ defined by Equation~\eqref{eq:score} becomes
    \[
    \Delta=\frac{\P[\partial_1\ell(F(x),y)\1_{A_0}(x)]^2}{\P(A_0)}+\frac{\P[\partial_1\ell(F(x),y)\1_{A_1}(x)]^2}{\P(A_1)}.
    \]
 \end{remark}

 \begin{remark}
   Since the reproducing kernel $k$ associated with $\mathcal{H}$ is
   continuous and bounded, by~\eqref{eq:bound-uniform-norm} the
   embedding of $\mathcal{H}$ into
   $(C^0([0,1]^p,\mathbb{R}), \norm{\cdot}_{\infty})$, the space of
   continuous functions endowed with the uniform norm, is continuous.
   Therefore, the almost sure convergence also holds in
   $C^1([0,+\infty), C^0([0,1]^p,\mathbb{R}))$.
\end{remark}

\subsubsection{Functional central limit theorem}
To analyze the fluctuations of infinitesimal gradient boosting, we need to work in a larger function space and introduce a family of  RKHSs extending $\mathcal{H}$.  For $\gamma\geq 0$, we define the weight $w_\gamma(A)$ of a hypercube $A\subset [0,1]^p$ by
\begin{equation}\label{eq:def-w-gamma}
w_\gamma(A) =
\begin{cases}
\P(A)^{-\gamma}, & \text{if $\P(A)>0$},\\
1, & \text{if $\P(A)=0$}.
\end{cases}
\end{equation}
This function is non-decreasing with respect to $\gamma$ and satisfies $w_0(A)=1$ for all $A$. We also define the weight of a splitting scheme $\xi$ by
\[
w_\gamma(\xi)=\max_{v\in\{0,1\}^d} w_\gamma(A_v^\xi).
\]
We finally define the critical exponent 
\begin{equation}\label{eq:gamma-max}
\gamma_{\max}
=
\sup\Big\{\gamma\geq 0 \,:\, \int w_\gamma(\xi)\, Q_0(\mathrm{d}\xi)<\infty\Big\}\in[0,+\infty].
\end{equation}
This critical exponent depends on the distribution $Q_0$ of the splitting scheme (and hence on the parameter $\alpha>0$ of the Beta distribution) as well as on the data distribution $\P$. Roughly speaking, it quantifies the integrability of $1/\P(A)$, where $A$ denotes a typical leaf in the random partition generated according to the splitting scheme distribution $Q_0$.

In the following, we provide a functional central limit theorem for the fluctuations of the vector field $\mathcal{T}_n$ around its limit $\mathcal{T}$ and then deduce a functional central limit theorem for the fluctuations of the infinitesimal gradient boosting processes $F^n$ around its limit $F$.

\begin{theorem}\label{thm:clt-igb}
  In addition to \Cref{ass:loss}, assume
  that~\eqref{eq:cond-expectation-bounded} holds. If
  $\gamma_{\max}>1$, then there exists an RKHS
  $\widetilde{\mathcal{H}}$ with continuous and bounded reproducing
  kernel $\tilde k$ such that
\begin{itemize}
    \item[-] $\mathcal{H}\subset \widetilde{\mathcal{H}}$ and $\|F\|_{\widetilde{\mathcal{H}}}\leq \|F\|_{\mathcal{H}}$ for all $F\in\mathcal{H}$;
    \item[-] $\mathcal{H}$ is dense in $\widetilde{\mathcal{H}}$,
\end{itemize}
and the following properties hold: 
\begin{itemize}
\item[(i)] The fluctuations of the vector fields satisfy a  weak convergence
\[
\big(\sqrt{n}\,(\mathcal{T}_n(F) - \mathcal{T}(F))\big)_{F\in\mathcal{H}}
\;\stackrel{d}{\longrightarrow}\;
(\mathscr{W}(F))_{F\in\mathcal{H}}
\]
in the space $C_b^{0}(\mathcal{H}, \widetilde{\mathcal{H}})$ with a
tight Gaussian limit process $\mathscr{W}$.
\item[(ii)] It the initial conditions satisfy
\begin{equation}\label{eq:igb-initialization-normal}
F_0^n=\P_n[f_0(y)]+o_{\P}(n^{-1/2}), \qquad \text{as $n\to\infty$},
\end{equation}
 with $f_0$ such that $\P[f_0(y)^2]<\infty$, then the fluctuations of the infinitesimal gradient boosting processes satisfy the weak convergence
\[
\big(\sqrt{n}(F^n_t - F_t)\big)_{t\ge 0}
\;\stackrel{d}{\longrightarrow}\;
(\mathscr{F}_t)_{t\ge 0}
\qquad\text{as } n\to\infty,
\]
in the space $C^1([0,\infty), \widetilde{\mathcal{H}})$ with a Gaussian limit process $\mathscr{F}$. The limit process is equal in distribution to the unique solution of the differential equation in $\widetilde{\mathcal{H}}$
\begin{equation}\label{eq:thm-gaussian-ode}
\frac{\mathrm{d}\mathscr{F}_t}{\mathrm{d}t}
=
\big(\mathrm{d}_{F_t}\mathcal{T}\big)(\mathscr{F}_t)
+\mathscr{W}(F_t),
\qquad t\ge 0,
\end{equation}
with initial condition $\mathscr{F}_0$ following a centered Gaussian distribution with variance $\P[f_0(y)^2]-\P[f_0(y)]^2$.
\end{itemize}
\end{theorem}

The joint distribution of $(\mathscr{F}_0,\mathscr{W})$ is difficult to describe explicitly; this Gaussian element arises through computations in the spirit of the functional delta method. Moreover, we will state during the proof that the map $\mathcal{T} : \mathcal{H} \to \mathcal{H}$ belongs to
$C^1_{\widetilde{\mathcal{H}}}(\mathcal{H},\mathcal{H})$, with the
notation of \Cref{prop:second-order-perturbation}, so that its differential
$\mathrm{d}_{ F}\mathcal{T}$ extends continuously to vectors in $\widetilde{\mathcal{H}}$. This guarantees that the differential equation~\eqref{eq:thm-gaussian-ode} is well-defined in
$\widetilde{\mathcal{H}}$ and thus is a linear ODE with continuous time-dependent coefficients. The existence and uniqueness of a global solution defined for all $t \geq 0$ are granted by the linear Picard--Lindelöf Theorem.

\begin{remark} We did not include the explicit definition of
  $\widetilde{\mathcal{H}}$ in Theorem~\ref{thm:clt-igb} since it is
  rather technical. The precise definition of
  $\widetilde{\mathcal{H}}$ -- more precisely, of a family of RKHSs
  $(\widetilde{\mathcal{H}}_\gamma)_{\gamma\in [0,\gamma_{\max})}$ -- is
  given in Lemma~\ref{lem:RKHS-H-gamma} in the section devoted to
  proofs. Note that, because the reproducing kernel $\tilde{k}$ is
  bounded and continuous, the embedding
  $\widetilde{\mathcal{H}} \hookrightarrow (C^0([0,1]^p,\mathbb{R}), \norm{\cdot}_{\infty})$ is
  continuous. As a consequence, the convergence stated in point~$(i)$
  also holds in the space
  $C_b^{0}\big(\mathcal{H}, C^0([0,1]^p,\mathbb{R})\big)$, and
  similarly the convergence in point~$(ii)$ holds in the space
  $C^1\big([0,\infty), C^0([0,1]^p,\mathbb{R})\big)$.
\end{remark}

\begin{remark}
For the two standard settings of regression and binary classification, we can check that the initialization~\eqref{eq:igb-initialization} satisfies the asymptotic expansion~\eqref{eq:igb-initialization-normal}. For least-squares regression, the empirical risk is minimized by the empirical mean $F_0^n=\P_n[y]$, 
so that~\eqref{eq:igb-initialization-normal} holds trivially with $f_0(y)=y$. For binary classification with cross-entropy loss, the minimizing constant is
\[
F_0^n=\log\!\left(\frac{\P_n(y)}{1-\P_n(y)}\right).
\]
By the delta method applied with $g(u)=\log(u/(1-u))$ at  $p=\P(y)$,
\begin{align*}
F_0^n &= g(\P_n(y))= g(p)+g'(p)\big(\P_n(y)-p\big)+o_{\P}(n^{-1/2})\\
&= \P_n[f_0(y)]+o_{\P}(n^{-1/2}),
\end{align*}
with
\[
f_0(y)=\log\!\left(\frac{p}{1-p}\right)+\frac{y-p}{p(1-p)}.
\]
Hence Equation~\eqref{eq:igb-initialization-normal} holds also in the case of binary classification.
\end{remark}

\begin{remark}
The condition $\gamma_{\max}>1$ imposes an integrability requirement
on $1/\P(A)$, where $A$ denotes a typical leaf in the random partition
generated according to the splitting scheme distribution $Q_0$. When
$\P(A)$ is bounded from below by $\varepsilon |A|$, where $|A|$
denotes the Lebesgue measure of $A$ -- for instance when $X$ is
uniformly distributed on $[0,1]^p$ or admits a density bounded away
from zero -- one can show that $\gamma_{\max}=\alpha$, the parameter of the Beta distribution defining $Q_0$.

Indeed, in the case $\P(A)=|A|$, each binary split multiplies the volume of the parent cell by an independent random factor with distribution $\mathrm{Beta}(\alpha,\alpha)$. As a result, the volume of a typical leaf has the same distribution as
\[
V=\prod_{l=1}^d \beta_l,
\qquad
\beta_1,\ldots,\beta_d \stackrel{i.i.d.}{\sim} \mathrm{Beta}(\alpha,\alpha),
\]
and a straightforward calculation shows that $\mathbb{E}[V^{-\gamma}]<\infty$ if and only if $\gamma<\alpha$.
Intuitively, larger values of $\alpha$ favor more balanced splits, which prevent the creation of very small leaves and thus improve the integrability of $1/\P(A)$.
\end{remark}

\begin{remark}
As discussed in the previous remark, when $X$ admits a density bounded from below, one has $\gamma_{\max}=\alpha$, and therefore Theorem~\ref{thm:clt-igb} holds for $\alpha>1$ only. It is natural to ask what happens when $\alpha\leq 1$, and in particular in the canonical case $\alpha=1$, which corresponds to split locations uniformly distributed on $(0,1)$. Unfortunately, our proof does not extend to this regime.

Nevertheless, the case of totally random trees, corresponding to $\beta=0$, can still be treated. In this setting, the Radon--Nikodym derivative of the splitting scheme distribution reduces to the constant $1$ (that is, $\Psi_2\equiv 1$ in Lemma~\ref{lem:tech-P}), which substantially simplifies the analysis. This allows one to adapt the proof and establish convergence in a reproducing kernel Hilbert space associated with an unbounded kernel. More precisely, when $\alpha>1/2$ and $\gamma\in(2-\alpha,1+\alpha)$, we can work with the RKHS $\widetilde{\mathcal{H}}_\gamma$ with reproducing kernel
\[
\tilde{k}_\gamma(z,z')
=
\int \sum_{v\in\{0,1\}^d}
w_\gamma(A_v^\xi)\,
\1_{A_v^\xi}(z)\1_{A_v^\xi}(z')\,
Q_0(\mathrm{d}\xi),
\qquad z,z'\in[0,1]^p.
\]
Similarly to \cite{DDD25}, for $\gamma\leq \alpha$ one can show that this kernel is unbounded but satisfies $\mathrm{P}[\tilde{k}_\gamma(x,x)]<\infty$, so that functions in  $\widetilde{\mathcal{H}}_\gamma$ are square integrable with respect to $\P$. Details are omitted for the sake of brevity.
\end{remark}

\section{Proofs}\label{sec:proofs}

\subsection{Preliminaries on empirical process theory}

Our approach uses empirical process theory and, in particular, the notions of Glivenko--Cantelli and Donsker classes. The Glivenko--Cantelli property is used to establish almost sure convergence and laws of large numbers, whereas the Donsker property allows one to study fluctuations around the limit and obtain (functional) central limit theorems.  In our analysis of the large-sample behavior of kernel gradient flow and infinitesimal gradient boosting, the following lemma plays a crucial role for the asymptotic analysis of the underlying vector fields. 

Following \citet[Section~19.2]{vdV98}, a class $\mathcal{C}$ of measurable functions $\phi:(x,y)\mapsto \phi(x,y)$ is said to be $\mathrm{P}$\emph{-Glivenko--Cantelli} if
\[
\sup_{\phi\in\mathcal{C}} 
\big| \mathrm{P}_n[\phi] - \mathrm{P}[\phi] \big|
\stackrel{*-a.s.}{\longrightarrow} 0,
\qquad n\to\infty,
\]
where $\mathrm{P}_n = n^{-1}\sum_{i=1}^n \delta_{(X_i,Y_i)}$ is the empirical measure associated with an i.i.d.\ sample $(X_i,Y_i)_{1\le i\le n}$ drawn from $\mathrm{P}$.  
The convergence $\stackrel{*-a.s.}{\longrightarrow}$ denotes almost sure convergence in \emph{outer} probability, a notion introduced to handle the potential non-measurability of the supremum.

A class $\mathcal{C}$ is said to be $\mathrm{P}$\emph{-Donsker} if the empirical process
\[
\mathbb{G}_n f 
= \sqrt{n}\,\big(\mathrm{P}_n[f] - \mathrm{P}[f]\big),
\qquad f\in\mathcal{C},
\]
converges in outer distribution in $\ell^\infty(\mathcal{C})$ to a tight limit process $\mathbb{G}_P$.  The limit $\mathbb{G}_P$ is necessarily a $\mathrm{P}$-Brownian bridge, that is, a tight centered Gaussian process in $\ell^\infty(\mathcal{C})$ with covariance
\[
\mathbb{E}\big[\mathbb{G}_P\phi_1 \,\mathbb{G}_P\phi_2\big]
= \mathrm{P}[\phi_1\phi_2] - \mathrm{P}[\phi_1]\,\mathrm{P}[\phi_2].
\]
We refer to \citet[Chapters~1.2 and~1.3]{vW96} for details on outer probability and convergence in outer distribution. Note that a $\mathrm{P}$-Donsker class is automatically $\mathrm{P}$-Glivenko--Cantelli.

\begin{lemma}\label{lem:Donsker} ~
\begin{itemize}
\item[(i)] In the kernel gradient flow framework, assume furthermore
  that~\eqref{eq:phi0-bounded} holds and let $\mathcal{X}$ be a
  separable metric space and $\mathcal{H}$ be an RKHS with bounded
  continuous kernel on $\mathcal{X}\times \mathcal{X}$. Then, for all
  $M>0$, the class of functions
\[
\mathcal{C}_M
= \big\{(x,y)\mapsto \partial_1\ell(F(x),y) \, G(x)
:\ \|F\|_{\mathcal{H}}\le M,\; \|G\|_{\mathcal{H}}\le 1 \big\}
\]
is $\mathrm{P}$-Donsker and therefore $\mathrm{P}$-Glivenko--Cantelli.

\item[(ii)]
In the boosting framework, assume that $\mathcal{X}=[0,1]^p$, and let $\mathcal{R}$ denote the class of hyperrectangles of the form $A=[a,b\rangle$ for some $a,b\in[0,1]^p$ with $a\le b$.  
Then, for all $M>0$, the class of functions
\[
\mathcal{C}_M'
= \big\{(x,y)\mapsto \partial_1\ell(F(x),y)\,\mathds{1}_A(x)
:\ \|F\|_{\mathcal{H}}\le M,\; A\in\mathcal{R}\big\}
\]
is $\mathrm{P}$-Donsker and therefore $\mathrm{P}$-Glivenko--Cantelli.
\end{itemize}
\end{lemma}

\begin{proof}
  In both cases, we have an embedding (which is of course a linear
  continuous mapping) from the RKHS $\mathcal{H}$ to the space of
  bounded continuous functions on $C^0_b(\mathcal{X},\mathbb{R})$.
  Furthermore by assumption, in both cases $\mathcal{H}$ is a
  separable Hilbert space and $\mathcal{X}$ is a second countable
  topological space (since it is a separable metric space), so the
  assumptions of \citet[Theorem~1.1]{M85} are satisfied, and this
  result implies that the unit ball of $\mathcal{H}$ is a Donsker
  class of functions (with respect to any probability measure). More
  generally, any bounded subset of $\mathcal{H}$ is Donsker as well.
  We then use permanence properties of Donsker and VC-classes under
  sums, products, and Lipschitz transformations (see
  \citealp[Section~2.10]{vW96}) to prove the two claims.

\emph{Proof of~(i).}
To apply the permanence properties, we use the decomposition
\[
\partial_1\ell(F(x),y)G(x)
= \big(\partial_1\ell(F(x),y)-\partial_1\ell(0,y)\big)G(x)
+ \partial_1\ell(0,y)\,G(x),
\]
where the factor
$|\big(\partial_1\ell(F(x),y)-\partial_1\ell(0,y)\big)|\leq L|F(x)|$
is uniformly bounded in the first term. We assume without loss of
generality that $M\geq 1$, and we introduce the following intermediary classes of functions:
\begin{align*}
  \mathcal{D}_1 &= \big\{(x,y)\mapsto F(x):\ \|F\|_{\mathcal{H}}\le M\big\},\\
  \mathcal{D}_2 &= \big\{(x,y)\mapsto \partial_1\ell(F(x),y):\ \|F\|_{\mathcal{H}}\le M\big\},\\
  \mathcal{D}_3 &= \big\{(x,y)\mapsto \partial_1\ell(F(x),y)-\partial_1\ell(0,y):\ \|F\|_{\mathcal{H}}\le M\big\}.
\end{align*}

By \citet{M85}, the class $\mathcal{D}_1$ is Donsker (the presence of
the redundant variable $y$ is harmless). To show that $\mathcal{D}_2$
is $\mathrm{P}$-Donsker, we can use Theorem~2.10.6 in~\cite{vW96}.
Indeed, this result applies if we can show that these maps are of the
form
\[
  \mathcal{D}_2 = \{(x,y) \mapsto \phi(f(x,y), g(x,y)) :\ (f,g)\in
  \mathcal{F}\times \mathcal{G}\},
\]
where $\mathcal{F}$ and $\mathcal{G}$ are $\mathrm{P}$-Donsker classes satisfying
\begin{equation}
  \label{eq:van-der-vaart-cond-1}
  \sup_{f\in \mathcal{F}\cup \mathcal{G}}\abs{\mathrm{P}[f]} < \infty
\end{equation}
and $\phi$ is such that there exists $L_\phi$ such that for all
$(x,y)\in \mathcal{X}\times \mathcal{Y}$ and
$f_1,f_2\in \mathcal{F}, g_1,g_2\in \mathcal{G}$, the Lipschitz condition
\begin{equation}
  \label{eq:phi-lipschitz-condition}
  \begin{aligned}
    &\abs{\phi(f_1(x,y), g_1(x,y)) - \phi(f_2(x,y), g_2(x,y))}^2  \\
    & \quad \leq L_\phi \abs{f_1(x,y)-f_2(x,y)} +  \abs{g_1(x,y)-g_2(x,y)}
  \end{aligned}
\end{equation}
holds. We let $\mathcal{F}=\mathcal{D}_1$ and
$\mathcal{G}=\{(x,y)\mapsto s(y)\}$, where $s:\mathbb{R} \to (0,1)$ is
any fixed sigmoid function, whose inverse we denote by $s^{-1}$. These
classes are Donsker (note that $\mathcal{G}$ contains a single bounded
function). Furthermore, by~\eqref{eq:bound-uniform-norm} and the
boundedness of the kernel, the maps in $\mathcal{D}_1$ are also
uniformly bounded, so~\eqref{eq:van-der-vaart-cond-1} is satisfied.
Finally, to obtain $\mathcal{D}_2$ from this construction, we must
take $\phi$ of the form $\phi(z,t)=\partial_1 \ell(z, s^{-1}(t))$,
for which~\eqref{eq:phi-lipschitz-condition} is satisfied under
Assumption~\ref{ass:loss} with $L_\phi=L^2$, since $(z,y)\mapsto \partial_1\ell(z,y)$ is
$L$-Lipschitz in its first argument, uniformly in $y\in \mathcal{Y}$.
To apply Theorem~2.10.6 in~\cite{vW96}, it only remains to check that
at least one map in $\mathcal{D}_2$ is square-integrable, which is
immediate using the assumption
$\mathrm{P}[\partial_1\ell(0,y)^2] < \infty$. Therefore
$\mathcal{D}_2$ is $\mathrm{P}$-Donsker.

The class $\mathcal{D}_3$ is Donsker because it is contained in
$\mathcal{D}_2 - \mathcal{D}_2$ (pairwise differences), which is
Donsker by Example~2.10.7 in \cite{vW96}. Likewise, the class
$\mathcal{D}_3 \cdot \mathcal{D}_1$ of pairwise products is Donsker as
the product of two uniformly bounded Donsker classes (Example~2.10.8
in \citealp{vW96}). Finally, $\mathcal{C}_M$ is included in the class
$\mathcal{D}_3 \cdot \mathcal{D}_1 + \phi_0 \cdot \mathcal{D}_1$,
where the function $\phi_0(x,y)=\partial_1\ell(0,y)$ is fixed.
According to Equation~\eqref{eq:phi0-bounded}, $\varphi_0$ is assumed
to be bounded so that the class $\phi_0 \cdot \mathcal{D}_1$ is
Donsker, and finally the sum of the Donsker classes
$\mathcal{D}_3 \cdot \mathcal{D}_1$ and $\phi_0 \cdot \mathcal{D}_1$
is also Donsker.

\emph{Proof of~(ii).}
The proof is similar, except that $G(x)$ is replaced by $\mathds{1}_A(x)$.  
Let
\[
\mathcal{D}_4=\big\{(x,y)\mapsto \mathds{1}_A(x):\ A\in\mathcal{R}\big\}.
\]
Hyperrectangles form a VC-class, so $\mathcal{D}_4$ is
$\mathrm{P}$-Donsker and uniformly bounded. Thus
$\mathcal{D}_3\cdot\mathcal{D}_4$ is $\mathrm{P}$-Donsker. Since
multiplication by the fixed function $\phi_0$ preserves the
VC-subgraph property (Lemma~2.6.18(vi) in \citealp{vW96}), the class
$\phi_0\cdot\mathcal{D}_4$ is VC-subgraph and hence
$\mathrm{P}$-Donsker. Note that we did not need to
assume~\eqref{eq:phi0-bounded} here. Finally, $\mathcal{C}_M'$ is
included in
$\mathcal{D}_3\cdot\mathcal{D}_4 + \phi_0\cdot\mathcal{D}_4$, and is
therefore $\mathrm{P}$-Donsker.
\end{proof}

\subsection{Proofs related to Section~\ref{sec:ODE}}

Recall the notation from \ref{sec:ODE}, in particular the spaces
$C^0_b(E,E)$ and $C^1([0,+\infty),E)$. In the following proofs, we
will also use (with transparent notation) spaces of the form
$C^0([0,+\infty), E)$ or $C^1([0,T],E)$ for some finite $T$.

\begin{proof}[Proof of
  Proposition~\ref{prop:first-order-perturbation}]
  The proof is standard, but we include it for the sake of completeness.
Let $x$ be the solution of the differential equation~\eqref{eq:abstract-ODE}. 
We introduce the set of times of uniform convergence,
\[
\mathscr{C} = \{ T \ge 0 : x_n \to x \ \text{uniformly on } [0,T] \text{ as } n \to \infty \}.
\]
We aim to prove that $\mathscr{C} = [0,+\infty)$. 
To this end, we show that $\mathscr{C} \subset [0,+\infty)$ is a non-empty, open, and closed subset, which implies, by connectedness of $[0,+\infty)$, that $\mathscr{C} = [0,+\infty)$.
Clearly, $T = 0$ belongs to $\mathscr{C}$ because of the assumption $x_{0,n} \to x_0$, hence non-emptiness.

To show that $\mathscr{C}$ is both open and closed, it suffices to prove that for any nondecreasing sequence $T_m \to T_0$, with $T_m \in \mathscr{C}$ for all $m \ge 1$, there exists $\delta > 0$ such that $T_0 + \delta \in \mathscr{C}$.
Since $g$ is assumed to be complete, $x(t)$ is well defined for all $t \ge 0$.
In particular, we can define open balls $B_r = B(x(T_0), r)$ for $r > 0$ and consider
\[
M = \sup\{\|g_n(v)\|_E : v \in B_2, \ n \ge 1\},
\]
which is finite because $g_n \to g$ uniformly on $B_2$. By continuity
of $x$ and by the definition of $\mathscr{C}$, we may fix
$m,n_0 \ge 1$ such that $T_m > T_0 - 1/(2M)$, $x(T_m) \in B_1$, and,
for all $n \ge n_0$, $x_n(T_m) \in B_1$. Because all solutions have
derivative bounded by $M$ as long as they stay in $B_2$, they cannot
exit $B_2$ starting from $B_1$ in a time shorter than $1/M$.
Therefore, for all $s \in [0,1/M]$ and $n \ge n_0$, we have
$x_n(T_m+s), x(T_m+s) \in B_2$. Because $g$ is complete, it is
locally Lipschitz continuous, and because $K=\{x(u) : u \in [T_m,
T_0+1/M]\}$ is necessarily compact, we have
\[
  C := \sup_{y\in B_2\setminus K} \sup_{u \in [T_m, T_0+1/M]}
  \frac{\norm{g(y)-g(x(u))}_E}{\norm{y-x(u)}_E} < \infty.
\]
We now compute
\begin{align*}
  &\|x_n(T_m+s) - x(T_m+s)\|_E - \|x_n(T_m) - x(T_m)\|_E \\
  &\le \int_{T_m}^{T_m+s} \|g_n(x_n(u)) - g(x(u))\|_E \, \mathrm{d}u \\
  &\le \int_{T_m}^{T_m+s} \big( \|g_n(x_n(u)) - g(x_n(u))\|_E + \|g(x_n(u)) - g(x(u))\|_E \big) \, \mathrm{d}u \\
  &\le s \sup_{v \in B_2} \|g_n(v) - g(v)\|_E + C \int_{T_m}^{T_m+s} \|x_n(u) - x(u)\|_E \, \mathrm{d}u.
\end{align*}
By Grönwall’s lemma, we deduce
\begin{equation*}
   \|x_n(T_m+s) - x(T_m+s)\|_E \leq \Big( \|x_n(T_m) - x(T_m)\|_E 
   + s \sup_{v \in B_2} \|g_n(v) - g(v)\|_E \Big)e^{sC}
\end{equation*}
which tends to $0$ uniformly in $s \in [0,1/M]$.
Because we chose $T_m > T_0 - 1/(2M)$, this shows that $T_0 + 1/(2M) \in \mathscr{C}$, concluding that $\mathscr{C} = [0,+\infty)$.

So far, we have proven the uniform convergence $x_n \to x$ on compact sets. 
Since $\dot{x}_n = g_n \circ x_n$ and $g_n \to g$ uniformly on bounded sets, we readily deduce the uniform convergence 
$\dot{x}_n \to g \circ x = \dot{x}$ on compact sets, hence $x_n \to x$ in $C^1([0,+\infty), E)$.
\end{proof}

\begin{proof}[Proof of
  Proposition~\ref{prop:second-order-perturbation}]
We will use Grönwall's lemma to show that
$\dot{y}_n\to \dot{y}$. By definition, we can write, for any $t\geq 0$,
\begin{align}
  \dot{y}_n(t) - \dot{y}(t)
  &= \sqrt{n} \big(g_n(x_n(t)) - g(x(t)) \big) -
    \mathrm{d}_{x(t)} g (y(t)) - h(x(t)) \notag \\
  &= \Big[\sqrt{n} \big(g_n(x_n(t)) - g(x_n(t)) \big) - h(x_n(t))
    \Big] + \big[h(x_n(t)) - h(x(t)) \big]  \label{eq:y-prime-gronwall-1} \\
  & \quad + \Big[\sqrt{n} \big(g(x_n(t)) - g(x(t)) \big)
    - \mathrm{d}_{x(t)}g(y_n(t)) \Big]
    + \big[  \mathrm{d}_{x(t)}g(y_n(t) - y(t)) \big]. \notag
\end{align}
Let us fix $T>0$. We now show that the
$\norm{\cdot}_{\mathcal{E}}$-norms of each of the four terms in the
last expression are all of the form $o(1)+O(\norm{y_n(t) - y(t)}_{\mathcal{E}})$,
uniformly in $t\in [0,T]$, as $n\to \infty$. For the first term, this
is a straightforward consequence of the fact that
\[
  \sqrt{n} (g_n - g) \longrightarrow h \qquad \text{in }
  C^0_b(E, \mathcal{E})
\]
and that by Proposition~\ref{prop:first-order-perturbation}, $x_n(t)$
remain in a bounded set (a compact set in fact) as $t\in [0,T]$ and
$n$ varies. For the second term, it suffices to note that $h$ is
assumed to be continuous. The third term is more complex. Noting that
$K=\{x(t) : t\in [0,T]\}$ is a compact space, continuity of
$x\mapsto \mathrm{d}_xg$ implies the existence of a common modulus of
continuity on $K$, i.e.\ for each $\epsilon >0$, there exists
$\delta>0$ such that for all $t\in [0,T]$,
$\norm{x- x(t)}_{E} < \delta \implies \norm{\mathrm{d}_xg -
  \mathrm{d}_{x(t)}g}_{\mathrm{op}} < \epsilon$. Let us fix an
arbitrary $\epsilon>0$ and a suitable $\delta$ as above. The mean
value inequality applied to the function
$z\mapsto g(z) - g(x(t)) - \mathrm{d}_{x(t)}g(z-x(t))$ implies that
there exists $z\in [x_n(t), x(t)]$ such that
\[
  \norm[\Big]{\sqrt{n} \big(g(x_n(t)) - g(x(t)) \big) -
    \mathrm{d}_{x(t)}g(y_n(t))}_{\mathcal{E}} \leq
  \norm{y_n(t)}_{\mathcal{E}} \cdot \norm{\mathrm{d}_zg -
  \mathrm{d}_{x(t)}g}_{\mathrm{op}}.
\]
Therefore, for $n$ large enough so that
$\sup_{t\in [0,T]} \norm{x_n(t) - x(t)} < \delta$, this is bounded by
\[
  \epsilon \norm{y_n(t)}_{\mathcal{E}} \leq \epsilon
  \sup_{s\in [0,T]}\norm{y(s)}_{\mathcal{E}} + \epsilon \norm{y_n(t) - y(t)}_{\mathcal{E}},
\]
where the supremum is finite since $y$ is continuous. Finally, the
fourth term in~\eqref{eq:y-prime-gronwall-1} is simply bounded by
\[
  \sup_{s\in [0,T]} \norm{\mathrm{d}_{x(s)}g}_{\mathrm{op}} \cdot \norm{y_n(t) - y(t)}_{\mathcal{E}},
\]
where the left-hand term is finite since $K$ is compact. In the end,
we have shown that there exists a constant $C>0$ such that for all
$t\in [0,T]$,
\[
  \norm{\dot{y}_n(t) - \dot{y}(t)}_{\mathcal{E}} \leq o(1) + C \norm{y_n(t) - y(t)}_{\mathcal{E}},
\]
where the $o(1)$ term does not depend on $t$. An application of
Grönwall's lemma yields
$\sup_{t\in [0,T]}\norm{\dot{y}_n(t) - \dot{y}(t)}_{\mathcal{E}} \to
0$, and since $y_n(0)$ is assumed to converge to $y_0$, the
convergence
$\sup_{t\in [0,T]}\norm{y_n(t) - y(t)}_{\mathcal{E}} \to
0$ readily follows, which concludes the proof.
\end{proof}

\begin{proof}[Proof of Corollary~\ref{cor:stochastic-perturbation}] The law of large numbers stated in point~$(i)$ is a direct application of Proposition~\ref{prop:first-order-perturbation}. Let $\Omega'\subset\Omega$ denote the subset on which the convergence $(X_{0,n},G_n)\stackrel{a.s.}\to (x_0,g)$ holds; by assumption, $\mathbb{P}(\Omega')=1$. For each fixed $\omega\in\Omega'$, we can apply the deterministic statement of Proposition~\ref{prop:first-order-perturbation} and deduce the convergence $X_n(\omega)\to x$. This holds for all $\omega\in\Omega'$, which implies almost sure convergence.

  The functional central limit theorem stated in point~$(ii)$ is
  proved similarly by a direct application of
  Proposition~\ref{prop:second-order-perturbation}. We appeal to a
  Skorokhod representation in the completely metrizable space
  $\mathcal{E} \times C_b^0(E,\mathcal{E})$ to represent the
  convergence in distribution
\[
\big(\sqrt{n}(X_{0,n}-x_0),\,\sqrt{n}(G_n-g)\big)
\xrightarrow{d} (\mathscr{X}_0,\mathscr{G})
\]
by an almost sure convergence on a suitable probability space. Note
that since the space $\mathcal{E} \times C_b^0(E,\mathcal{E})$ is not
a Polish space, we are not in the classic case of Skorokhod's
representation theorem. We use Theorem~1.10.3 from \citet{vW96}, where
the underlying space is not required to be Polish (but only a metric
space) under the additional assumption that the limiting random
variable is separable. We can apply the result since
$(\mathscr{X}_0,\mathscr{G})$ is tight: indeed on a complete metric
space, tightness is equivalent to separability, see Section~1.3
of~\citealt{vW96}.

Under the Skorokhod representation, for each fixed $\omega$ in a
subset of probability~$1$, we apply
Proposition~\ref{prop:second-order-perturbation} to deduce the almost
sure convergence
\[
\mathscr{X}_n := \sqrt{n}(X_n-x) \xrightarrow{a.s.} \mathscr{X},
\]
which implies convergence in distribution on the original probability space. Finally, the fact that the limit $\mathscr{X}$ is Gaussian follows directly from the assumption that $(\mathscr{X}_0,\mathscr{G})$ is Gaussian and from the linearity of the differential equation~\eqref{eq:gaussian-ODE}.
\end{proof}

\subsection{Proofs related to Section~\ref{sec:CLT-KGF}}
\begin{proof}[Proof of Lemma~\ref{lem:gradient-field-regularity}] 
We first establish that $\mathcal{L}_n$ and $\mathcal{L}$ defined in 
Equations~\eqref{eq:empirical-risk} and~\eqref{eq:population-risk} are twice 
continuously differentiable on $\mathcal{H}$. By the reproducing kernel 
property, the gradient of the evaluation map $F \mapsto F(x)$, $x\in\mathcal{X}$, 
is given by the feature map $\Phi(x) = k(x,\cdot) \in \mathcal{H}$, since 
$F(x) = \langle \Phi(x), F \rangle_{\mathcal{H}}$.  
As a consequence, and by Assumption~\ref{ass:loss}, for each fixed $(x,y)$ the map 
$F \mapsto \ell(F(x),y)$ is twice continuously differentiable with gradient
\[
\nabla_F\,\ell(F(x),y) = \partial_1\ell(F(x),y)\,\Phi(x)
\]
and Hessian
\[
\nabla_F^2\,\ell(F(x),y)
= \partial_1^2\ell(F(x),y)\,\Phi(x)\otimes \Phi(x),
\]
where $\Phi(x)\otimes \Phi(x)\colon \mathcal{H}\to\mathcal{H}$ is the rank-one 
operator
\[
G \mapsto \langle \Phi(x), G \rangle_{\mathcal{H}}\,\Phi(x) = G(x)\,\Phi(x).
\]

The operator norm of the Hessian at $F$ is bounded by
\[
\|\partial_1^2\ell(F(x),y)\,\Phi(x)\otimes \Phi(x)\|_{\mathrm{op}}
\leq L\,\|\Phi(x)\otimes \Phi(x)\|_{\mathrm{op}}
= L\,\|\Phi(x)\|_{\mathcal{H}}^2
= L\,\kappa,
\]
where $L$ is the uniform bound on $\partial_1^2\ell$ from Assumption~\ref{ass:loss} and $\kappa$ the uniform bound on $\|\Phi(x)\|_{\mathcal{H}}=\sqrt{k(x,x)}$ from Assumption~\ref{ass:kernel}.
Since this bound is 
independent of $F$ and integrable with respect to both $\mathrm{P}_n$ and 
$\mathrm{P}$, we deduce that $\mathcal{L}_n$ and $\mathcal{L}$ are twice 
continuously differentiable with
\[
\nabla_F\mathcal{L}_n= \mathrm{P}_n[\partial_1\ell(F(x),y)\,\Phi(x)], 
\qquad
\nabla^2_F\mathcal{L}_n= \mathrm{P}_n[\partial_1^2\ell(F(x),y)\,\Phi(x)\otimes \Phi(x)],
\]
and similarly for $\mathcal{L}$ with $\mathrm{P}_n$ replaced by $\mathrm{P}$.

\textit{Proof of $(i)$.} 
This follows directly from the fact that $\mathcal{L}_n$ and $\mathcal{L}$ are twice continuously differentiable on $\mathcal{H}$ with Hessians uniformly bounded in operator norm: 
\[
  \|\nabla^2_F\mathcal{L}_n\|_{\mathrm{op}} \leq L\,\mathrm{P}_n[k(x,x)], \quad \|\nabla^2_F\mathcal{L}\|_{\mathrm{op}} \leq L\,\mathrm{P}[k(x,x)],
\]
for all $F\in\mathcal{H}$.

\textit{Proof of $(ii)$.} 
Point~$(i)$ implies that the gradient vector fields 
$(\nabla_F \mathcal{L}_n)_{F\in\mathcal{H}}$ and $(\nabla_F \mathcal{L})_{F\in\mathcal{H}}$ 
are complete, hence the flows $(F_t^n)_{t\ge 0}$ and $(F_t)_{t\ge 0}$ are well defined 
for all $t\ge 0$. It remains to check that, for $n\ge 1$, $F^n$ depends continuously on the training sample. Let $(x_i,y_i)_{1\le i\le n}$ and $(x_i',y_i')_{1\le i\le n}$ be two training samples, 
and let $\mathrm{P}_n$ and $\mathrm{P}_n'$ denote the corresponding empirical measures, $\mathcal{L}_n$ and $\mathcal{L}_n'$ the corresponding empirical risks, and $(F_t^n)_{t\geq 0}$ and $({F_t^n}')_{t\geq 0}$ the corresponding gradient flows.
Then, for any $F\in\mathcal{H}$,
\begin{align*}
\|\nabla_F \mathcal{L}_n - \nabla_F \mathcal{L}_n'\|_{\mathcal{H}}
&= \big\|\mathrm{P}_n[\partial_1\ell(F(x),y)\,\Phi(x)] 
      - \mathrm{P}_n'[\partial_1\ell(F(x'),y')\,\Phi(x')]\big\|_{\mathcal{H}} \\
&\leq \frac{1}{n}\sum_{i=1}^n 
\big\|\partial_1\ell(F(x_i),y_i)\,\Phi(x_i) 
- \partial_1\ell(F(x_i'),y_i')\,\Phi(x_i')\big\|_{\mathcal{H}}.
\end{align*}
Let $M>0$ be fixed; Equation~\ref{eq:bound-uniform-continuity} implies that
\[
  \sup_{\norm{F}_{\mathcal{H}}\leq M}\max_{1\leq i\leq n}\norm{F(x_i) - F(x_i')}_{\mathcal{H}} \leq M \max_{1\leq i\leq n} (k(x_i,x_i) + k(x_i', x_i') - 2k(x_i, x_i'))
\]
which tends to $0$ as $(x_i')_{i=1}^n \to (x_i)_{i=1}^n$ by continuity of $k$. Since $\partial_1 \ell$ is supposed uniformly continuous by Assumption~\ref{ass:loss}, it follows
\[
  \sup_{\norm{F}_{\mathcal{H}}\leq M}\max_{1\leq i \leq n}\abs[\big]{\partial_1\ell(F(x_i),y_i) - \partial_1\ell(F(x_i'),y_i')} \to 0,
\]
as $(x_i')_{i=1}^n \to (x_i)_{i=1}^n$. Likewise, \[
  \max_{1\leq i\leq n}\norm{\Phi(x_i) - \Phi(x_i')}_{\mathcal{H}} = \max_{1\leq i\leq n} (k(x_i,x_i) + k(x_i', x_i') - 2k(x_i, x_i'))\to 0.
\]
The last two computations therefore show that
\[
  \sup_{\norm{F}_{\mathcal{H}}\leq M} \|\nabla_F \mathcal{L}_n - \nabla_F \mathcal{L}_n'\|_{\mathcal{H}} \to 0,
\]
and Proposition~\ref{prop:first-order-perturbation} implies that $({F^n_t}')_{t\geq 0} \to (F^n_t)_{t\geq 0}$ in $C^1([0, +\infty), \mathcal{H})$, which concludes the proof.
\end{proof}

Before proving Theorem~\ref{thm:CLT-gradient-flow}, we will
state a technical result that will be needed to build random variables
$\mathcal{G}\in C^0_b(\mathcal{H}, \mathcal{H})$ from our limiting
tight Gaussian processes
$\mathbb{G} = \lim_{n\to \infty} \mathbb{G}_n \in
\ell^\infty(\mathcal{C}_M)$, for $\mathcal{C}_M$ the class of
functions defined in Lemma~\ref{lem:Donsker}, for some $M>0$. Let us
first introduce some notation. For $F,G\in \mathcal{H}$, let us write
\[
  \phi_{F,G}:(x,y)\in \mathcal{X}\times \mathcal{Y} \;\longmapsto \; \partial_1\ell(F(x),y)G(x),
\]
and denote the $\mathcal{H}$-ball of radius $M$ by
$B_M=\{F\in \mathcal{H} : \norm{F}_{\mathcal{H}} \leq M\}$, so that we
can write $\mathcal{C}_M = \{\phi_{F,G} : F\in B_M, G\in B_1\}$. For
each $F\in B_M$ and $\mathbb{A}\in \ell^\infty(\mathcal{C}_M)$, we
define the partial mappings:
\[
  \widetilde{\mathbb{A}}:
  \begin{cases}
    B_M &\longrightarrow \ell^\infty(B_1) \\
    F &\longmapsto \mathbb{A}_F,
  \end{cases}
  \quad \text{ where }
  \mathbb{A}_F :
  \begin{cases}
    B_1 &\longrightarrow \mathbb{R} \\
    G &\longmapsto \mathbb{A}(\phi_{F,G}).
  \end{cases}
\]
Note that $\ell^\infty(B_1)$, equipped with the uniform norm, is a
(non-separable) Banach space, and similarly $\widetilde{\mathbb{A}}$
lives in the Banach space $\ell^\infty(B_M,\ell^\infty(B_1))$,
equipped with the uniform norm
\begin{equation}
  \label{eq:norm-A-tilde}
  \norm{\widetilde{\mathbb{A}}} := \sup_{F\in B_M} \sup_{G\in B_1} \abs{\widetilde{\mathbb{A}}(F)(G)}.
\end{equation}
We finally define some subspaces
$\mathscr{L}_M,\mathscr{U}_M \subset \ell^\infty(\mathcal{C}_M)$ as
follows:
\begin{align*}
  \mathscr{U}_M & = \{\mathbb{A}\in \ell^\infty(\mathcal{C}_M) :
  \widetilde{\mathbb{A}} \text{ is continuous}\}, \\
  \text{and}\quad
  \mathscr{L}_M & = \{\mathbb{A}\in \ell^\infty(\mathcal{C}_M) :
  \text{for all }
  F\in B_M, \mathbb{A}_F \text{ is linear}\},
\end{align*}
where we abuse the definition of being linear; we should say instead:
$G\mapsto \mathbb{A}(\phi_{F,G})$ extends uniquely to a (necessarily
continuous) linear function defined on $\mathcal{H}$. Note also that
the mappings $\mathbb{A}_F$ are necessarily bounded, so if
$\mathbb{A}\in \mathscr{L}_M$, the $\mathbb{A}_F$ are linear bounded,
so they are continuous.

\begin{remark}
  \label{rmk:G_n-continuous}
  Since Assumption~\ref{ass:loss} ensures that
  $z\mapsto \partial_1\ell(z,y)$ is $L$-Lipschitz, and that
  Assumption~\ref{ass:kernel} ensures that elements of $\mathcal{H}$
  are continuous and bounded, with a bound given
  by~\eqref{eq:bound-uniform-norm}, it is readily checked that for all
  $n\in \mathbb{N}$, and all $M>0$,
  \[
    \big(\mathbb{G}_n(\phi)\big)_{\phi\in \mathcal{C}_M} =
    \big(\sqrt{n}(\mathrm{P}_n - \mathrm{P})(\phi)\big)_{\phi \in \mathcal{C}_M} \in
    \mathscr{U}_M \cap \mathscr{L}_M.
  \]
\end{remark}

\begin{lemma} \label{lem:bbG-to-linear-app}
  The subspaces $\mathscr{U}_M$ and $\mathscr{L}_M$ are
  closed subsets of $\ell^\infty(\mathcal{C}_M)$. There is a unique
  continuous map
  \[
    \Theta_M :
    \begin{cases}
      \mathscr{U}_M\cap \mathscr{L}_M &\longrightarrow C^0_b(B_M, \mathcal{H}) \\
      \mathbb{A} &\longmapsto \Theta_M(\mathbb{A}) = \mathbb{A}'
    \end{cases}
  \]
  satisfying
  \[
    \mathbb{A}(\phi_{F,G}) = \langle \mathbb{A}'(F), G
    \rangle_{\mathcal{H}}, \qquad F\in B_M, G\in B_1.
  \]
\end{lemma}

\begin{proof}
  Linearity can be readily expressed as an intersection of relations
  of the form
  \[
    \mathbb{A}(\phi_{F,aG_1+G_2}) = a\mathbb{A}(\phi_{F,G_1}) + \mathbb{A}(\phi_{F,G_2}),
  \]
  for all suitable values $F,a,G_1,G_2$. These relations are clearly
  closed under the topology of $\ell^\infty(\mathcal{C}_M)$, so this
  shows that $\mathscr{L}_M$ is closed. We now turn to
  $\mathscr{U}_M$. Assume that $\mathbb{A}_n \to \mathbb{A}$ in
  $\ell^\infty(\mathcal{C}_M)$, with $\mathbb{A}_n \in \mathscr{U}_M$.
  Using~\eqref{eq:norm-A-tilde} and the convergence
  $\mathbb{A}_n \to \mathbb{A}$, we have
  \[
    \widetilde{\mathbb{A}}_n \to \widetilde{\mathbb{A}}
  \]
  uniformly, as functions defined on $B_M$. Since the
  $\widetilde{\mathbb{A}}_n$ are continuous, the limiting map is also
  continuous, i.e.\ $\mathbb{A}\in \mathscr{U}_M$. This proves that
  $\mathscr{U}_M$ is closed.

  We now assume that $\mathbb{A}\in \mathscr{U}_M\cap \mathscr{L}_M$,
  and take $F\in B_M$. Since $\mathbb{A}_F:\mathcal{H}\to \mathbb{R}$
  is linear bounded, $\mathbb{A}'(F) \in \mathcal{H}$ is uniquely
  defined as the Riesz representation of $\mathbb{A}_F$. The
  continuity and boundedness of $\mathbb{A}'$ is a simple consequence
  of the fact that $\mathbb{A}\in \mathscr{U}_M$. It remains to show
  that $\Theta_M$ itself is continuous. Note that it is easily checked
  that $\mathscr{U}_M$ and $\mathscr{L}_M$ are linear subspaces, and
  that $\Theta_M$ is a linear map. Therefore it suffices to show that
  it is bounded. Consider $\mathbb{A}$ in
  $\mathscr{U}_M\cap \mathscr{L}_M$ with
  $\norm{\mathbb{A}}_{\ell^\infty(\mathcal{C}_M)} \leq 1$. By
  definition and using the self-duality of $\mathcal{H}$, we have
  \[
    \norm[\big]{\Theta_M(\mathbb{A})}_{C^0_b(B_M, \mathcal{H})} =
    \sup_{F\in B_M} \sup_{G\in B_1} \abs{\langle \mathbb{A}'(F), G
      \rangle_{\mathcal{H}}} = \sup_{F\in B_M,G\in B_1}
    \abs{\mathbb{A}(\phi_{F,G})} \leq 1,
  \]
  so the proof is complete.
\end{proof}

\begin{proof}[Proof of Theorem~\ref{thm:CLT-gradient-flow}] 
The theorem is an application of Corollary~\ref{cor:stochastic-perturbation}, and we need to check the assumptions of this corollary. 
Here the ambient space is $E=\mathcal{E}=\mathcal{H}$, the vector fields are $G_n=\nabla \mathcal{L}_n$ and $g=\nabla \mathcal{L}$, and the initial positions are $X_{0,n}=x_0=0$.

\textit{Proof of $(i)$.} 
We must check that
\[
\nabla \mathcal{L}_n \xrightarrow{a.s.}\nabla \mathcal{L} 
\quad \text{in } C_b^0(\mathcal{H},\mathcal{H}).
\]
By Lemma~\ref{lem:gradient-field-regularity}, the vector fields belong to $C_b^0(\mathcal{H},\mathcal{H})$, and it remains to prove that
\begin{equation}\label{eq:uniform-convergence}
\sup_{\|F\|_\mathcal{H}\leq M} 
\big\|\nabla_F \mathcal{L}_n - \nabla_F \mathcal{L}\big\|_\mathcal{H}
\xrightarrow{a.s.} 0
\end{equation}
for all $M\ge 0$. Note that the supremum is measurable because $\mathcal{H}$ is separable and the vector fields are continuous.

To prove~\eqref{eq:uniform-convergence}, observe that the left-hand side may be written as 
\begin{equation*} 
  \sup_{\phi \in \mathcal{C}_M} \big|\,\mathrm{P}_n[\phi(x,y)] - \mathrm{P}[\phi(x,y)]\,\big|,
\end{equation*}
where $\mathcal{C}_M$ is the class of functions introduced in Lemma~\ref{lem:Donsker}.
Since this class is Glivenko--Cantelli, the empirical process 
$(\mathrm{P}_n(\phi))_{\phi\in\mathcal{C}_M}$, converges $*$-a.s.\ to 
$(\mathrm{P}(\phi))_{\phi\in\mathcal{C}_M}$ in $\ell^\infty(\mathcal{C}_M)$. 
Hence the supremum above converges to $0$ $*$-a.s., proving 
\eqref{eq:uniform-convergence} and completing the proof of~$(i)$.

\textit{Proof of $(ii)$.}
We must check that $\nabla \mathcal{L} \in C^1(\mathcal{H},\mathcal{H})$ 
(which follows from Lemma~\ref{lem:gradient-field-regularity}), and that
\begin{equation}\label{eq:weak-convergence}
\sqrt{n}\big(\nabla \mathcal{L}_n - \nabla \mathcal{L}\big)
\xrightarrow{d} \mathcal{G}
\quad \text{in } C_b^0(\mathcal{H},\mathcal{H}),
\end{equation}
for a tight centered Gaussian limit $\mathcal{G}$. It is sufficient to
show that for any $M>0$, this convergence in distribution holds in
$C_b^0(\mathcal{H},\mathcal{H})$. For $\|F\|_{\mathcal{H}}\le M$ and
$\|G\|_{\mathcal{H}}\le 1$, we have
\begin{equation}
  \label{eq:delta-L-empirical-process}
  \big\langle \sqrt{n}\big(\nabla_F \mathcal{L}_n - \nabla_F
  \mathcal{L}\big), G \big\rangle_{\mathcal{H}}
  = \sqrt{n}\,(\mathrm{P}_n - \mathrm{P})(\phi_{F,G}),
\end{equation}
where 
\[
  \phi_{F,G}(x,y) = \partial_1\ell(F(x),y)\,G(x)
\]
belongs to the class $\mathcal{C}_M$. By Lemma~\ref{lem:Donsker}, the
class $\mathcal{C}_M$ is Donsker, so the empirical process
$(\sqrt{n}(\mathrm{P}_n-\mathrm{P})(\phi))_{\phi\in\mathcal{C}_M}$,
converges in distribution to the tight centered Gaussian limit
$(\mathbb{G}_{M}(\phi))_{\phi\in\mathcal{C}_M}$ in
$\ell^\infty(\mathcal{C}_M)$, with covariance structure given by
\[
  \mathrm{Cov}(\mathbb{G}_{M}(\phi), \mathbb{G}_{M}(\phi')) = \mathrm{P}[\phi(x,y)\phi'(x,y)] - \mathrm{P}[\phi(x,y)]\mathrm{P}[\phi'(x,y)],
\]
see~\citet[Section~19.2]{vdV98}. By Remark~\ref{rmk:G_n-continuous}
and Lemma~\ref{lem:bbG-to-linear-app}, using the notation of the
lemma, the support of $\mathbb{G}_{M}$ must be in the closed subspace
$\mathscr{U}_M \cap \mathscr{L}_M\subset \ell^\infty(\mathcal{C}_M)$.
Consequently $\mathcal{G}_M := \Theta_M(\mathbb{G}_{M})$ is a tight
Gaussian random variable with values in $C^0_b(B_M, \mathcal{H})$
which satisfies
\[
  \langle \mathcal{G}_M(F), G\rangle_{\mathcal{H}} =
  \mathbb{G}_{M}(\phi_{F,G}).
\]
The equality~\eqref{eq:delta-L-empirical-process} and the convergence
of empirical processes then ensure the weak
convergence~\eqref{eq:weak-convergence}. Furthermore, we get the
following covariance structure for $\mathcal{G}$:
\begin{align*}
  \mathrm{Cov}(\langle\mathcal{G}(F), G\rangle_{\mathcal{H}}, \langle\mathcal{G}(F'), G'\rangle_{\mathcal{H}}) 
  &= \mathrm{Cov}(\mathbb{G}(\phi_{F,G}), \mathbb{G}(\phi_{F',G'}))\\
  &= \mathrm{P}[\partial_1\ell(F(x),y)\partial_1\ell(F'(x),y)G(x)G'(x)] \\
  & \quad - \mathrm{P}[\partial_1\ell(F(x),y)G(x)]\mathrm{P}[\partial_1\ell(F'(x),y)G'(x)].
\end{align*}
In other words, the Gaussian family $(\mathcal{G}(F))_{F\in \mathcal{H}}$ has the same covariance structure as $(\partial_1\ell(F(x),y)\Phi(x))_{F\in \mathcal{H}}$ under the probability measure $\P$.

Equation~\eqref{eq:limit-ODE}, with $\mathscr{G}_t := \mathcal{G}(F_t)$, can be deduced from~\eqref{eq:gaussian-ODE}, since the  operator $A_t$ defined in~\eqref{eq:At} is precisely the Fréchet derivative at $F_t$ of $\nabla\mathcal{L}$, that is $\nabla^2_{F_t}\mathcal{L}$.
Furthermore, the computations above show that the covariance structure of the process $(\mathscr{G}_t)_{t\geq 0}$ is the same as that of $(\partial_1\ell(F_t(x),y)\Phi(x))_{t\geq 0}$,
which concludes the proof.
\end{proof}

\begin{proof}[Proof of Theorem~\ref{thm:KGF-linear-case}]
  The proof of $(i)$ is a consequence of the law of large numbers
  (LLN) and the central limit theorem (CLT) for square-integrable
  random Hilbert-valued random variables, see Theorem~7.9 combined
  with Theorem~10.5 in \citet{LT91}. Indeed, the space
  $(\mathcal{H} \otimes \mathcal{H})\times \mathcal{H}$ is a separable
  Hilbert space and the pair $(B_n, C_n)$, defined
  in~\eqref{eq:B_n-C_n-def}, is simply the empirical mean of i.i.d.\
  random variables
  $(\Phi(x_i)\otimes \Phi(x_i), y_i\Phi(x_i))_{i=1}^n$. Furthermore, by Assumption~\ref{ass:kernel},
  \[
  \norm{\Phi(x)\otimes\Phi(x)}_{\mathrm{HS}} =
  \norm{\Phi(x)}_{\mathcal{H}}^2 =k(x,x)\leq \kappa^2,
  \]
  and we have also
  \[
    \mathrm{P}[\norm{y\Phi(x)}_{\mathcal{H}}^2] =
    \mathrm{P}[y^2k(x,x)] \leq \kappa^2 \mathrm{P}[y^2],
  \]
  which is finite by Assumption~\ref{ass:loss} -- note that, for the square loss $\ell(z,y)=\frac{1}{2}(z-y)^2$, we have $\partial_1 \ell(0,y)=-y$. This implies that the
  $(\Phi(x_i)\otimes \Phi(x_i), y_i\Phi(x_i))_{i\geq 1}$ are i.i.d.\
  and square-integrable random variables with values in a separable
  Hilbert space, so the LLN and CLT apply and we get directly $(i)$.

  In order to prove $(ii)$, we use the explicit form
  \[
    F^n_t = t \widetilde{\mathrm{exp}}(-tB_n) C_n = t \sum_{k\geq 0} \frac{(-tB_n)^k}{(k+1)!}C_n, \qquad t\geq 0,
  \]
  given in~\eqref{eq:least-squares-F_n-explicit}. Using Skorokhod's
  representation theorem, let us assume that the CLT convergence from
  $(i)$ holds almost surely. It is not too hard to see that the
  mapping
  $\widetilde{\mathrm{exp}}:\mathcal{H}\to \mathcal{H}, A\mapsto
  \sum_{k=0}^{\infty}\frac{A^k}{(k+1)!}$ is Fréchet differentiable
  with a derivative at $A$ that we denote by
  $d_A \widetilde{\mathrm{exp}}(\cdot)$ and that for all $T>0$, it
  holds a.s.\ uniformly in $t\in [0,T]$ that
  \[
    \sqrt{n}(F^n_t - F_t) \;\underset{n\to\infty}{\longrightarrow} \;
    \mathscr{F}_t \;:=\; t \widetilde{\mathrm{exp}}(-tB) \mathscr{C} -
    t^2 d_{(-tB)}\widetilde{\mathrm{exp}}(\mathscr{B})C.
  \]
  We now leverage this convergence in $C^0([0,\infty), \mathcal{H})$ and
  the fact that $\frac{d}{dt}F^n_t = -B_n F^n_t + C_n$ and $\frac{d}{dt}F_t = -B F_t + C$
  to deduce that
  similarly:
  \[
    \frac{d}{dt} \sqrt{n}(F^n_t - F_t) \;\underset{n\to\infty}{\longrightarrow} \;
    -B \mathscr{F}_t -\mathscr{B} F_t
    + \mathscr{C},
  \]
  uniformly for $t$ in compact sets. Therefore the convergence
  $\sqrt{n}(F^n - F) \to \mathscr{F}$ holds in
  $C^1([0,\infty), \mathcal{H})$ and we can identify the derivative
  $\frac{d}{dt} \mathscr{F}_t$ as the right-hand side above. This
  concludes the proof of $(ii)$.
\end{proof}

\subsection{Proofs related to Section~\ref{sec:CLT-IGB}}

\subsubsection{Preliminaries: splitting schemes and partitions}\label{sec:splitting-schemes}
We recall the formal definition of splitting schemes introduced in \citet[Section~2.2]{DD24a}.

The binary rooted tree with depth $d\geq 1$ (from graph theory) is defined on the vertex set $\mathscr{T}_d=\cup_{l=0}^d \{0,1\}^l$. A vertex $v\in \{0,1\}^l$ is   seen as a word of size $l$ in the letters $0$ and $1$. The empty word  $v=\emptyset$ corresponds to the tree root. The vertex set is divided into the internal nodes $v\in \mathscr{T}_{d-1}$ and the terminal nodes $v\in\{0,1\}^d$, also called leaves. Each internal node $v$ has two children denoted $v0$ and $v1$ (concatenation of words) while the terminal nodes have no offspring. 

A regression tree is encoded by a \textit{splitting scheme} 
\[
\xi=(j_v,u_v)_{v\in\mathscr{T}_{d-1}}\in ([\![1,p]\!]\times (0,1))^{\mathscr{T}_{d-1}}
\]
giving the splits at each internal node, and its leaf values $(\bar r_v)_{v\in\{0,1\}^d}$. The splitting scheme $\xi$ allows to associate to each vertex $v\in\mathscr{T}_{d}$ a region $A_v=A_v^\xi$ defined recursively by $A_\emptyset=[0,1]^p$ and, for $v\in\mathscr{T}_{d-1}$, 
\begin{equation} \label{eq:split-A0-A1}
  \begin{aligned}
    &A_{v0} = A_v \cap \{x\in[0,1]^p\ :\ x^{j_{v}} < a_v+u_v(b_v-a_v)\},\\
 &A_{v1} = A_v \cap \{x\in[0,1]^p\ :\ x^{j_{v}} \geq a_v+u_v(b_v-a_v)\},
  \end{aligned}
\end{equation}
with $a_v=\inf_{x\in A_v} x^{j_v}$ and $b_v=\sup_{x\in A_v} x^{j_v}$. Note that $A_v$ depends only on the splits attached to the ancestors of $v$. For each level $l=0,\ldots,d$,  $(A_v)_{v\in \{0,1\}^l}$ is a partition   of $[0,1]^p$ into $2^l$ hypercubes. The leaf values are then
\[
\bar r_n(A_v)=\frac{1}{n(A_v)}\sum_{i=1}^n r_i\1_{A_v}(x_i).
\]
Finally, the tree associated with the sample $(x_i,r_i)_{1\leq i\leq n}$ and  splitting scheme  $\xi=(j_v,u_v)_{v\in\mathscr{T}_{d-1}}$ is the piecewise constant function
\begin{equation}\label{eq:reg-tree}
T(x;(x_i,r_i)_{1\leq i\leq n},\xi)=\sum_{v\in\{0,1\}^d}\bar r_n(A_v)\1_{A_v}(x).
\end{equation}

For self-containedness, we provide the expression for the Radon--Nikodym derivative $\rmd Q_{n,F}/\rmd Q_0$ that was derived in \citet[Proposition~2.1]{DD24a}. The notation $\Delta_n(j,u;A)$ stands for the score of the binary split  $A=A_0\cup A_1$ resulting from the choice of covariate $j$ and threshold $u$ --- see Equation~\eqref{eq:split-A0-A1} --- and defined by  \[
  \Delta_n(j,u;A)= \frac{\P_n[\partial_1\ell(F(x),y)\1_{A_0}(x)]^2}{\P_n(A_0)}+\frac{\P_n[\partial_1\ell(F(x),y)\1_{A_1}(x)]^2}{\P_n(A_1)}.
\]

\begin{proposition}\label{prop:app-RN-splitting-scheme}
For all measurable bounded $F:[0,1]^p\to \mathbb{R}$, the splitting scheme distribution $Q_{n,F}$ is absolutely continuous with respect to $Q_0 $ with Radon-Nikodym derivative
\begin{equation}\label{eq:RN}
\frac{\rmd Q_{n,F}}{\rmd Q_0}(\xi)= 
\int \prod_{v\in \mathscr{T}_{d-1}}  \frac{\exp(\beta \Delta_n(j_v^1,u_v^1; A_v^\xi))}{K^{-1}\sum_{k=1}^K \exp(\beta \Delta_n(j_v^{k},u_v^{k}; A_v^\xi))} Q_0(\rmd \xi_2)\cdots Q_0(\rmd \xi_K),
\end{equation}
with $\xi_k=(j_v^k,u_v^k)_{v\in\mathscr{T}_{d-1}}$, $2\leq k\leq K$, and for $k=1$, we take $\xi_1=(j_v^1,u_v^1)_{v\in\mathscr{T}_{d-1}}=\xi$.
Note that the Radon-Nikodym derivative~\eqref{eq:RN} is bounded from above by $K^{2^d-1}$.
Similarly, $\frac{\rmd Q_{\infty,F}}{\rmd Q_0}(\xi)$ is obtained by replacing $\P_n$ by $\P$ in the definition of $\Delta_n$.
\end{proposition}

\subsubsection{Preliminaries: an integral representation for $\mathcal{T}_n(F)$ and $\mathcal{T}(F)$}
As a preliminary step in our analysis, we establish an integral representation for the random forest predictors defined in Equations~\eqref{eq:def-rf}, and show that they belong to the RKHS $\mathcal{H}$ with kernel $k$ defined in Equation~\eqref{eq:RKHS-kernel}.

We begin with a integral representation of general elements of the RKHS. We denote by $\mathscr{S} = ([\![1,p]\!] \times (0,1))^{\mathscr{T}_{d-1}}$ the space of splitting schemes, equipped with the probability measure $Q_0$. Recall that $Q_0$ is the reference splitting-scheme distribution in the totally random case (see Definition~\ref{def:softmax-gradient-tree}). 
Because softmax selection involves $K$ random candidate splits at each node, we introduce $K$-tuples of splitting schemes $\bar\xi = (\xi^1,\ldots,\xi^K)$  and endow $\mathscr{S}^K$ with the product measure $Q_0^K$. We denote by $\nu$  the counting measure on $\{0,1\}^d$ and consider  $L^2(Q_0^K \otimes \nu)$, the space of measurable functions
\[
\Psi : (\bar{\xi},v)\longmapsto \Psi(\bar{\xi},v), 
\qquad \bar{\xi} \in \mathscr{S}^K,\; v\in\{0,1\}^d,
\]
that are square-integrable with respect to $Q_0^K \otimes \nu$.

\begin{lemma}\label{lem:RKHS-representation}
For every $\Psi \in L^2(Q_0^K \otimes \nu)$, the function
\[
\Phi(z) = \int \Psi(\bar{\xi},v)\,\1_{A_v^{\xi^1}}(z)\,Q_0^K(\mathrm{d}\bar{\xi})\,\nu(\mathrm{d}v),
\qquad z\in[0,1]^p,
\]
belongs to $\mathcal{H}$ and satisfies
\[
\|\Phi\|_{\mathcal{H}} \le \|\Psi\|_{L^2(Q_0^K\otimes \nu)}.
\]
Moreover, the linear operator $\Psi \mapsto \Phi$ is surjective, and its restriction to
\[
\mathring{\mathcal{H}}
=
\overline{\operatorname{span}\!\left\{
(\bar{\xi},v)\mapsto \1_{A_v^{\xi^1}}(z)
:\; z\in[0,1]^p
\right\}}^{L^2(Q_0^K\otimes \nu)}
\]
defines an isomorphism $\mathring{\mathcal{H}} \to \mathcal{H}$.
\end{lemma}

\begin{proof}[Proof of Lemma~\ref{lem:RKHS-representation}]
This follows directly from \cite[Theorem~2.7.7]{HE15}, since the kernel $k$ defined by Equation~\eqref{eq:RKHS-kernel} can be written as
\begin{align*}
k(z,z') &= \int \1_{A_v^\xi}(z)\1_{A_v^\xi}(z')\,Q_0(\mathrm{d}\xi)\,\nu(\mathrm{d}v)\\
&= \int \1_{A_v^{\xi^1}}(z)\1_{A_v^{\xi^1}}(z')\,Q_0^K(\mathrm{d}\bar{\xi})\,\nu(\mathrm{d}v),
\end{align*}
and $(\bar{\xi},v)\mapsto\1_{A_v^{\xi^1}}(z)$ belongs to $L^2(Q_0^K\otimes\nu)$ for every $z\in[0,1]^p$.
\end{proof}

We now provide explicit integral representations of $\mathcal{T}_n(F)$, showing that they belong to $\mathcal{H}$, and also define the limit $\mathcal{T}(F)$ thanks to a similar integral representation. Note that the following Lemma implies in particular Lemma~\ref{lem:RKHS}.

\begin{lemma}\label{lem:RKHS-rf}
\begin{enumerate}[(i)]
    \item For any measurable $F:[0,1]^p\to\mathbb{R}$ and any $n\ge 1$, $\mathcal{T}_n(F)\in\mathcal{H}$ and can be written as
    \[
    \mathcal{T}_n(F)(z)
    =
    \int \Psi(\bar{\xi},v,F,\P_n)\,\1_{A_v^{\xi^1}}(z)\,Q_0^K(\mathrm{d}\bar{\xi})\,\nu(\mathrm{d}v),
    \]
    with
    \begin{equation}\label{eq:def-Psi}
    \Psi(\bar{\xi},v,F,\P_n)
    = \Psi_1(\xi^1,v,F,\P_n)\Psi_2(\bar{\xi},F,\P_n),
    \end{equation}
    where
    \begin{equation}\label{eq:def-Psi1}
    \Psi_1(\xi^1,v,F,\P_n)
    = \frac{\P_n[\partial_1\ell(F(x),y)\1_{A_v^{\xi^1}}(x)]}{\P_n(A_v^{\xi^1})}
    \end{equation}
    and
    \begin{equation}\label{eq:def-Psi2}
    \Psi_2(\bar{\xi},F,\P_n)
    = \prod_{w\in \mathscr{T}_{d-1}}
      \frac{\exp\!\big(\beta\,\Delta_n(j_w^{1},u_w^{1}; A_w^{\xi^1})\big)}
           {K^{-1}\sum_{k=1}^K \exp\!\big(\beta\,\Delta_n(j_w^{k},u_w^{k}; A_w^{\xi^1})\big)}.
    \end{equation}
    We recall that $(j_w^k,u_w^k)$ denotes the split at node $w$ in the splitting scheme $\xi^k$ and $\Delta_n(j,u,A)$ denotes the score of the split $(j,u)$ of $A$ as defined in Equation~\eqref{eq:score}.
    \item In the infinite-population case, for any  $F\in L^2(\P)$, we define  $\mathcal{T}(F)\in\mathcal{H}$  by the formula
    \[
    \mathcal{T}(F)(z)
    =
    \int \Psi(\bar{\xi},v,F,\P)\,\1_{A_v^{\xi^1}}(z)\,Q_0^K(\mathrm{d}\bar{\xi})\,\nu(\mathrm{d}v),
    \]
    where $\Psi(\bar{\xi},v,F,\P)\in L^2(Q_0^K\otimes\nu)$ is defined by Equations~\eqref{eq:def-Psi},~\eqref{eq:def-Psi1},~\eqref{eq:def-Psi2} with $\P_n$ replaced by $\P$. 
\end{enumerate}
\end{lemma}

\begin{proof}[Proof of Lemma~\ref{lem:RKHS-rf}]
From Equation~\eqref{eq:def-rf} and the absolute continuity of $Q_{n,F}$ with respect to $Q_0$, we have
\[
\mathcal{T}_n(z;F)
=
\int
\frac{\P_n[\partial_1\ell(F(x),y)\1_{A_v^\xi}(x)]}{\P_n(A_v^\xi)}
\,\frac{\mathrm{d}Q_{n,F}}{\mathrm{d}Q_0}(\xi)\,
\1_{A_v^\xi}(z)\,Q_0(\mathrm{d}\xi)\,\nu(\mathrm{d}v).
\]
On the other hand, the Radon-Nikodym derivative is expressed, thanks to Proposition~\ref{prop:app-RN-splitting-scheme}, as
\[
\frac{\mathrm{d}Q_{n,F}}{\mathrm{d}Q_0}(\xi)
=
\int \Psi_2((\xi,\xi^2,\ldots,\xi^K),F,\P_n)\,Q_0^{K-1}(\mathrm{d}\xi^2,\ldots,\mathrm{d}\xi^K).
\]
Setting $\xi=\xi_1$, these two Equations together imply
\[
\mathcal{T}_n(z;F)
=
\int \Psi_1(\xi^1,v,F,\P_n)\Psi_2(\bar{\xi},F,\P_n)\,\1_{A_v^{\xi_1}}(z)\,Q_0^K(\mathrm{d}\bar{\xi})\,\nu(\mathrm{d}v).
\]
For fixed $F,\P_n$, the function $(\bar\xi,v)\mapsto \Psi(\bar\xi,v,F,\P_n)$ belongs to $L^2(Q_0^K\otimes\nu)$ because $\Psi_2$ is bounded (by $K^{2^d-1}$) and $(\bar\xi,v)\mapsto \Psi_1(\bar\xi,v,F,\P_n)$ takes finitely many values. Hence Lemma~\ref{lem:RKHS-representation} implies $\mathcal{T}_n(F)\in\mathcal{H}$. The proof for the infinite-population case follows the same argument, and we use  Assumption~\ref{ass:loss} to bound $(\bar\xi,v)\mapsto\Psi_1(\xi^1,v,F,\P)$.
\end{proof}

\subsubsection{Proofs related to Subsection~\ref{subsec:IGB-background}}
We now prove Lemma~\ref{lem:gradient-field-regularity-igb} stating that the vector fields
$\mathcal{T}_n : \mathcal{H} \to \mathcal{H}$, $n\ge 1$, 
are continuously differentiable and satisfy a linear growth condition. As a consequence, these vector fields are complete, and the infinitesimal gradient boosting processes
$(F_t^n)_{t\ge 0}$, $n\ge 1$, are well defined and continuously differentiable. We further show that the finite-sample trajectory $t\mapsto F_t^n$ depends continuously on the input sample $(x_i,y_i)_{1\le i\le n}$.

\begin{proof}[Proof of Lemma~\ref{lem:gradient-field-regularity-igb}]
  \textit{Proof of (i).} Using the integral representation of
  $\mathcal{T}_n(F)$ from Lemma~\ref{lem:RKHS-rf}, it is enough to
  prove that $F \mapsto \Psi(\cdot,\cdot,F,\P_n)$ is continuously
  Fréchet differentiable as a map
  $\mathcal{H} \to L^2(Q_0^K \otimes \nu)$. Composition with the
  continuous linear operator $\Psi \mapsto \Phi$ from
  Lemma~\ref{lem:RKHS-representation} then yields that
  $\mathcal{T}_n : \mathcal{H}\to\mathcal{H}$ is continuously Fréchet
  differentiable.

  We now use a technical fact whose proof is deferred to the appendix
  in Lemma~\ref{lem:tech-F}(ii). For each fixed $(\bar\xi,v)$, the map
  $ F \mapsto \Psi(\bar\xi,v,F,\P_n)$ is continuously Fréchet
  differentiable as a map $\mathcal{H}\to\mathbb{R}$, with a
  derivative $F\mapsto \mathrm{d}_F \Psi(\bar\xi,v,\P_n)$ whose
  operator norm is uniformly bounded on bounded sets
  $\{F:\|F\|_{\mathcal{H}}\le M\}$. Therefore by integrating with
  respect to $Q_0^K \otimes\nu$, it is readily checked that
    \begin{align*}
      & \norm[\big]{\Psi(\cdot, \cdot,F+G,\P_n) - \Psi(\cdot, \cdot,F,\P_n)
      - \mathrm{d}_F \Psi(\cdot, \cdot,\P_n)\cdot G}_{L^2(Q_0^K
        \otimes\nu)} = o(\norm{G}_{\mathcal{H}})
    \end{align*}
  as $\norm{G}_{\mathcal{H}}\to 0$, uniformly when $F$ and $G$ are
  taken in bounded sets of $\mathcal{H}$.
  Therefore
  $F \mapsto \Psi(\cdot,\cdot,F,\P_n)$ is continuously Fréchet
  differentiable as a map $\mathcal{H} \to L^2(Q_0^K\otimes\nu)$, with
  a bounded derivative on bounded sets. This proves that
  $\mathcal{T}_n \in C_b^1(\mathcal{H},\mathcal{H})$.

  We now check the linear growth condition. In the decomposition
  $\Psi = \Psi_1 \Psi_2$ from~\eqref{eq:def-Psi}, the factor $\Psi_2$
  corresponds to the Radon--Nikodym derivative
  $\mathrm{d}Q_{n,F}/\mathrm{d}Q_0$ and is uniformly bounded by
  $K^{2^d-1}$, this is clear from its definition
  in~\eqref{eq:def-Psi2}. We again use a technical fact whose proof is
  found in \Cref{lem:tech-F}\textit{(i)}: the map
  $F\mapsto \Psi_1(\xi^1,v,F,\mathrm{P}_n)$ has gradient norm bounded
  by the constant $L\kappa$. Using this fact, we obtain the linear
  growth bound
\[
  |\Psi_1(\xi^1,v,F,\P_n)| \le |\Psi_1(\xi^1,v,0,\P_n)| +
  L\kappa\,\|F\|_{\mathcal{H}}.
\]
Consequently, the norm of $\Psi(\cdot,\cdot,F,\P_n)$ in
$L^2(Q_0^K \otimes \nu)$ grows at most linearly in
$\|F\|_{\mathcal{H}}$, and thus the vector field $\mathcal{T}_n(F)$
satisfies the linear growth condition.

\textit{Proof of $(ii)$.} By point $(i)$, the vector field
$\mathcal{T}_n$ is complete, i.e.\ the differential
equation~\eqref{eq:ode-igb} admits a unique global solution
$(F_t^n)_{t\ge 0}$. To show that this solution depends continuously on
the input sample $(x_i,y_i)_{1\leq i\leq n}$, by
\Cref{prop:first-order-perturbation} we need only show that the
initial condition $F^n_0$ and the vector field $\mathcal{T}_n$ both
depend continuously on $(x_i,y_i)_{1\leq i\leq n}$. Concerning the
initial condition, note that because we assumed that the maps
$z\mapsto \ell(z,y)$ are strictly convex and $C^2$, any sample such
that~\eqref{eq:igb-initialization} is well-defined admits a
neighborhood such that the initial condition $F^n_0$ is given by
\[
  (x_i,y_i)_{1\leq i\leq n} \longmapsto \Big( \text{unique zero of
  }z\mapsto \sum_{i=1}^n \partial_1\ell(z,y_i) \Big),
\]
which is well-defined and continuous by the implicit function theorem.
Let us now prove that the mapping
$(x_i,y_i)_{1\leq i\leq n} \mapsto \mathcal{T}_n \in
C^0_b(\mathcal{H},\mathcal{H})$ is continuous. For this, we exploit
the expression for $\mathcal{T}_n$ given in \Cref{lem:RKHS-rf}.
Consider two samples $(x_i,y_i)_{1\leq i\leq n}$ and
$(x'_i,y'_i)_{1\leq i\leq n}$ and write
$\mathrm{P}'_n=\frac{1}{n}\sum_{i=1}^n \delta_{(x'_i,y'_i)}$, and
$\mathcal{T}_n$ and $\mathcal{T}'_n$ for the vector fields defined
from \eqref{eq:def-rf} respectively using $\mathrm{P}_n$ and
$\mathrm{P}'_n$. Then for any $F\in \mathcal{H}$,
  \begin{equation}
    \label{eqpr:continuity-on-sample}
    \norm{\mathcal{T}_n(F)-\mathcal{T}'_n(F)}_{\mathcal{H}} = \int
    \Big(\Psi(\bar{\xi},v,F,\P_n) - \Psi(\bar{\xi},v,F,\P'_n)\Big)^2
    \,Q_0^K(\mathrm{d}\bar{\xi})\,\nu(\mathrm{d}v),
  \end{equation}
for $\Psi$ the mapping defined by Equations~\eqref{eq:def-Psi},
\eqref{eq:def-Psi1} and~\eqref{eq:def-Psi2}. Considering the first
sample $(x_i,y_i)_{1\leq i\leq n}$ as fixed, note that by the
assumption~\eqref{eq:reference-distribution}, for $Q_0$-almost every
tentative splitting schemes $\xi$ and every leaf $v\in \{0,1\}^d$,
there are no points among the $(x_i)_{1\leq i\leq n}$ at the boundary
of any $A^\xi_v$, i.e.\ the $\1_{A^\xi_v}$ are continuous at all
points of the input sample -- let us call such a splitting scheme a
\emph{good} splitting scheme. Fixing $M>0$, $F\in \mathcal{H}$ with
$\norm{F}_{\mathcal{H}}\leq M$ and a collection of good splitting
schemes $\bar{\xi}=(\xi^1,\dots,\xi^K)$, the mapping
$(x'_i,y'_i)_{1\leq i \leq n} \mapsto
\Psi(\bar{\xi},v,F,\mathrm{P}'_n)$ is a composition of elementary
functions that are analytical on a neighborhood of $(x_i,y_i)_{1\leq i
  \leq n}$, applied to the following terms:
\[
  \partial_1\ell(F(x'_i),y'_i), \quad y'_i, \quad \text{and} \quad
  \1_{A^{\xi^j}_v}(x'_i), \qquad \text{for } 1\leq i\leq n, 1\leq
  j\leq K, v\in \{0,1\}^d.
\]
Using the bound~\eqref{eq:bound-uniform-continuity} and
\Cref{ass:loss}, the first of these terms is continuous at each sample
point $(x_i,y_i)$, and one can find a modulus of continuity that
depends only on $M$, and not on $F$. For the same reason, there is a
uniform bound on $\Psi$ that does not depend on $F$. Therefore, we can
apply the dominated convergence theorem
in~\eqref{eqpr:continuity-on-sample} to get
\[
  \sup_{\norm{F}_{\mathcal{H}}\leq
    M}\norm{\mathcal{T}_n(F)-\mathcal{T}'_n(F)}_{\mathcal{H}}
  \longrightarrow 0 \qquad \text{as }(x'_i,y'_i)_{1\leq i\leq n} \to
  (x_i,y_i)_{1\leq i\leq n},
\]
which concludes the proof.
\end{proof}

For future reference, we state  an analog of Lemma~\ref{lem:gradient-field-regularity-igb} in the infinite population case, which  implies that the vector field $\mathcal{T}:\mathcal{H}\to\mathcal{H}$  is complete and that the ODE~\eqref{eq:ode-igb-population} has a unique solution $(F_t)_{t\ge 0}$ which is continuously differentiable. The proof is omitted for the sake of brevity, because it is a straightforward adaptation of the one of Lemma~\ref{lem:gradient-field-regularity-igb}$(i)$ with $\P_n$ replaced by $\P$.
\begin{lemma}\label{lem:gradient-field-regularity-igb-pop} The operator $F\mapsto \mathcal{T}(F)$ defined in Lemma~\ref{lem:RKHS-rf}$(ii)$ belongs to $C_b^1(\mathcal{H},\mathcal{H})$ and satisfies the linear growth condition.
\end{lemma}

\subsubsection{Proofs related to Subsection~\ref{subsec:IGB-large-sample}: law of large numbers.}

\begin{proof}[Proof of Theorem~\ref{thm:lln-igb}] 
\textit{Proof of $(i)$.} We want to prove the almost sure uniform convergence $\mathcal{T}_n(F)\to \mathcal{T}(F)$ on bounded sets $\{F:\|F\|_{\mathcal{H}}\le M\}$ for all $M>0$. Lemma~\ref{lem:RKHS-representation} together with Lemma~\ref{lem:RKHS-rf} imply
\[
\|\mathcal{T}_n(F)-\mathcal{T}(F)\|_{\mathcal{H}}^2
\le \int\big( \Psi(\bar\xi,v,F,\P_n)-\Psi(\bar\xi,v,F,\P)\big)^2\,Q_0^K(\mathrm{d}\bar\xi)\nu(\mathrm{d}v),
\]
so that we need to prove that, for all $M> 0$, as $n\to\infty$,
\begin{equation}\label{eq:L2-conv-sup}
\sup_{\|F\|_{\mathcal{H}}\le M}
\int\big( \Psi(\bar\xi,v,F,\P_n)-\Psi(\bar\xi,v,F,\P)\big)^2\,Q_0^K(\mathrm{d}\bar\xi)\nu(\mathrm{d}v)
\xrightarrow{a.s.} 0.
\end{equation}

By the Glivenko--Cantelli property stated in point $(ii)$ of Lemma~\ref{lem:Donsker}, the class $\mathcal{C}_M$ is $\P$-Glivenko--Cantelli, so that
\[
\sup_{\|F\|_{\mathcal{H}}\le M,\; A\in\mathcal{R}}
\big|\P_n[\partial_1\ell(F(x),y)\1_{A}(x)]-\P[\partial_1\ell(F(x),y)\1_{A}(x)]\big|
\stackrel{a.s.^*}{\longrightarrow} 0.
\]
Since the class $\mathcal{R}$ of hyperrectangles is also $\P$-Glivenko--Cantelli (it is a VC class), we also have
\[
\sup_{A\in\mathcal{R}}|\P_n(A)-\P(A)|\stackrel{a.s.^*}{\longrightarrow} 0.
\]
Outer almost sure convergence implies the existence of $\Omega_0\subset\Omega$ with $\mathbb{P}^*(\Omega\setminus\Omega_0)=0$ such that all above convergences hold on $\Omega_0$; see Chapter~1.9 in \cite{vW96}. On this set $\Omega_0$, for all fixed $A\in\mathcal{R}$ such that $\P(A)>0$,
\[
\sup_{\|F\|_{\mathcal{H}}\le M}
\Big|\frac{\P_n[\partial_1\ell(F(x),y)\1_{A}(x)]}{\P_n(A)}
      -\frac{\P[\partial_1\ell(F(x),y)\1_{A}(x)]}{\P(A)}\Big|
\longrightarrow 0.
\]
Similarly,
\[
\sup_{\|F\|_{\mathcal{H}}\le M}
\Big|\frac{\P_n[\partial_1\ell(F(x),y)\1_{A}(x)]^2}{\P_n(A)}
      -\frac{\P[\partial_1\ell(F(x),y)\1_{A}(x)]^2}{\P(A)}\Big|
\longrightarrow 0.
\]
The same converges hold when $\P(A)=0$ with the convention $0/0=0$. Next, we observe that $\Psi(\bar\xi,v,F,\P_n)$ can be written as a continuous function of the ratios
\[
\frac{\P_n[\partial_1\ell(F(x),y)\1_A(x)]}{\P_n(A)},\quad
\frac{\P_n[\partial_1\ell(F(x),y)\1_A(x)]^2}{\P_n(A)},
\]
where $A$ ranges over a finite collection of hyperrectangles $(A_i^{\bar\xi})_{i\in\mathcal{I}}$ that depend only on $(\bar\xi,v)$ (see Equations~\eqref{eq:ratios}--\eqref{eq:index-set} and Lemma~\ref{lem:tech-P} and its proof for details). As a consequence, we deduce that, on $\Omega_0$,
\[
\sup_{\|F\|_{\mathcal{H}}\le M}
\big|\Psi(\bar\xi,v,F,\P_n)-\Psi(\bar\xi,v,F,\P)\big|
\longrightarrow 0
\]
for each fixed $(\bar\xi,v)$.

We finally establish the uniform integrability condition
\begin{equation}\label{eq:uniform-integrability}
\int \sup_{n\ge 1}\sup_{\|F\|_{\mathcal{H}}\le M}
\Psi(\bar\xi,v,F,\P_n)^2 \,Q_0^K(\mathrm{d}\bar\xi)\nu(\mathrm{d}v)<\infty,
\end{equation}
which entails
\[
\int \sup_{\|F\|_{\mathcal{H}}\le M}
\big|\Psi(\bar\xi,v,F,\P_n)-\Psi(\bar\xi,v,F,\P)\big|^2
\,Q_0^K(\mathrm{d}\bar\xi)\nu(\mathrm{d}v)
\longrightarrow 0,
\]
and proves a strengthened version of~\eqref{eq:L2-conv-sup}. Lemma~\ref{lem:uniform-integrability} from the Appendix implies that 
\[
\mathbb{E}\!\int \sup_{n\ge1}\sup_{\|F\|_{\mathcal{H}}\le M}
|\Psi(\bar\xi,v,F,\P_n)|^2\,
Q_0^K(\mathrm{d}\bar\xi)\nu(\mathrm{d}v)
<\infty,
\]
which proves that the integral in~\eqref{eq:uniform-integrability} is finite almost surely and concludes the proof.

\textit{Proof of $(ii)$.} We apply Corollary~\ref{cor:stochastic-perturbation} with $E=\mathcal{H}$, $X_{0,n}=F_0^n$ and $G_n=\mathcal{T}_n$. The almost sure convergence of the initial positions $F_0^n\to F_0$ in $\mathcal{H}$ is assumed as an hypothesis. The almost sure convergence of the vector fields $\mathcal{T}_n\to\mathcal{T}$ in $C_b(\mathcal{H},\mathcal{H})$ has just been established in point $(i)$.  The limit vector field $\mathcal{T}$ is complete by Lemma~\ref{lem:gradient-field-regularity-igb-pop}.  Therefore,  Corollary~\ref{cor:stochastic-perturbation} implies the almost sure convergence $F^n\to F$  in $C^1([0,+\infty),\mathcal{H})$, where the limit $F$ is the solution of the ODE~\eqref{eq:ode-igb-population} started at $F_0$ at time $t=0$.
\end{proof}

\subsubsection{Proofs related to Subsection~\ref{subsec:IGB-large-sample}: functional central limit theorem.}

We begin with the definition of a family of RKHSs  $(\widetilde{\mathcal{H}}_\gamma)_{\gamma\in(0,\gamma_{\max})}$ defined by their reproducing kernels
\[
\tilde{k}_\gamma(z,z')
=
\int w_\gamma(\bar\xi)\,
\1_{A_v^{\xi^1}}(z)\,\1_{A_v^{\xi^1}}(z')
\,Q_0^K(\mathrm{d}\bar\xi)\,\nu(\mathrm{d}v),
\qquad z,z'\in[0,1]^p,
\]
with weights given by
\[
w_\gamma(\bar\xi) = \max_{i\in\mathcal{I}}\big(w_\gamma(A^{\bar\xi}_i)\big)
\]
where $(A^{\bar\xi}_i)_{i\in\mathcal{I}}$ denotes the collection of hypercubes appearing in the definition of $\Psi(\bar\xi,v,F,\P_n)$ -- see Equation~\eqref{eq:index-set} in the discussion preceding Lemma~\ref{lem:tech-P} -- and $w_\gamma(A)$ is defined in Equation~\eqref{eq:def-w-gamma}. 

\begin{lemma}\label{lem:RKHS-H-gamma}
For $\gamma\in[0,\gamma_{\max})$, the kernel $\tilde{k}_\gamma$ is continuous and bounded. Moreover, the family $(\widetilde{\mathcal{H}}_\gamma)_{\gamma\in[0,\gamma_{\max})}$ of associated RKHSs  satisfies the following properties:
\begin{itemize}
    \item[-] $\widetilde{\mathcal{H}}_0=\mathcal{H}$ is the RKHS with reproducing kernel $k$ defined in Equation~\eqref{eq:RKHS-kernel};
    \item[-] for all $0\leq \gamma_1< \gamma_2<\gamma_{\max}$,
    \[
    \widetilde{\mathcal{H}}_{\gamma_1}\subset \widetilde{\mathcal{H}}_{\gamma_2}
    \quad\text{and}\quad
    \|F\|_{\widetilde{\mathcal{H}}_{\gamma_2}}
    \leq
    \|F\|_{\widetilde{\mathcal{H}}_{\gamma_1}},
    \qquad F\in\widetilde{\mathcal{H}}_{\gamma_1},
    \]
    and $\widetilde{\mathcal{H}}_{\gamma_1}$ is dense in $\widetilde{\mathcal{H}}_{\gamma_2}$.
\end{itemize}
\end{lemma}

As we will see during the proof, we can take  $\widetilde{H}=\widetilde{H}_\gamma$ in  Theorem~\ref{thm:clt-igb}, for any $\gamma\in (2-\gamma_{\max},\gamma_{\max})$. The condition $\gamma_{\max}>1$ ensures that this choice is non empty.

\begin{proof}[Proof of Lemma~\ref{lem:RKHS-H-gamma}]
Each cell of the collection $(A_i^{\bar\xi})_{i\in\mathcal{I}}$ appears among the cells $(A_w^{\tilde\xi^1_j})_{j\in\mathcal{J},w\in\mathcal{T}_d}$  generated by the splitting scheme $\tilde\xi^1_j$ obtained from $\xi^1$ by replacing at most one split with a candidate split taken from $\bar\xi$ (this replacement is parametrized by $\mathcal{J}$). Since cells at inner nodes are further partitioned into leaf cells, we have
\[
w_\gamma(\bar\xi)\leq \max_{j\in\mathcal{J}}\max_{w\in\mathcal{T}_d} w_\gamma(A^{\tilde\xi^1_j}_w)= \max_{j\in\mathcal{J}}w_\gamma (\tilde{\xi}^1_j)\leq \sum_{j\in\mathcal{J}}w_\gamma (\tilde{\xi}^1_i).
\]
Then we can deduce that
\begin{align*}
\sup_{z,z'\in[0,1]^p} |\tilde{k}_\gamma(z,z')|
&\leq \sum_{j\in\mathcal{J}} \int w_\gamma(\tilde{\xi}^1_j)\,Q_0^K(\rmd\bar\xi)\,\nu(\rmd v) \\
&= \mathrm{card}(\mathcal{J}) \int w_\gamma(\xi)\,Q_0(\rmd\xi)\nu(\rmd v) < \infty.
\end{align*}
The equality holds because all terms in the sum are identical: under
$Q_0^K$, each $\tilde{\xi}^1_j$ has distribution $Q_0$ since the
splits are i.i.d. The finiteness follows from the definition of
$\gamma_{\max}$, as we assume $\gamma<\gamma_{\max}$.This proves that
$\tilde{k}_\gamma$ is bounded. Continuity of $\tilde{k}_\gamma$
follows from the same arguments as in the proof of Proposition~8 in
\citet{DDD25}, using that the map
\[
(z,z')\longmapsto w_\gamma(\bar\xi)\,\1_{A_v^{\xi^1}}(z)\,\1_{A_v^{\xi^1}}(z')
\]
is continuous at every $(z,z')\in[0,1]^p\times[0,1]^p$ for $Q_0^K(\rmd\bar\xi)\otimes\nu(\rmd v)$-almost every $(\bar\xi,v)$, together with the dominated convergence theorem.

Since $\gamma\mapsto w_\gamma(A)$ is nondecreasing, so is $\gamma\mapsto w_\gamma(\bar\xi)$, and, for $0\leq \gamma_1<\gamma_2<\gamma_{\max}$, 
\begin{equation}\label{eq:tilde-k-gamma}
(\tilde k_{\gamma_2}-\tilde k_{\gamma_1})(z,z')
=
\int
\big(w_{\gamma_2}(\bar\xi)-w_{\gamma_1}(\bar\xi)\big)\,
\1_{A_v^\xi}(z)\1_{A_v^\xi}(z')\,
Q_0^K(\rmd\bar\xi)\nu(\rmd v)
\end{equation}
defines a positive semidefinite kernel. The inclusion
$\widetilde{\mathcal{H}}_{\gamma_1}\subset \widetilde{\mathcal{H}}_{\gamma_2}$ together with the norm comparison then follows from a direct application of Theorem~2.7.11 in \cite{HE15}.

When $\gamma=0$, $w_0(A)\equiv 1$ and therefore $w_0(\bar\gamma)\equiv 1$, so that $\tilde{k}_0=k$ and $\widetilde{\mathcal{H}}_0=\mathcal{H}$.

To prove the density of $\widetilde{\mathcal{H}}_{\gamma_1}$ in $\widetilde{\mathcal{H}}_{\gamma_2}$, it is sufficient to show that $\widetilde{\mathcal{H}}_0$ is dense in $\widetilde{\mathcal{H}}_\gamma$ for all $\gamma\in(0,\gamma_{\max})$. By Theorem~2.7.7 in \cite{HE15} (see also Lemma~\ref{lem:RKHS-representation} and its proof), any function $F\in\widetilde{\mathcal{H}}_\gamma$ admits a representation of the form
\[
F(z)
=
\int \Psi(\bar\xi,v)\,w_\gamma(\bar\xi)^{1/2}\,
\1_{A_v^{\xi^1}}(z)\,
Q_0^K(\rmd\bar\xi)\,\nu(\rmd v)
\]
for some $\Psi\in L^2(Q_0^K\otimes\nu)$. For $\tau>0$, we define
\[
F^\tau(z)
=
\int
\Psi(\bar\xi,v)\,
w_\gamma(\bar\xi)^{1/2}\1_{\{w_\gamma(\bar\xi)\leq \tau \}}\,
\1_{A_v^{\xi^1}}(z)\,
Q_0^K(\rmd\bar\xi)\,\nu(\rmd v).
\]
It is clear that 
\[
(\bar\xi,v)\mapsto \Psi(\bar\xi,v)\,
w_\gamma(\bar\xi)^{1/2}\1_{\{w_\gamma(\bar\xi)\leq \tau \}} \in L^2(Q_0^K\otimes\nu)
\]
so that $F^\tau\in\widetilde{\mathcal{H}}_0$. Moreover, 
\begin{align*}
\|F-F_\tau\|_{\widetilde{\mathcal{H}}_\gamma}^2&\leq \int \Psi(\bar\xi,v)^2\1_{\{w_\gamma(\bar\xi)>\tau\}}Q_0^K(\rmd\bar\xi)\,\nu(\rmd v)  
\end{align*}
converges to $0$ as $\tau\to\infty$, proving the convergence $F^\tau\to F$ in $\widetilde{\mathcal{H}}_\gamma$. This proves  the density of $\widetilde{\mathcal{H}}_0$ in $\widetilde{\mathcal{H}_\gamma}$.
\end{proof}

The following lemma explains the role played by the RKHS $\widetilde{\mathcal{H}}_\gamma$ and will be used in the proof of Theorem~\ref{thm:clt-igb}.

\begin{lemma}\label{lem:role-H-gamma}
Let $\gamma_{\max}>1$. Then, for any bounded functions
$\phi_i\in L^\infty(Q_0^K\otimes \nu)$, $i\in\mathcal{I}$, satisfying
\[
\max_{i\in\mathcal{I}}
\sup_{\bar\xi\in\mathscr{S}^K,\; v\in\{0,1\}^d}
|\phi_i(\bar\xi,v)| \leq C,
\]
the function $F$ defined by
\[
F(z)
=
\int
\Bigg(\sum_{i\in\mathcal{I}}
\frac{\phi_i(\bar\xi,v)}{\P(A_i^{\bar\xi})}\1_{\{\P(A_i^{\bar\xi})>0\}}\Bigg)
\1_{A_v^{\xi^1}}(z)\,
Q_0^K(\rmd\bar\xi)\,\nu(\rmd v),
\qquad z\in[0,1]^p,
\]
belongs to $\widetilde{\mathcal{H}}_\gamma$ for all
$\gamma\in(2-\gamma_{\max},\gamma_{\max})$, and its norm satisfies
\[
\|F\|_{\widetilde{\mathcal{H}}_\gamma}\leq C\,\mathrm{card}(\mathcal{I})\,
\|w_{1-\gamma/2}(\bar\xi)\|_{L^2(Q_0^K\otimes\nu)}.
\]
\end{lemma}

In other words, the lemma states a continuity property of the linear map
\[
(\phi_i)_{i\in\mathcal{I}}
\in \big(L^\infty(Q_0^K\otimes \nu)\big)^{\mathcal{I}}
\longmapsto
F\in \widetilde{\mathcal{H}}_\gamma .
\]

\begin{proof}[Proof of Lemma~\ref{lem:role-H-gamma}]
Using the definition~\eqref{eq:tilde-k-gamma} of the kernel $\tilde{k}_\gamma$
and applying Theorem~2.7.7 in \cite{HE15}, we obtain that
$F\in \widetilde{\mathcal{H}}_\gamma$ as soon as the function
\[
\Psi(\bar\xi,v)
=
\Bigg(\sum_{i\in\mathcal{I}}
\frac{\phi_i(\bar\xi,v)}{\P(A_i^{\bar\xi})}\1_{\{\P(A_i^{\bar\xi})>0\}}\Bigg)
 w_\gamma(\bar\xi)^{-1/2}
\]
belongs to $L^2(Q_0^K\otimes\nu)$.

By the uniform bound on the $\phi_i$ and the definition of the weights
$w_\gamma$, we have
\[
|\Psi(\bar\xi,v)|
\leq
C\,\mathrm{card}(\mathcal{I})\, w_{1-\gamma/2}(\bar\xi).
\]
By definition of $\gamma_{\max}$, this upper bound is square integrable
whenever $2-\gamma<\gamma_{\max}$. Hence, for
$\gamma\in(2-\gamma_{\max},\gamma_{\max})$, we obtain
$F\in \widetilde{\mathcal{H}}_\gamma$ with
\[
\|F\|_{\widetilde{\mathcal{H}}_\gamma}
\leq
\|\Psi\|_{L^2(Q_0^K\otimes \nu)}
\leq
C\,\mathrm{Card}(\mathcal{I})\,
\|w_{1-\gamma/2}(\bar\xi)\|_{L^2(Q_0^K\otimes\nu)}.
\]
This concludes the proof.
\end{proof}

\begin{proof}[Proof of Theorem~\ref{thm:clt-igb} $(i)$]
Throughout the proof, we fix $\gamma\in (2-\gamma_{\max},\gamma_{\max})$ and establish~$(i)$ with  $\widetilde{H}=\widetilde{H}_\gamma$.

\textit{Strategy for the proof.} Because $\mathcal{H}\subset \widetilde{H}_\gamma$, we can view the scaled difference
\[
\sqrt{n}\big(\mathcal{T}_n(F)-\mathcal{T}(F)\big)
=
\int
\sqrt{n}\big(\Psi(\bar\xi,v,F,\P_n)-\Psi(\bar\xi,v,F,\P)\big)
\,\1_{A_v^{\xi^1}}(\cdot)\,Q_0^K(\rmd \bar\xi)\nu(\rmd v),
\]
 as an element of $C^0_b(B,\widetilde{\mathcal{H}}_\gamma)$, where $B=\{F\colon\|F\|_{\mathcal{H}}\leq M\}$ with arbitrary $M>0$.

To avoid arbitrarily small denominators, we introduce the following notion:
\begin{equation}\label{eq:regular}
\text{$\bar\xi$ is $\varepsilon_n$-regular} \quad \text{if}\quad \P(A^{\bar\xi}_i)\geq\varepsilon_n \quad \text{for all } i\in \mathcal{I},
\end{equation}
where $(A^{\bar\xi}_i)_{i\in\mathcal{I}}$ denotes the (finite) collection of all hypercubes involved in the definition of $\Psi(\bar\xi,v,F,\P_n)$ --  details are given in the Appendix, see Equation~\eqref{eq:index-set}. The integration domain $\mathcal{D}=\mathscr{S}^K\times \{0,1\}^d$ is partitioned accordingly as
\[
\mathcal{D}
=
\mathcal{D}_{\varepsilon_n}
\cup
(\mathcal{D}\setminus\mathcal{D}_{\varepsilon_n}),
\qquad
\mathcal{D}_{\varepsilon_n}
=
\{(\bar\xi,v)\in\mathcal{D}:\ \text{$\bar\xi$ is $\varepsilon_n$-regular}\}.
\]

We introduce the decomposition
\[
\mathcal{T}_n(F)=\mathcal{T}_n^{\varepsilon_n}(F)+\mathcal{R}_n^{\varepsilon_n}(F),
\]
with
\begin{align*}
\mathcal{T}_n^{\varepsilon_n}(F)
&=
\int_{\mathcal{D}_{\varepsilon_n}}
\Psi(\bar\xi,v,F,\P_n)\,\1_{A_v^{\xi^1}}(\cdot)\,
Q_0^K(\rmd \bar\xi)\nu(\rmd v),\\
\mathcal{R}_n^{\varepsilon_n}(F)
&=
\int_{\mathcal{D}\setminus\mathcal{D}_{\varepsilon_n}}
\Psi(\bar\xi,v,F,\P_n)\,\1_{A_v^{\xi^1}}(\cdot)\,
Q_0^K(\rmd \bar\xi)\nu(\rmd v).
\end{align*}
Similarly, with obvious notation,
\[
\mathcal{T}(F)=\mathcal{T}^{\varepsilon_n}(F)+\mathcal{R}^{\varepsilon_n}(F).
\]
We will see that, for a suitable choice of the truncation sequence $\varepsilon_n\to 0$,
the term $\sqrt{n}\big(\mathcal{T}_n^{\varepsilon_n}-\mathcal{T}^{\varepsilon_n}\big)$ captures the main contribution,
while $\sqrt{n}\big(\mathcal{R}_n^{\varepsilon_n}-\mathcal{R}^{\varepsilon_n}\big)$ is a negligible remainder.
More precisely, we prove below that, for $\gamma\in(2-\gamma_{\max},\gamma_{\max})$ and a suitable choice of $\delta\in(0,1/2)$, the truncation level $\varepsilon_n=n^{-\delta}$ ensures that
\begin{equation}\label{eq:step1}
n\,\mathbb{E}\!\left[
\sup_{F\in B}
\|\mathcal{R}_n^{\varepsilon_n}(F)-\mathcal{R}^{\varepsilon_n}(F)\|_{\widetilde{\mathcal{H}}_\gamma}^2
\right]
\;\longrightarrow\; 0
\quad \text{as } n\to\infty,
\end{equation}
and
\begin{equation}\label{eq:step2}
\sqrt{n}\big(\mathcal{T}_n^{\varepsilon_n}-\mathcal{T}^{\varepsilon_n}\big)
\;\xrightarrow{d}\;
\mathscr{W}
\quad \text{in } C^0_b(B,\widetilde{\mathcal{H}}_\gamma),
\end{equation}
with a tight Gaussian limit $\mathscr{W}$. 
Equation~\eqref{eq:step1} implies the convergence in probability 
$\sqrt{n}\big(\mathcal{R}_n^{\varepsilon_n}-\mathcal{R}^{\varepsilon_n}\big) \rightarrow 0$  in $C^0_b(B,\widetilde{\mathcal{H}}_\gamma)$. Then, 
Equation~\eqref{eq:step2}, together with Slutsky’s lemma,  yields
\[
\sqrt{n}\big(\mathcal{T}_n-\mathcal{T}\big)
\;\xrightarrow{d}\;
\mathscr{W}
\quad \text{in } C^0_b(B,\widetilde{\mathcal{H}}_\gamma).
\]
Since the bounded set $B\subset\mathcal{H}$ is arbitrary, this establishes convergence in distribution in
$C^0_b(\mathcal{H},\widetilde{\mathcal{H}}_\gamma)$. In the sequel, we prove Equations~\eqref{eq:step1} and~\eqref{eq:step2} successively.

\textit{Proof of Equation~\eqref{eq:step1}.}
It is enough to prove that
\begin{equation}\label{eq:step1bis}
n\mathbb{E}\Big[\sup_{F\in B}\|\mathcal{R}_n^{\varepsilon_n}(F)\|_{\widetilde{\mathcal{H}}_\gamma}^2\Big]\longrightarrow 0
\quad\mbox{and}\quad
n\sup_{F\in B}\|\mathcal{R}^{\varepsilon_n}(F)\|_{\widetilde{\mathcal{H}}_\gamma}^2\longrightarrow 0.
\end{equation}
The RKHS squared norm satisfies, for $F\in B$,
\[
\|\mathcal{R}_n^{\varepsilon_n}(F)\|_{\widetilde{\mathcal{H}}_\gamma}^2
=
\int_{\mathcal{D}\setminus\mathcal{D}_{\varepsilon_n}}
|\Psi(\bar\xi,v,F,\P_n)|^2
\big(\min_{i\in\mathcal{I}}\P(A^{\bar\xi}_i)\big)^\gamma\,
Q_0^K(\rmd\bar\xi)\nu(\rmd v).
\]
Considering the supremum over $F\in B$ and taking expectations, we obtain
\begin{align*}
&n\mathbb{E}\Big[\sup_{F\in B}\|\mathcal{R}_n^{\varepsilon_n}(F)\|_{\widetilde{\mathcal{H}}_\gamma}^2\Big]\\
\leq&
n \int_{\mathcal{D}\setminus\mathcal{D}_{\varepsilon_n}}
\mathbb{E}\Big[\sup_{F\in B}|\Psi(\bar\xi,v,F,\P_n)|^2\Big]
\big(\min_{i\in\mathcal{I}}\P(A^{\bar\xi}_i)\big)^\gamma\,
Q_0^K(\rmd\bar\xi)\nu(\rmd v).
\end{align*}
By Lemma~\ref{lem:uniform-integrability}, the expectation is uniformly bounded by some constant $C_M>0$, and we obtain a bound of the form
\[
n\mathbb{E}\Big[\sup_{F\in B}\|\mathcal{R}_n^{\varepsilon_n}(F)\|_{\widetilde{\mathcal{H}}_\gamma}^2\Big]
\leq
C_M\, n
\int_{\mathcal{D}\setminus\mathcal{D}_{\varepsilon_n}}
\big(\min_{i\in\mathcal{I}}\P(A^{\bar\xi}_i)\big)^\gamma\,
Q_0^K(\rmd\bar\xi)\nu(\rmd v).
\]
Since $\min_{i\in\mathcal{I}}\P(A^{\bar\xi}_i)<\varepsilon_n$ on the domain of integration, we can bound, for any $\gamma'\in(0,\gamma_{\max})$,
\begin{align*}
n \int_{\mathcal{D}\setminus\mathcal{D}_{\varepsilon_n}}
\big(\min_{i\in\mathcal{I}}\P(A^{\bar\xi}_i)\big)^\gamma\,
Q_0^K(\rmd\bar\xi)\nu(\rmd v)
&\leq
n\varepsilon_n^{\gamma+\gamma'}
\int_{\mathcal{D}}
\frac{1}{\big(\min_{i\in\mathcal{I}}\P(A^{\bar\xi}_i)\big)^{\gamma'}}\,
Q_0^K(\rmd\bar\xi)\nu(\rmd v)\\
&\leq
C n\varepsilon_n^{\gamma+\gamma'}
\end{align*}
for some constant $C>0$, where the integral on $\mathcal{D}$ is finite because $\gamma'<\gamma_{\max}$.  
Setting $\varepsilon_n=n^{-\delta}$, convergence to $0$ follows provided $(\gamma+\gamma')\delta>1$.  
By the restrictions on $\delta$ and $\gamma'$, this condition can be satisfied as soon as
\[
\frac{\gamma+\gamma_{\max}}{2}>1,
\quad\text{i.e.,}\quad \gamma>2-\gamma_{\max}.
\]
The proof of the second convergence in Equation~\eqref{eq:step1bis} is analogous (with $\P_n$ replaced by $\P$) and is omitted for brevity.  
This concludes the proof of Equation~\eqref{eq:step1}.

\textit{Proof of Equation~\eqref{eq:step2}.} We observe that, for $F$ in a bounded set $B$, $\Psi(\bar\xi,v,F,\P)$ and hence $\mathcal{T}(F)$ depend on $\P$ only through the collection $(\P[\varphi])_{\varphi\in\mathcal{G}_0}$, indexed by the class \begin{equation}\label{eq:def-G}
\mathcal{G}_0=\mathcal{G}_1\cup\mathcal{G}_2
\end{equation}
 with
\begin{align*}
\mathcal{G}_1 &= \{(x,y)\mapsto \partial_1\ell(F(x),y)\1_A(x): A\in\mathcal{R},\ F\in B\},\\
\mathcal{G}_2 &= \{(x,y)\mapsto \1_A(x): A\in\mathcal{R}\}.
\end{align*}
The collection $(\P[\varphi])_{\varphi\in\mathcal{G}_0}$ is uniformly bounded and can therefore be viewed as an element of $\ell^\infty(\mathcal{G}_0)$. The same observation applies to $\Psi(\bar\xi,v,F,\P_n)$ and $\mathcal{T}_n(F)$, which can be expressed as a function of $(\P_n[\varphi])_{\varphi\in\mathcal{G}}$ viewed as an element of $\ell^\infty(\mathcal{G}_0)$.

By Lemma~\ref{lem:Donsker}$(ii)$, the class $\mathcal{G}_0$ is Donsker so that the normalized empirical process
\[
\mathbb{G}_nf =\sqrt{n}(\P_n[f]-\P[f]),\quad f\in\mathcal{G}_0,
\]
converges in distribution in $\ell^\infty(\mathcal{G}_0)$ to a tight $\P$-Brownian bridge denoted by $\mathbb{G}_{\P}$. We then appeal to 
a non-measurable version of Skorokhod's representation theorem, namely
Theorem~1.10.4 in \citet{vW96}. We may thus assume without loss of generality that the empirical processes $\P_n$, $n\geq 1$, the population measure $\P$, and the Brownian bridge $\mathbb{G}_\P$, viewed as mappings $\Omega \to \ell^\infty(\mathcal{G})$, are such that:
\[
\mathbb{G}_n=\sqrt{n}(\P_n-\P)\stackrel{au}\longrightarrow \mathbb{G}_{\P} \quad \text{ in $\ell^\infty(\mathcal{G}_0)$ \ as $n\to\infty$},
\]
where $\stackrel{au}\longrightarrow$ denotes outer almost-uniform
convergence as defined in \citet[Definition~1.9.1]{vW96}.

Using this Skorokhod representation, we prove below that
\begin{align}
&\sqrt{n}\big(\mathcal{T}_n^{\varepsilon_n}(F)-\mathcal{T}^{\varepsilon_n}(F)\big)\nonumber\\
=&
\int_{\mathcal{D}_{\varepsilon_n}}
\sqrt{n}\big(\Psi(\bar\xi,v,F,\P_n)-\Psi(\bar\xi,v,F,\P)\big)
\1_{A_v^{\xi^1}}(\cdot)\,
Q_0^K(\rmd \bar\xi)\nu(\rmd v) \label{eq:integral-form}
\end{align}
satisfies
\begin{equation}\label{eq:cv-as-Skorohod-2}
\sqrt{n}\big(\mathcal{T}_n^{\varepsilon_n}-\mathcal{T}^{\varepsilon_n}\big)
\stackrel{a.s.}\longrightarrow \mathscr{W}
\quad\text{in } C^0_b(B,\mathcal{K}_\gamma).
\end{equation}
Let us stress that the left-hand side is measurable (as a function of
the input sample) in the space $C^0_b(B,\mathcal{K}_\gamma)$, so that almost sure convergence is understood here in the usual sense. Almost sure convergence under the Skorokhod representation implies convergence in distribution in the original probability space. Therefore, Equation~\eqref{eq:cv-as-Skorohod-2} implies Equation~\eqref{eq:step2}.

The remainder of the proof is devoted to the proof of Equation~\eqref{eq:cv-as-Skorohod-2}. Using Equation~\eqref{eq:integral-form} and Lemma~\ref{lem:tech-P} from the appendix, we introduce, for $F\in\mathcal{B}$,
\[
\mathscr{W}^{\varepsilon_n}(F)
=
\int_{\mathcal{D}_{\varepsilon_n}}
\sum_{i\in\mathcal{I}}
\left(
\frac{\Gamma_i\,\mathbb{G}_\P[\varphi_{A_i^{\bar\xi}}]
+\Gamma_i'\,\mathbb{G}_\P[\varphi_{A_i^{\bar\xi},F}]}
{\P(A_i^{\bar\xi})}
\right)
\1_{A_v^{\xi^1}}(\cdot)\,
Q_0^K(\rmd \bar\xi)\nu(\rmd v).
\]
Note that both the numerators and the denominators are bounded (because $\bar\xi$ is $\epsilon_n$-regular) so that $\mathscr{W}_{\varepsilon_n}(F)\in\mathcal{H}\subset \widetilde{\mathcal{H}}_\gamma$.
According to Lemma~\ref{lem:tech-P}, this definition ensures that
\[
\sqrt{n}\big(\mathcal{T}_n^{\varepsilon_n}(F)-\mathcal{T}^{\varepsilon_n}(F)\big)
-\mathscr{W}^{\varepsilon_n}(F)
=
\int_{\mathcal{D}_{\varepsilon_n}}
\sum_{i\in\mathcal{I}}
\left(
\frac{o(1)}{\P(A_i^{\bar\xi})}
\right)
\1_{A_v^{\xi^1}}(\cdot)\,
Q_0^K(\rmd \bar\xi)\nu(\rmd v),
\]
with a $o(1)$ term uniform for $F\in B$, $\bar\xi\in \mathcal{D}_{\varepsilon_n}$ and $i\in\mathcal{I}$. Therefore we deduce
\[
\sup_{F\in B} \Big\| \sqrt{n}\big(\mathcal{T}_n^{\varepsilon_n}(F)-\mathcal{T}^{\varepsilon_n}(F)\big)
-\mathscr{W}^{\varepsilon_n}(F)\Big\|_{\widetilde{\mathcal{H}}_\gamma}\stackrel{a.s.}\longrightarrow 0.
\]
On the other hand, Lemma~\ref{lem:role-H-gamma} implies that
\begin{equation}\label{eq:def-mathscrW}
\mathscr{W}(F)
=
\int_{\mathcal{D}}
\sum_{i\in\mathcal{I}}
\left(
\frac{\Gamma_i\,\mathbb{G}_\P[\varphi_{A_i^{\bar\xi}}]
+\Gamma_i'\,\mathbb{G}_\P[\varphi_{A_i^{\bar\xi},F}]}
{\P(A_i^{\bar\xi})}
\right)
\1_{A_v^{\xi^1}}(\cdot)\,
Q_0^K(\rmd \bar\xi)\nu(\rmd v)
\end{equation}
defines an element of $\widetilde{\mathcal{H}}_\gamma$. Moreover, with similar arguments as in the proof of Lemma~\ref{lem:role-H-gamma}, we obtain
\begin{align*}
\|\mathscr{W}(F)-\mathscr{W}^{\varepsilon_n}(F)\|_{\widetilde{\mathcal{H}}_\gamma}
&\leq \mathrm{card}(\mathcal{I})\, \max_{i\in\mathcal{I}} \Big\| \Gamma_i\,\mathbb{G}_\P[\varphi_{A_i^{\bar\xi}}]
+\Gamma_i'\,\mathbb{G}_\P[\varphi_{A_i^{\bar\xi},F}]\Big\|_{L^\infty(Q_0^K\otimes\nu)}\\
&\qquad \times \int_{\mathcal{D}\setminus\mathcal{D}_{\varepsilon_n}}
w_{2-\gamma}(\bar\xi)\, Q_0^K(\rmd \bar\xi)\nu(\rmd v),
\end{align*}
which converges to $0$  uniformly in $F\in B$ as $n\to\infty$. Altogether, in Equation~\eqref{eq:cv-as-Skorohod-2}, we have proven the almost sure convergence
\[
\sqrt{n}\big(\mathcal{T}_n^{\varepsilon_n}(F)-\mathcal{T}^{\varepsilon_n}(F)\big)
\stackrel{a.s.}\longrightarrow \mathscr{W}(F)
\quad\text{in } \widetilde{\mathcal{H}}_\gamma \text{ uniformly on } B.
\]
Since the left-hand side is continuous in $F$ and uniformly bounded on $B$, the convergence takes place in $C_b^0(B,\widetilde{\mathcal{H}}_\gamma)$, concluding the proof of Equation~\eqref{eq:cv-as-Skorohod-2}.

We finally show that the limit $\mathscr{W}$ is tight in $\mathcal{C}_b^0(B,\widetilde{\mathcal{H}}_\gamma)$. To this aim, we argue in view of Equation~\eqref{eq:def-mathscrW} that $\mathscr{W}=\mathfrak{O}(\mathbb{G}_{\P})$ is the image of the tight variable $\mathbb{G}_{\P}$ under the  mapping $\mathfrak{O}\colon \ell^\infty(\mathcal{G}_0)\to\ell^\infty(B, \widetilde{\mathcal{H}}_\gamma)$ defined by
\[
\mathfrak{O}(G):F\in B\mapsto \int_{\mathcal{D}}
\sum_{i\in\mathcal{I}}
\left(
\frac{\Gamma_i\,G[\varphi_{A_i^{\bar\xi}}]
+\Gamma_i'\,G[\varphi_{A_i^{\bar\xi},F}]}
{\P(A_i^{\bar\xi})}
\right)
\1_{A_v^{\xi^1}}(\cdot)\,
Q_0^K(\rmd \bar\xi)\nu(\rmd v)\in \widetilde{\mathcal{H}}_\gamma.
\]
Note that $(\Gamma_i)_{i\in\mathcal{I}}$ and $(\Gamma_i')_{i\in\mathcal{I}}$ implicitely depend on $F\in B$ and are uniformly bounded according to Lemma~\ref{lem:tech-P}. The mapping $G\mapsto \mathfrak{O}(G)$ is  linear and Lemma~\ref{lem:role-H-gamma} implies its continuity. Therefore $\mathscr{W}$ is tight in $\ell^\infty(B, \widetilde{\mathcal{H}}_\gamma)$ as the continuous image of the tight variable $\mathbb{G}_\P$. Since    $\mathscr{W}\in C^0_b(B, \widetilde{\mathcal{H}}_\gamma)$ a.s.,  tightness also holds   in the closed subspace $C^0_b(B, \widetilde{\mathcal{H}}_\gamma)\subset\ell^\infty(B, \widetilde{\mathcal{H}}_\gamma)$.
\end{proof}

\begin{proof}[Proof of Theorem~\ref{thm:clt-igb} $(ii)$]
In order to apply Corollary~\ref{cor:stochastic-perturbation} $(ii)$, it is sufficient to prove that
\begin{equation}\label{eq:final-proof-1}
\big(\sqrt{n}(F_0^n-F_0),\,\sqrt{n}(\mathcal{T}_n-\mathcal{T})\big)
\xrightarrow{d}
(\mathscr{F}_0,\mathscr{W})
\quad
\text{in }
\widetilde{\mathcal{H}}_\gamma \times C_b^0(\mathcal{H}, \widetilde{\mathcal{H}}_\gamma),
\end{equation}
and that
\begin{equation}\label{eq:final-proof-2}
\mathcal{T}\in C^1_{\widetilde{\mathcal{H}}_\gamma}(\mathcal{H},\mathcal{H}).
\end{equation}

We first comment on~\eqref{eq:final-proof-1}.  
The convergence of the second marginal,
\[
\sqrt{n}(\mathcal{T}_n-\mathcal{T})\xrightarrow{d}\mathscr{W},
\]
has just been established in the proof of point~$(i)$, using the Donsker property of the function class $\mathcal{C}_M$ and functional delta-method computations. The convergence of the first marginal,
\[
\sqrt{n}(F_0^n-F_0)\xrightarrow{d}\mathscr{F}_0,
\]
is a direct consequence of Equation~\eqref{eq:igb-initialization-normal}, together with the central limit theorem applied to $\P_n[f_0(y)]$, yielding a centered Gaussian limit $\mathscr{F}_0$ with variance $\P[f_0(y)^2]-\P[f_0(y)]^2$. Note that this convergence in distribution holds in $\mathbb{R}$ and that the embedding $\mathbb{R}\hookrightarrow \widetilde{\mathcal{H}}_\gamma$ is continuous.

Since the enlarged class $\mathcal{C}_M\cup\{f_0\}$ remains Donsker, the two convergences can be combined within the same functional delta-method argument, which yields the joint convergence stated in~\eqref{eq:final-proof-1}.

We next consider~\eqref{eq:final-proof-2}. We prove in fact that we can extend the definition of $\mathcal{T}$ into a continuously Fréchet differentiable extension   $\mathcal{T}: \widetilde{\mathcal{H}}_\gamma\to \mathcal{H}$, which clearly implies \eqref{eq:final-proof-2}. 
In Lemma~\ref{lem:RKHS-rf}$(ii)$ the definition of $\mathcal{T}(F)$ and $\Psi(\bar\xi,v,F,\P)$  holds for all $F\in L^2(\P)$, and in particular for all $F\in\widetilde{\mathcal{H}}_\gamma$ for $0<\gamma<\gamma_{\max}$, because the RKHS  contains bounded continuous functions. A straightforward adaptation of Lemma~\ref{lem:tech-F} then shows that, for all $\bar\xi$, $v$, the function $F\mapsto \Psi(\bar\xi,v,F,\P)$ is continuously Fr\'echet differentiable on $\widetilde{\mathcal{H}}_\gamma$ with a uniformly bounded gradient 
\[
\|\widetilde{\nabla}\Psi(\bar\xi,v,F,\P)\|_{\widetilde{\mathcal{H}}_\gamma} \leq \widetilde{C}_M
\]
on the ball $\|F\|_{\widetilde{\mathcal{H}}_\gamma}\leq M$. The only modification in the proof is to replace the feature map $x\mapsto \Phi(x)\in\mathcal{H}$ by the feature map $x\mapsto \widetilde{\Phi}(x)\in\widetilde{\mathcal{H}}_\gamma$ and the upper bound $\kappa$ for the kernel $k$ by the upper bound $\tilde \kappa$ for the kernel $\tilde{k}_\gamma$. 
Then, differentiating the relation
\[
\mathcal{T}(F)=\int \Psi(\bar\xi,v,F,\P)  \1_{A_v^{\xi^1}}(\cdot)\, Q_0^K(\rmd \bar\xi)\nu(\rmd v),
\]
we obtain 
\[
\rmd_F\mathcal{T}(G)=\int \langle \widetilde{\nabla}_F\Psi(\bar\xi,v,F,\P),G \rangle_{\widetilde{\mathcal{H}}_\gamma} \1_{A_v^{\xi^1}}(\cdot)\, Q_0^K(\rmd \bar\xi)\nu(\rmd v)
\]
with 
\begin{align*}
\|\rmd_F\mathcal{T}(G)\|_{\mathcal{H}}^2 &\leq  \int \langle \widetilde{\nabla}_F\Psi(\bar\xi,v,F,\P),G \rangle_{\widetilde{\mathcal{H}}_\gamma}^2 \, Q_0^K(\rmd \xi)\nu(\rmd v)\\
&\leq  \int \| \widetilde{\nabla}_F\Psi(\bar\xi,v,F,\P)\|_{\widetilde{\mathcal{H}}_\gamma}^2 \|G\|_{\widetilde{\mathcal{H}}_\gamma}^2 \, Q_0^K(\rmd \xi)\nu(\rmd v)\\
&\leq 2^d\,\widetilde{C}_M^2    \|G\|_{\widetilde{\mathcal{H}}_\gamma}^2.
\end{align*}
 The Fr\'echet differentiability of $\mathcal{T}: \widetilde{\mathcal{H}}_\gamma\to \mathcal{H}$ follows. The continuity of the differential $F\mapsto \rmd_F\mathcal{T}$ is a consequence of the continuity of $F\mapsto \widetilde{\nabla}_F\Psi(\bar\xi,v,F,\P)$ and of the dominated convergence Theorem. This concludes the proof of Equation~\eqref{eq:final-proof-2} and of part $(ii)$ of the Theorem.
\end{proof}

\appendix

\section{Technical computations}\label{app:technical-computations}
 
In two technical Lemmas, we analyze how the function
$\Psi(\bar\xi,v,F,\P)$ defined in Lemma~\ref{lem:RKHS-rf} depends on
$F$ and $\P$ respectively. We provide in a third lemma an uniform
integrability property for $\Psi$. In all of
\Cref{app:technical-computations}, we work in the setting of
infinitesimal gradient boosting, as defined in
Section~\ref{subsec:IGB-background}, and we use the notation of
Lemma~\ref{lem:RKHS-rf}.

\subsection{Dependency of $\Psi$ with respect to $F$}

\begin{lemma}\label{lem:tech-F}
  Let $\Psi,\Psi_1,\Psi_2$ be defined as
  in Equations~\eqref{eq:def-Psi}, \eqref{eq:def-Psi1}, \eqref{eq:def-Psi2}.
  Then:
  \begin{enumerate}[(i)]
  \item For all fixed $(\xi^1,v)$, the mapping
    $F \mapsto \Psi_1(\xi^1,v,F,\P)$ is Fréchet differentiable on
    $\mathcal{H}$ and its gradient norm is bounded by
    $L\kappa$, where $L$ and $\kappa$ are respectively defined in
    Assumptions~\ref{ass:loss} and~\ref{ass:kernel}.
  \item For all fixed $(\bar\xi,v)$, the mapping
    $F \mapsto \Psi(\bar\xi,v,F,\P)$ is Fréchet differentiable on
    $\mathcal{H}$ and its gradient satisfies, for all $M\geq 0$,
    \[
      \sup_{\|F\|_{\mathcal{H}}\leq M} \big\| \nabla_F
      \Psi(\bar{\xi},v,F,\P) \big\|_{\mathcal{H}} \;\leq \; C_M,
    \]
    for some constant $C_M>0$ that does not depend on $(\bar\xi,v)$.
    
    Furthermore, when $\P$ is replaced by $\P_n$, similar uniform
    bounds hold with constants that may depend on $n\geq 1$.
  \end{enumerate}
\end{lemma}

\begin{proof}[Proof of Lemma~\ref{lem:tech-F}] We present the proof only for $\P$, since the argument for $\P_n$ follows the same lines and is in fact simpler, as it involves only finite sums.

Recall from Lemma~\ref{lem:RKHS-rf} that
\[
  \Psi(\bar{\xi},v,F,\P)=\Psi_1(\xi^1,v,F,\P)\,\Psi_2(\bar{\xi},F,\P).
\]
For clarity we divide the argument into four steps. During the proof,
the constant $C_M>0$ does not depend on $(\bar\xi,v)$ and may change
from line to line.

\textit{Step 1. Bounding the gradient norm of the scores.} The map
$F \mapsto F(x)$ is linear and continuous on $\mathcal{H}$, with
gradient $\Phi(x)$. By composition, for fixed $(x,y)$, the map
$F \mapsto \partial_1\ell(F(x),y)\1_A(x)$ is continuously
differentiable on $\mathcal{H}$ with gradient
$F \mapsto \partial_1^2\ell(F(x),y)\1_A(x)\Phi(x)$. This gradient has
norm bounded by
\[
\sup_{z,y}|\partial_1^2\ell(z,y)| \, \sup_x \|\Phi(x)\|_{\mathcal{H}} = L\kappa,
\]
and therefore the map
\[
F \mapsto \P\!\left[\partial_1\ell(F(x),y)\1_A(x)\right]
\]
is continuously differentiable with gradient
\[
F \mapsto \P\!\left[\partial_1^2\ell(F(x),y)\1_A(x)\Phi(x)\right],
\]
whose norm is bounded by $L\kappa$.

For any split $A=A_0\cup A_1$, the score
\[
\Delta(A_0,A_1,F,\P)=\sum_{l=0,1} \frac{\P[\partial_1\ell(F(x),y)\1_{A_l}(x)]^2}{\P(A_l)}
\]
is continuously differentiable in $F$, with gradient
\[
\nabla_F \Delta(A_0,A_1,F,\P)
=2\sum_{l=0,1}\frac{\P[\partial_1\ell(F(x),y)\1_{A_l}(x)]\;\P[\partial_1^2\ell(F(x),y)\1_{A_l}(x)\Phi(x)]}{\P(A_l)}.
\]
Thus
\[
\|\nabla_F \Delta(A_0,A_1,F,\P)\|_{\mathcal{H}}
\le 4L\kappa\,\P\!\left[|\partial_1\ell(F(x),y)|\right]\leq C_M,
\]
where the last inequality holds for some constant $C_M>0$ when $\|F\|_{\mathcal{H}}\leq M$.

\textit{Step 2. Proof of (i).} The term $\Psi_1$ is directly controlled by Step~1. For fixed $(\bar\xi,v)$,
\[
F\mapsto\Psi_1(\xi^1,v,F,\P)=\frac{\P[\partial_1\ell(F(x),y)\1_{A_v^{\xi^1}}(x)]}{\P(A_v^{\xi^1})}
\]
is Fréchet differentiable on $\mathcal{H}$, with gradient
\[
\nabla_F \Psi_1(\xi^1,v,F,\P)
= \frac{\P[\partial_1^2\ell(F(x),y)\1_{A_v^{\xi^1}}(x)\Phi(x)]}{\P(A_v^{\xi^1})},
\]
and norm bounded by
\[
\|\nabla_F \Psi_1(\xi^1,v,F,\P)\|_{\mathcal{H}} \le L\kappa.
\]

\smallskip\noindent
\textit{Step 3. Bounding the gradient norm of $\Psi_2$.} The term $\Psi_2(\bar\xi,F,\P)$ consists of a product over $v\in\mathscr{T}_{d-1}$ of $\beta$-softmax functions
\[
\mathrm{softmax}_\beta(u_1,\ldots,u_K)=\frac{e^{\beta u_1}}{\sum_{k=1}^K e^{\beta u_k}} \, \in(0,1),
\]
whose gradient has norm at most $\beta/2$. Each argument $u_k$ is a score of the form
\[
\Delta(A_{v0}^{\xi^1},A_{v1}^{\xi^1},F,\P),
\]
which is continuously differentiable by Step~1.
By the product rule, $\nabla_F \Psi_2(\bar\xi,F,\P)$ is a sum (indexed by $v\in\mathscr{T}_{d-1}$) of products involving uniformly bounded $\beta$-softmax factors and exactly one gradient term, whose norm is bounded by $C_M$ by Step~1. Hence,
\[
\|\nabla_F \Psi_2(\bar\xi,F,\P)\|_{\mathcal{H}} \le C_M,\qquad  \|F\|_{\mathcal{H}}\leq M,
\]
for a suitable constant $C_M>0$.

\textit{Step 4. Conclusion.} Finally,
\[
\Psi(\bar\xi,v,F,\P)=\Psi_1(\xi^1,v,F,\P)\,\Psi_2(\bar\xi,F,\P)
\]
is continuously differentiable with gradient
\[
\nabla_F \Psi(\bar\xi,v,F,\P)
= \Psi_2(\bar\xi,F,\P)\,\nabla_F\Psi_1(\xi^1,v,F,\P)
  + \Psi_1(\xi^1,v,F,\P)\,\nabla_F\Psi_2(\bar\xi,F,\P).
\]
Since the functions and their gradients in the right hand side are uniformly bounded for $\|F\|_{\mathcal{H}}\leq M$, the same holds for $\nabla_F \Psi$, proving the claim.
\end{proof}

\subsection{Dependency of $\Psi$ with respect to $\P$}
According to the proof of Theorem~\ref{thm:clt-igb}$(i)$, we see the empirical processes $\P_n$, $n\geq 1$, and $\P$ as elements of $\ell^\infty(\mathcal{G})$, with $\mathcal{G}$ the class of function defined in Equation~\eqref{eq:def-G}. Using Lemma~\ref{lem:Donsker} stating that $\mathcal{G}$ is Donsker together with a (non-measurable) Skorokhod representation (Theorem~1.10.4 in \cite{vW96}), we assume without loss of generality the outer almost-uniform convergence
\[
\mathbb{G}_n=\sqrt{n}(\P_n-\P)\stackrel{au}\longrightarrow \mathbb{G}_{\P} \quad \text{ in $\ell^\infty(\mathcal{G})$ \ as $n\to\infty$}.
\]
This representation makes it possible to consider the convergence as deterministic on an event of outer probability $1$ and to compute asymptotic expansions, similarly as in the functional delta-method.  

The next Lemma provides an asymptotic expansion, as $n\to\infty$, for  $\Psi(\bar\xi,v,F,\P_n)$ defined in Equation~\eqref{eq:def-Psi}. The definition involves ratios of the form
\begin{equation}\label{eq:ratios}
\frac{\P_n[\partial_1\ell(F(x),y)\1_{A}(x)]}{\P_n(A)},\quad 
\frac{\P_n[\partial_1\ell(F(x),y)\1_{A}(x)]^2}{\P_n(A)},
\end{equation}
for a collection of hyperrectangles (or cells) $A$ that we need to describe more precisely. It depends on the splitting schemes $\bar\xi=(\xi^1,\ldots,\xi^K)$ and can be indexed by $i=(w,k,l)$ in the index set
\begin{equation}\label{eq:index-set}
\mathcal{I}=\mathcal{T}_{d-1}\times [\![1,K]\!]\times \{0,1\}.
\end{equation}
The index $w$ refers to an internal node of the binary tree in the recursive construction. At this node, the cell $A_w^{\xi^1}$ is split using $K$ candidate splits $(j_w^k,u_w^k)$, $k=1,\ldots,K$, taken from the splitting schemes; each candidate split generates a left and a right daughter cell, indexed by $l=0$ and $l=1$, respectively. We denote by $(A_i^{\bar\xi})_{i\in\mathcal{I}}$ the resulting collection of cells appearing in $\Psi(\bar\xi,v,F,\P_n)$.

\begin{lemma}\label{lem:tech-P}
Assume that $\P_n$, $n\ge 1$, and $\P$ are deterministic elements of $\ell^\infty(\mathcal{G})$ such that the convergence
\begin{equation}\label{eq:lem-tech-P}
G_n[\varphi]=\sqrt{n}\,(\P_n[\varphi]-\P[\varphi]) \longrightarrow G[\varphi],
\qquad \varphi\in\mathcal{G},
\end{equation}
holds in $\ell^\infty(\mathcal{G})$ as $n\to\infty$. Then there exist functions
\[
\Gamma_i=\Gamma_i(\bar\xi,v,F,\P)
\quad\text{and}\quad
\Gamma_i'=\Gamma_i'(\bar\xi,v,F,\P),
\qquad i\in\mathcal{I},
\]
uniformly bounded for $\|F\|_{\mathcal{H}}\leq M$, such that
\begin{align*}
\sqrt{n}\big(\Psi(\bar\xi,v,F,\P_n)-\Psi(\bar\xi,v,F,\P)\big)
=
\sum_{i\in\mathcal{I}}
\left(
\frac{\Gamma_i\,G[\varphi_{A_i^{\bar\xi}}]+\Gamma_i'\,G[\varphi_{A_i^{\bar\xi},F}]}{\P(A_i^{\bar\xi})}
+\frac{o(1)}{\P(A_i^{\bar\xi})}
\right),
\end{align*}
where we use the notation
\[
\varphi_A(x,y)=\1_A(x)\quad \text{and}\quad \varphi_{A,F}(x,y)=\partial_1\ell(F(x),y)\1_A(x).
\]
Furthermore, given a sequence $\varepsilon_n\to 0$ such that $\sqrt{n}\varepsilon_n\to+\infty$, the $o(1)$ terms are uniform for $\|F\|_{\mathcal{H}}\leq M$ and $\bar\xi$ such that $\min_{i\in\mathcal{I}} \P(A_i^{\bar\xi})\ge \varepsilon_n$. 
\end{lemma}

\begin{proof}[Proof of Lemma~\ref{lem:tech-P}] For clarity we divide the proof into four steps.

\medskip\noindent
\textbf{Step 1.} We first consider asymptotic expansions for the ratios \eqref{eq:ratios} involved in the definition of $\Psi(\bar\xi,v,F,\P_n)$. 

We define
\[
\mathcal{R}_{\varepsilon_n}=\{A\in\mathcal{R}:\P(A)\geq \varepsilon_n\}
\]
and note that $\sqrt{n}\P(A)\geq \sqrt{n}\varepsilon_n$ tends to infinity uniformly for $A\in\mathcal{R}_{\varepsilon_n}$.

Equation~\eqref{eq:lem-tech-P} ensures that
\[
\P_n[\varphi]=\P[\varphi]+\frac{G[\varphi]+o(1)}{\sqrt{n}},\qquad \text{as $n\to\infty$},
\]
with the $o(1)$ term uniform in $\varphi\in\mathcal{G}$.
Considering the functions 
\[
\varphi_{A}(x,y)=\1_A(x)\quad \text{and}\quad  \varphi_{A,F}(x,y)=\partial_1\ell(F(x),y)\1_{A}(x),
\]
we deduce the asymptotic expansion 
\[
\frac{\P_n[\varphi_{A,F}]}{\P_n(A)}
=\frac{\P[\varphi_{A,F}]}{\P(A)}+\frac{\gamma_1}{\sqrt{n}\P(A)}+\frac{o(1)}{\sqrt{n}\P(A)},
\]
where 
\begin{equation}\label{eq:def-gamma1}
\gamma_1= G[\varphi_{A,F}] -\frac{\P[\varphi_{A,F}]}{\P(A)}G[\varphi_{A}],
\end{equation}
is uniformly bounded and the $o(1)$ is uniform in $\|F\|\leq M$ and $A\in\mathcal{R}_{\varepsilon_n}$.

Similarly,
\[
\frac{\P_n[\varphi_{A,F}]^2}{\P_n(A)}
=\frac{\P[\varphi_{A,F}]^2}{\P(A)}+\frac{\gamma_2}{\sqrt{n}\P(A)}+\frac{o(1)}{\sqrt{n}\P(A)},
\]
where 
\begin{equation}\label{eq:def-gamma2}
\gamma_2= 2\P[\varphi_{A,F}]\,G[\varphi_{A,F}] -\frac{\P[\varphi_{A,F}]^2}{\P(A)}G[\varphi_{A}],
\end{equation}
is uniformly bounded and the $o(1)$ is uniform in $\|F\|\leq M$ and $A\in\mathcal{R}_{\varepsilon_n}$.

\medskip\noindent
\textbf{Step 2.} We provide an asymptotic expansion for $\Psi_1(\xi^1,v,F,\P_n)$. Using Equation~\eqref{eq:def-Psi1} and the expansion from Step~1, we obtain
\begin{equation}\label{eq:expansion-Psi1}
\Psi_1(\xi^1,v,F,\P_n)=\Psi_1(\xi^1,v,F,\P)
+\frac{\gamma_1(A_v^{\xi^1},F,\P,G)}{\sqrt{n}\P(A_v^{\xi^1})}
+\frac{o(1)}{\sqrt{n}\P(A_v^{\xi^1})}.
\end{equation}
Note that the cell $A_v^{\xi^1}$ is equal to $A^{\bar\xi}_i$ for some well chosen index  $i\in \in\mathcal{I}$, see the definition~\eqref{eq:index-set}.

\medskip\noindent
\textbf{Step 3.} We provide an asymptotic expansion for $\Psi_2(\bar\xi,F,\P_n)$. From Equation~\eqref{eq:def-Psi2}, $\Psi_2(\bar\xi,F,\P_n)$ can be written as
\[
\Psi_2(\bar\xi,F,\P_n)=\psi_2\big((u_{n,i})_{i\in\mathcal{I}}\big),
\qquad
u_{n,i}=\frac{\P_n[\varphi_{A_i^{\bar\xi},F}]^2}{\P_n(A_i^{\bar\xi})}.
\]
By Step~1,
\[
u_{n,i}=u_i+\frac{\gamma_2(A_i^{\bar\xi},F,\P,G)}{\sqrt{n}\P(A_i^{\bar\xi})}
+\frac{o(1)}{\sqrt{n}\P(A_i^{\bar\xi})},
\]
with an obvious notation for $u_i$. Since $\psi_2$ is continuously differentiable with bounded partial derivatives -- it is a product of softmax functions --, we obtain
\begin{align}
\Psi_2(\bar\xi, F, \P_n)
&=\psi_2((u_{i})_{i\in\mathcal{I}})+\sum_{i\in\mathcal{I}}\partial_i \psi_2((u_{i})_{i\in\mathcal{I}})\Big(\frac{\gamma_2(A^{\bar\xi}_i,F,\P,G)}{\sqrt{n}\P(A^{\bar\xi}_i)}+\frac{o(1)}{\sqrt{n}\P(A^{\bar\xi}_i)}\Big)\nonumber\\
&=\Psi_2(\bar\xi, F, \P)+\sum_{i\in\mathcal{I}}\Big(\frac{\partial_i \psi_2((u_{i})_{i\in\mathcal{I}})\gamma_2(A^{\bar\xi}_i,F,\P,G)}{\sqrt{n}\P(A^{\bar\xi}_i)}+\frac{o(1)}{\sqrt{n}\P(A^{\bar\xi}_i)}\Big).\label{eq:expansion-Psi2}
\end{align}
The $o(1)$ terms are uniform for $\|F\|_{\mathcal{H}}\leq M$ and $\bar\xi$ such that $A_i^{\bar\xi}\in\mathcal{R}_{\varepsilon_n}$ for all $i\in\mathcal{I}$.

\medskip\noindent
\textbf{Step 4.} We conclude the proof of the Lemma. By Equation~\eqref{eq:def-Psi}, $\Psi(\bar\xi,v,F,\P_n)$ is the product of the terms analyzed in Steps~2 and~3. Multiplying Equations~\eqref{eq:expansion-Psi1} and~\eqref{eq:expansion-Psi2}, plugging the expression of $\gamma_1$ and $\gamma_2$ from Equations~\eqref{eq:def-gamma1} and~\eqref{eq:def-gamma2}, and rearranging terms, we obtain
\[
\Psi(\bar\xi,v,F,\P_n)
=
\Psi(\bar\xi,v,F,\P)
+
\sum_{i\in\mathcal{I}}
\left(
\frac{\Gamma_i\,G[\varphi_{A_i^{\bar\xi}}]+\Gamma_i'\,G[\varphi_{A_i^{\bar\xi},F}]}{\sqrt{n}\P(A_i^{\bar\xi})}
+\frac{o(1)}{\sqrt{n}\P(A_i^{\bar\xi})}
\right),
\]
for some uniformly bounded functions $\Gamma_i=\Gamma_i(\bar\xi,v,F,\P)$ and $\Gamma_i'=\Gamma_i'(\bar\xi,v,F,\P)$. The $o(1)$ terms are uniform provided $\|F\|_{\mathcal{H}}\leq M$ and $\min_{i\in\mathcal{I}}\P(A_i^{\bar\xi})\ge \varepsilon_n$.
\end{proof}

\subsection{Uniform integrability for $\Psi$}
\begin{lemma}\label{lem:uniform-integrability}
For all fixed $\bar\xi$, $v$, and $M>0$, there exists a constant $C_M>0$ that depends only on $M$ and such that
\[
\mathbb{E}\!\left[
\sup_{n\geq 1}\sup_{\|F\|_{\mathcal{H}}\leq M}
|\Psi(\bar\xi,v,F,\P_n)|^2
\right] \leq C_M.
\]
\end{lemma}

\begin{proof}[Proof of Lemma~\ref{lem:uniform-integrability}] 
Recall from Lemma~\ref{lem:RKHS-rf} that 
\[
\Psi(\bar\xi,v,F,\P_n)=\Psi_1(\xi^1,v,F,\P_n)\Psi_2(\bar\xi,F,\P_n)
\]
with $\Psi_1$ and $\Psi_2$ defined respectively in Equations~\eqref{eq:def-Psi1} and~\eqref{eq:def-Psi2}. The inequality
\[
|\partial_1\ell(y,F(x))|
\le|\partial_1\ell(y,0)|+L|F(x)|
\le|\partial_1\ell(y,0)|+L\kappa\|F\|_{\mathcal{H}}
\]
implies
\[
\sup_{\|F\|_{\mathcal{H}}\le M} |\Psi_1(\xi^1,v,F,\P_n)|
\le \sup_{\|F\|_{\mathcal{H}}\le M}
\frac{\P_n\big[|\partial_1\ell(y,0)|\,\1_{A_v^{\xi^1}}(x)\big]}{\P_n(A_v^{\xi^1})}
+L\kappa M.
\]
On the other hand, $\Psi_2(\bar\xi,F,\P_n)\le K^{2^d-1}$. Thus,
\[
\sup_{\|F\|_{\mathcal{H}}\le M}|\Psi(\bar\xi,v,F,\P_n)|
\le K^{2^d-1}\Big(
\frac{\P_n\big[|\partial_1\ell(y,0)|\1_{A_v^{\xi^1}}(x)\big]}{\P_n(A_v^{\xi^1})}
+L\kappa M\Big).
\]
By a backward martingale argument combined with Doob's maximal inequality (argument  taken from the proof of Theorem~2.10 in \citet{DD24}), we obtain
\[
\mathbb{E}\Bigg[\Bigg(
\sup_{n\ge1}
\frac{\P_n\big[|\partial_1\ell(y,0)|\1_{A_v^{\xi^1}}(x)\big]}{\P_n(A_v^{\xi^1})}
\Bigg)^2\Bigg]
\le 4\,\frac{\P\big[|\partial_1\ell(y,0)|^2\1_{A_v^{\xi^1}}(x)\big]}{\P(A_v^{\xi^1})},
\]
where the right hand side is uniformly bounded by Equation~\eqref{eq:cond-expectation-bounded}. This concludes the proof.  
\end{proof}

\paragraph{Acknowledgments}
  The second author acknowledges funding through the ANR project GARP
  (ANR-24-CE40-7154). The LmB receives support from the EIPHI Graduate
  School (contract ANR-17-EURE-0002)

\phantomsection

\end{document}